\documentclass[11pt,preprint]{imsart}

\RequirePackage[OT1]{fontenc}
\usepackage{a4wide}
\usepackage{amsmath,amsthm,amssymb}
\usepackage[isolatin]{inputenc}
\usepackage{amsfonts,dsfont}
\usepackage{pstricks}
\usepackage{comment}
\usepackage[numbers]{natbib}
\usepackage{ulem	}

\usepackage{colordvi}
\usepackage{graphicx}

\usepackage{xr-hyper}

\definecolor{vert}{rgb}{0.09,0.7,0.17}
\definecolor{violet}{rgb}{0.69,0.13,0.69}
\definecolor{greens}{rgb}{0.09,0.7,0.17}

\usepackage[pdftex,                %
bookmarks         = true,
bookmarksnumbered = true,
pdfpagemode       = None,
pdfstartview      = FitH,
pdfpagelayout     = SinglePage,
colorlinks        = true,
citecolor		= vert,
urlcolor          = violet,
pdfborder         = {0 0 0}
]{hyperref}

\newtheorem{thm}{Theorem}[section]
\newtheorem{prp}[thm]{Proposition}
\newtheorem{lem}[thm]{Lemma}
\newtheorem{cor}[thm]{Corollary}

\newtheorem{remark}[thm]{Remark}

\newcommand{\R}{\mathbb{R}} 

\renewcommand{\P}{\mathbb{P}}
\newcommand{\E}{\mathbb{E}} 
\DeclareMathOperator{\var}{Var} 

\renewcommand{\l}{\ell}
\newcommand{\FDR}{\mbox{FDR}}
\newcommand{\BH}{\mbox{BH}}

\newcommand{\wt}[1]{{\widetilde{#1}}}
\newcommand{\wh}[1]{{\widehat{#1}}}
\newcommand{\ol}[1]{\overline{#1}}

\newcommand{\ind}{{\mathds{1}}}

\newcommand{\arccosh}{\mathrm{arccosh}}

\newcommand{\const}{a}
\newcommand{\thetachapup}{\wh{\theta}_{\mathrm{up}}}
\newcommand{\thetachaplow}[1]{\wh{\theta}_{\mathrm{low},#1}}
\newcommand{\thetachapmin}{\wh{\theta}_{\mathrm{min}}}
\newcommand{\thetachapmed}{\wh{\theta}_{\mathrm{med}}}

\newcommand{\eps}{\varepsilon}

\newcommand{\cB}{{\mtc{B}}}
\newcommand{\cH}{{\mtc{H}}}

\newcommand{\cN}{{\mtc{N}}}

\usepackage{graphics}
\newcommand{\FDP}{\mbox{FDP}} 
\newcommand{\TDP}{\mbox{TDP}}

\newcommand{\ad}{\mathrm{ad}}

\def\cA{\mathcal{A}}
\def\cB{\mathcal{B}}
\def\cC{\mathcal{C}}

\def\cH{\mathcal{H}}

\def\cM{\mathcal{M}}
\def\cN{\mathcal{N}}

\def\cP{\mathcal{P}}
\def\cQ{\mathcal{Q}}
\def\cR{\mathcal{R}}

\def\c\Gamma(Y){\mathcal{\Gamma(Y)}}
\def\med{\mathrm{med}}

\def\b\Gamma(Y){\mathbf{\Gamma(Y)}}

\def\beq{\begin{equation}}
\def\eeq{\end{equation}}
\def\beqn{\begin{eqnarray*}}
\def\eeqn{\end{eqnarray*}}

\begin{document}

\begin{frontmatter}

\title{Estimating minimum effect with outlier selection}
\runtitle{Estimating minimum effect with outlier selection}
\begin{aug}
\author{Alexandra Carpentier, Sylvain Delattre, Etienne Roquain and Nicolas Verzelen
}
\runauthor{Carpentier et al.}
\end{aug}

\begin{abstract}
We introduce one-sided versions of Huber's contamination model, in which corrupted samples tend to take larger values than uncorrupted ones.
Two intertwined problems are addressed: 
estimation of the mean of uncorrupted samples (minimum effect) and selection of corrupted samples (outliers). Regarding the minimum effect estimation, we derive the minimax risks and introduce adaptive estimators to the unknown number of contaminations. Interestingly, the optimal convergence rate highly differs from that in classical Huber's contamination model. Also, our analysis uncovers the effect of particular structural assumptions on the distribution of the contaminated samples. 
As for the problem of selecting the outliers, we formulate the problem in a multiple testing framework for which the location/scaling of the null hypotheses are unknown. We rigorously prove how estimating the null hypothesis is possible while maintaining a theoretical guarantee on the amount of the falsely selected outliers, both through false discovery rate (FDR) or post hoc bounds. 
As a by-product, we address a long-standing open issue on FDR control under equi-correlation, which reinforces the interest of removing dependency when making multiple testing.

\end{abstract}



\end{frontmatter}

\section{Introduction}
\label{sec:intro}

We are interested in a statistical framework where some data have been corrupted. Depending on how one defines and considers the corruption, such problems have been addressed by different fields in statistics such as robust estimation or sparse modeling. In the former, Huber's contamination model~\cite{huber1964robust, huber2011robust} is the prototypical setting for handling this problem. It assumes that among $n$ observations $Y_1,\ldots, Y_n$, most of them follow some normal distribution $\cN(\theta,\sigma^2)$ whereas the corrupted ones are arbitrarily distributed. In sparse modeling, one typically assumes that the data $Y_1, \ldots, Y_n$ are normally distributed with mean $\gamma_i$ where $\gamma_i=\theta$ for uncorrupted samples and arbitrary $\gamma_i \neq \theta$ for corrupted samples (see~\cite{CJ2010} for a related model).

However, in some practical problems, corrupted samples do not take arbitrary values and satisfy a structural assumption. Consider for instance the following situation where $Y_i$'s are measurements of a pollutant, coming from $n$ sensors spread out at $n$ locations of a city. The background value for this  pollutant in the city is $\theta$, but, due to local pollution effects, some sensors may record larger values at some locations. Health authorities are then interested in evaluating the degree of background pollution and in finding where the most affected regions in the city are.

\medskip 

In this work, we introduce one-sided contamination models taking into account the structural assumption that 
corrupted samples tend to take larger values than uncorrupted ones. Then, we consider the twin problems of estimating the distribution of the uncorrupted samples and identifying the corrupted samples.

\subsection{Models and objectives}

\subsubsection{One-sided Contamination Model (OSC)} 


We first introduce a one-sided counterpart of Huber contamination model for which some samples $Y_i$'s follow a $\mathcal{N}(\theta,\sigma^2)$ distribution, whereas the remaining samples are positively contaminated, that is, have a  distribution that stochastically dominates $\mathcal{N}(\theta,\sigma^2)$, but is otherwise arbitrary.

More formally, we assume that 
\begin{equation}\label{classicmodel_2}
 Y_i = \theta+ \sigma \eps_i, \:\:\: 1\leq i \leq n\ ,
 \end{equation}
where $\sigma>0$ is some standard deviation parameter (either equal to $1$ or unknown), $\theta\in \R$ is a fixed {\it minimum effect} and  the $\eps_i$ are independent noise random variables. 
Denoting $\pi_i$ the unknown distribution of the noise, we assume that, for some $k$, the distribution $\pi=\otimes_{i=1}^n \pi_i$ of $\eps$ belongs to the set
\begin{equation}\label{equbarcmk} 
  \overline{\cM}_k=\left\{ \pi=\otimes_{i=1}^n \pi_i \::\: \mbox{ $\pi_i\succeq \mathcal{N}(0,1)$},\:  \sum_{i=1}^n \mathds{1}_{\{\pi_i \succ \mathcal{N}(0,1)\}}\leq k\right\}\ ,
 \end{equation}
 where $\succeq$ (resp. $\succ$) denotes the stochastic domination (resp. strict stochastic domination).  In $\overline{\cM}_k$, at most $k$ distributions $\pi_i$'s are allowed to strictly dominate the Gaussian measure. The model~\eqref{classicmodel_2} satisfies the heuristic explanation described above. 
 If $\pi\in  \overline{\cM}_k$, then at least $n-k$ samples are non-contaminated and are distributed as $\mathcal{N}(\theta,\sigma)$ whereas the remaining contaminated samples stochastically dominate this distribution. 
 
 In this model, henceforth referred as the One-Sided Contamination (OSC) model, the parameter $\theta$ corresponds to the expectation of the non-contaminated samples. If $k\leq  n-1$, it also satisfies
\begin{equation}\label{thetaidentif}
\theta=\min_{1\leq i\leq n} \E (Y_i)\ ,
\end{equation}
 and interprets therefore as a minimum  theoretical effect. In particular, this model is identifiable for $k\in [n/2, n-1]$, whereas it is not in the classical Huber's model.

Throughout the paper,  the probability (resp. expectation) in model \eqref{classicmodel_2} is denoted by $\P_{\theta,\pi,\sigma}$ (resp. $\E_{\theta,\pi,\sigma}$). The parameter $\sigma$ is dropped in the notation whenever $\sigma=1$.

\subsubsection{One-sided Gaussian Contamination Model (gOSC)}

In analogy with the sparse Gaussian vector model, we also consider a specific case of OSC model where the contaminated samples are still assumed to be normally distributed, that is, the  $\pi_i$'s are Gaussian distribution with unit variance and positive mean $\mu_i/\sigma$ where $\mu \in \R_+^n$ is a 
{\it contamination effect}. In that case, the model can be rewritten as
\begin{equation}
Y_i=  \theta + \mu_i+ \sigma\xi_i,\:\:\:1\leq i \leq n\ ,\label{model}
\end{equation}
where $\xi_i$'s  are i.i.d.~$\mathcal{N}(0,1)$ distributed and $\mu\in \mathbb{R}_+^n$ is unknown. Defining the mean vector 
\begin{equation}
\gamma = \theta + \mu\ ,\label{decompgamma}
\end{equation}
we deduce that $Y$ follows a normal distribution with unknown mean $\gamma$ and variance $\sigma^2 I_n$
whereas $\theta$ corresponds to $\min_i \gamma_i$, that is, the minimum component of the mean vector.

%

To formalize the connection with the OSC model, we let $\eps_i = \mu_i/\sigma + \xi_i$  and $\pi_i = \mathcal N(\mu_i/\sigma, 1)$ for all $i$. Then,  \eqref{model} is a particular case of \eqref{classicmodel_2} since  $\mathcal N(\mu_i/\sigma, 1) \succeq \mathcal{N}(0,1)$.
In analogy with the OSC model where we define a collection $\overline{\mathcal{M}}_k$ prescribing the number of contaminated samples to be less or equal to $k$, we introduce
 \begin{equation}\label{equcmk} 
 \cM_k=\left\{\mu \in \mathbb{R}_+^n \::\: \sum_{i=1}^n \mathds{1}_{\{\mu_i\neq 0\}} \leq k\right\}\ .
 \end{equation}

In what follows, we refer to the model~\eqref{model} as One-Sided Gaussian Contamination (gOSC) model. The probability (resp. expectation) in that model \eqref{model} is denoted by $\P_{\theta,\mu,\sigma}$ (resp. $\E_{\theta,\mu,\sigma}$). Whenever we assume that the variance parameter $\sigma$ is known and is equal to one,  the subscript $\sigma$ is dropped in the above notation.

\subsubsection{Objectives}

We are interested in the two following intertwined problems:
\begin{itemize}
\item[-] {\it Objective 1: Optimal estimation of the minimum effect.} We aim at  establishing the minimax estimation rates of $\theta$ both  in OSC~\eqref{classicmodel_2} and in gOSC~\eqref{model} models. In particular, we explore the role of the one-sided assumption for the computation on such estimation rates. As explained below, this problem is at the crossroads of several lines of research such as robust estimation and non-smooth linear functional estimation.

\item[-] {\it Objective 2: controlled selection of the outliers.} Here, we are interested in finding the contaminated samples. In the Gaussian case (gOSC), this is equivalent to selecting the positive entries of $\mu$ in \eqref{decompgamma}. Adopting a multiple testing framework, we aim at designing a selection procedure with suitable false discovery rate (FDR) control \cite{BH1995} and providing a uniformly valid post hoc bound \cite{GW2006,GS2011}. The difficulty stems from the fact the minimum effect $\theta$ is unknown. In contrast to objective 1 where the contaminated samples were considered as nuisance quantities, in this second objective the contaminated samples are now interpreted as the signal whereas $\theta$ is a nuisance parameter. 

\end{itemize}
Furthermore, Objective 2 is intrinsically connected to the problem of removing the correlation when making (one-sided) multiple testing from Gaussian equi-correlated test statistics: when the equi-correlation is carried by  the latent factor $\theta$, we can remove this correlation by subtracting an estimator of $\theta$ to the test statistics. Although this simple strategy is quite common (see, e.g., \cite{FKC2009} and references therein), assessing the theoretical performances of such a procedure is a longstanding question in the multiple testing literature.  In this work, we establish a positive answer to this question, by showing that it is possible to (asymptotically) control the FDR while having (at least) the same power as if the test statistics had been independent. 

In the remainder of the introduction, we first describe our contribution for minimum effect estimation and then turn to outlier selection.

\subsection{Optimal estimation of the minimum effect}\label{sec:intro:estimation}

Given the sparsity $k\in\{1,\dots,n-1\}$ and $\sigma^2=1$, we define the $L_1$ minimax estimation  risk of $\theta$   for both  gOSC~\eqref{model} and OSC~\eqref{classicmodel_2} models: 
\begin{align}
\cR[k,n]&= \inf_{\wh{\theta}}\sup_{(\theta,\mu)\in \mathbb{R}\times \cM_k}\E_{\theta,\mu}\big[|\wh{\theta}-\theta|\big]; \, \quad 
\overline{\cR}[k,n]= \inf_{\wh{\theta}}\sup_{\theta\in \mathbb{R}, \pi \in \overline{\cM}_k }\E_{\theta,\pi}[|\wh{\theta}-\theta|]\label{eq:definition_minimax_robust_one_sided}\ . 
\end{align}

First, we characterize these minimax risks by deriving matching (up to numerical constants) lower and upper bounds, this uniformly over all numbers $k$ of contaminated data, see Sections~\ref{sec:fixedcont} and~\ref{sec:robust}. The results are summarized in Table~\ref{Tableminimax} below. 
It is mostly interesting to compare these orders of magnitude with those derived for the Huber contamination model with $k$ contamination. From e.g.~\cite[Sec.2]{chen2018robust}, we derive\footnote{Actually, the results in~\cite{chen2018robust} are proved for a model where the number of contaminated sample follows a Binomial distribution with parameters $(k,k/n)$, but the proofs straightforwardly extend to our setting}, that for $k< n/2$, the minimax risk is of order $\min(n^{-1/2},\tfrac{k}{n})$. For $k\leq \sqrt{n}$, the rate is parametric in all three models. For $k\in (\sqrt{n}, n/2)$, one-sided contamination lead to some $\sqrt{\log(k^2/n)}$ gain over the Huber's model, whereas assuming that the contaminations are Gaussian lead to an additional logarithmic gain. For $k\in [n/2, n-1]$, recall that Huber's model is not identifiable whereas the one-sided contamination model is, and we identify various minimax rates.  For a fixed proportion ($k/n$) of contaminated samples, the optimal rate still converges to $0$ at a  polylogarithmic rate. For slowly decaying (with $n$) proportion $\frac{n-k}{n}$ of non-contaminated samples, the estimation rate still goes to $0$. 

\begin{table}[h!]
\begin{center}
\begin{tabular}{|c||c||c|c|c|}
\hline
&&&&\\
&General bound  
& $1\leq k\leq 2\sqrt{n}$ & $2\sqrt{n}\leq k\leq n/2$ & $n/2\leq k \leq n-1$\\
\hline
&&&&\\
 $\overline{\cR}[k,n]$ & $\frac{\log\left(\frac{n}{n-k}\right)}{\log^{1/2}(1+ \frac{k^2}{n})}$ & $n^{-1/2}$ & $\frac{k/n}{\log^{1/2}(k^2/n)}$ & $\frac{\log\left(\frac{n}{n-k}\right)}{\log^{1/2}n}$\\
\hline
&&&&\\
 $\cR[k,n]$ & $\frac{\log^{2}\big(1+ \sqrt{\frac{k}{n-k}}\big)}{\log^{3/2}\big(1+ (\frac{k}{\sqrt{n}})^{2/3}\big)}$ & $n^{-1/2}$ & $\frac{k/n}{\log^{3/2}(k^2/n)}$ & $\frac{\log^2\left(\frac{n}{n-k}\right)}{\log^{3/2} n}$\\
\hline
\end{tabular}\smallskip\\
\end{center}
\caption{Minimax estimation risks of $\theta$ (up to numerical constants).  \label{Tableminimax}}
\end{table}

For both models (OSC and gOSC), we also devise estimation procedures that are adaptive to the unknown number $k$ of contaminated samples.  Finally, we consider the case where the   noise level $\sigma$ in \eqref{model} unknown, see Section~\ref{sec:randcontsigma}. We prove that, in OSC model, adaptation to unknown $\sigma$ is possible and characterize the optimal estimation risk for $\sigma$.

\medskip

\paragraph{OSC: Technical aspects and connection to robust estimation.}
As explained earlier, OSC~\eqref{classicmodel_2} model is a one-sided counterpart of Huber's contamination model 
~\cite{huber1964robust, huber2011robust} - see also~\cite{neyman1948consistent} for the historical reference on the concept of contamination and~\cite{lancaster2000incidental,jurevckova2012methodology} for more recent reviews.
From a technical perspective, minimax bounds for OSC proceed from the same general ideas as for Huber's contamination model with a twist. For the latter, the empirical median turns out to be optimal~\cite{huber2011robust}. In OSC model, there is a benefit of using other empirical quantiles. Since the contaminations are one-sided, the left tail is indeed less perturbed than the right tail. Correcting for the bias and choosing suitably a quantile, we prove that the resulting estimator achieves (up to constants) the optimal rate $\overline{\cR}[k,n]$. Adaptation to unknown $k$ is performed via Lepski's method whereas adaptation to unknown $\sigma$ is based on a difference of empirical quantiles.

\paragraph{gOSC: Technical aspects and connection to non-smooth functional estimation.}
Pinpointing the minimax risk in the Gaussian contamination model (gOSC) is much more technical. Indeed,
 standard estimators, as those based on quantiles for instance,  are not optimal in that setting. The key idea of our upper bound is to invert  a collection of local tests of the form ``$\theta\geq u$'' vs ``$\theta < u$'' for $u\in \mathbb{R}$, by following an approach from \cite{carpentier2017adaptive} developed for sparsity testing. Recall that $\gamma_i$ in~\eqref{decompgamma} stands for the expectation of $Y_i$. If ``$\theta\geq u$'', then $\sum_{i}\mathds{1}_{\gamma_i< u}=0$ whereas under the alternative, one has $\sum_{i}\mathds{1}_{\gamma_i< u}\geq n-k$. Thus, this boils down to estimating the non smooth-functional, $\sum_{i}\mathds{1}_{\gamma_i< u}$. 

Since the seminal work~\cite{ibragimov1985nonparametric,donoho1990minimax} (for respectively the linear and quadratic functional), there is an extensive literature on estimating smooth functionals of the mean of a Gaussian vector. Under a sparsity assumption, the problem has been investigated in~\cite{MR2253108, MR2879672, collier2015minimax, collier2016optimal}, and has some deep connections with problem of signal detection~\cite{ingster2012nonparametric, baraud02}.

However, estimation of non-smooth functionals (such as $\sum_{i=1}^n |\gamma_i|^q$ for $q\in (0,1])$) is significantly more involved even without sparsity assumptions. For related papers, see e.g.~\cite{JC2007, CJ2010, cailow2011, MR2420411, lepski1999estimation, han2016minimax, wu2015chebyshev, jiao2016minimax,carpentier2017adaptive, Juditsky_convexity,collier2018estimation}.  For that class  of problem, one powerful approach, coined as polynomial approximation~\cite{lepski1999estimation, han2016minimax}, amounts to build a best polynomial approximation of the non-smooth function and plug them with unbiased estimators of the moment $\sum_{i=1}^n \gamma_i^s$ for some integers $s=1,\ldots, s_{\max}$. 
Unfortunately, we cannot rely on this strategy for estimating $\sum_{i}\mathds{1}_{\gamma_i< u}$, mainly  because the contaminated $\gamma_i$'s may be arbitrarily large. In a related setting, where the contaminated means $\gamma_i\neq \theta$ are distributed according to some smooth prior distributions supported on $\R$, \cite{CJ2010} have pinpointed the optimal rate by relying on empirical Fourier transform (see also~\cite{jin2004}). However, this approach falls down in our framework because the contaminated $\gamma_i$'s are arbitrary.
In this work, we introduce a new strategy that combines polynomial approximation methods with the empirical Laplace transform. 

As for the minimax lower bound, we rely on moment matching techniques following the approach of~\cite{lepski1999estimation} and recently applied to other non-smooth functional models \cite{cailow2011,han2016minimax,wu2015chebyshev, carpentier2017adaptive}.

 \subsection{Controlled selection of the outliers}\label{descriptionMT}

 This section presents the state of the art and our contributions for the second objective, that is, controlled selection of the outliers.   Our approach relies on multiple testing paradigm and builds upon some of our results on the estimation of $\theta$.

 \subsubsection{Multiple testing formulation}\label{settingMT}

 Our second objective is to identify the active set of outliers in the general model~\eqref{classicmodel_2}.
Again, we emphasize that what we call outliers becomes in this part the quantities of interest (e.g., the city locations with abnormal pollutant concentration in our motivating example).
In OSC model, we formulate this selection problem under the form of $n$ simultaneous tests of
\begin{center}
$H_{0,i}: ``\pi_{i}= \mathcal{N}(0,1)"$ against $H_{1,i}: ``\pi_{i} \succ \mathcal{N}(0,1)"$, for all $1\leq i \leq n$.
\end{center}
(Remember that ``$\succ $" stands for strict stochastic domination). In the specific case of gOSC model~\eqref{model}, this problem reduces to simultaneously test 
\begin{equation}\label{hypomu}
\mbox{$H_{0,i}: ``\mu_{i}= 0"$ against $H_{1,i}: ``\mu_{i} >0"$.}
\end{equation}
We denote the set of non-outlier coordinates by 
$\cH_0(\pi)=\{1\leq i \leq n\::\: \pi_{i}=\mathcal{N}(0,1)\}$, 
and the set of outlier coordinates by 
 $\cH_1(\pi)=\{1\leq i \leq n\::\: \pi_{i}\succ\mathcal{N}(0,1)\}$.
 
The cardinal of $\cH_0(\pi)$ (resp. $\cH_1(\pi)$) is denoted by $n_0(\pi)$ (resp. $n_1(\pi)$).
 Hence, $\pi\in \overline{\cM}_k$, means that the number of outliers is $n_1(\pi)\leq k$. Thus, our selection problem amounts to estimate $\cH_1(\pi)$ (or equivalently $\cH_0(\pi)$). The dependence in $\pi$ of $\cH_0(\pi)$, $\cH_1(\pi)$, $n_0(\pi)$, $n_1(\pi)$ is sometimes removed for simplicity.

For any procedure declaring as outliers the elements of $R\subset \{1,\dots,n\}$, we quantify the amount of false positives in $R$ by a classical metric, introduced in \cite{BH1995}, which is called  the false discovery proportion of $R$:
\begin{equation}\label{equ-FDP}
\FDP(\pi, R) = \frac{|R \cap \cH_0(\pi) |}{|R|\vee 1}\ ,
\end{equation}
which corresponds to the proportion of errors among the set $R$ of selected outliers.
The expectation of this quantity is the  false discovery rate $\FDR(\pi,R)=\E_{\theta,\pi,\sigma} [\FDP(\pi, R)]$, which can be considered as the standard generalization of the single testing type I error rate  in large scale multiple testing.
The true discovery proportion is then  defined by
\begin{equation}\label{equ-TDP}
\TDP(\pi, R) = \frac{|R \cap \cH_1(\pi)|}{n_1(\pi)\vee 1}\ ,
\end{equation}
 and  corresponds to the proportion of (correctly) selected outliers  among the set of false null hypotheses.
The expectation of this quantity $\E_{\theta,\pi,\sigma} [\TDP(\pi, R)]$ is a widely used analogue of the power in single testing, see, e.g., \cite{RW2009,AC2017,RRJW2017}. 
 Our contribution falls into two frameworks:
\begin{itemize}
\item  {\it Multiple testing}: find a procedure selecting a subset $R\subset \{1,\dots ,n\}$ as close as possible to $\cH_1(\pi)$, i.e.~that has a TDP as high as possible while maintaining a controlled FDR;
\item {\it Post hoc bound}: provide a confidence bound on $\FDP(\pi,S)$, uniformly valid over all possible $S\subset \{1,\dots ,n\}$. 
\end{itemize}
 While the first objective is a classical multiple testing aim, see, e.g., \cite{BH1995,BY2001,FDR2007,GBS2009}, the second objective, relatively new, has been proposed in \cite{GW2004,GW2006,GS2011}. It is connected to the burgeoning research field of selective inference, see, e.g., \cite{BNR2017} and references therein. 
 The rationale behind developing such a bound is that, since the control is uniform, the probability coverage is guaranteed even if the user chooses an arbitrary $S$, possibly using the same data $Y$ and possibly several times. In other words, the commonly used ``data-snooping" is allowed with such bound. We denote the outlier selected set either by $R$ or $S$ depending on the considered issue: $R$ is typically a procedure designed by the statistician, whereas $S$ is chosen by the user.

 \subsubsection{Relation to the first objective and to previous literature}
 
In OSC model~\eqref{classicmodel_2}, solving the above multiple testing issues  is challenging
primarily  because of the unknown 
 parameters $\theta$ and $\sigma$. Indeed, this entails that the scaling of the null distribution (i.e.~the distribution under the null hypothesis) is unknown. A natural idea is to design a two-stage procedure: first,  we estimate $\theta$ and $\sigma$ 
 by some estimators $\wh{\theta}$ and  $\wh{\sigma}$ (actually this is precisely what we do in the first part of this paper). 
Then, in a testing stage, we apply a standard multiple testing procedure to the rescaled observation $Y'_i=(Y_i-\wh{\theta})/\wh{\sigma}$.

Estimating the null distribution in a multiple testing context has been popularized in a series of work of Efron, see \cite{Efron2004, Efron2007b,  Efron2009b}. Through careful data analyses, Efron noticed that the theoretical null distribution  often turns out to be  wrong in practical situations, which can lead to an uncontrolled increase of false positives. To address this issue, Efron recommends to
estimate the scaling parameters of the null distribution ($\theta,\sigma$ here) by ``central matching", that is, by fitting a parametric curve to the trimmed data. 
In his work, Efron provides compelling empirical evidence on his approach. However, up to our knowledge, the FDP and TDP of  such two-stage testing procedures has never been theoretically controlled.
Note that estimating the null in a multiple testing context was also the motivation of the minimax results of \cite{JC2007,CJ2010}, although the corresponding multiple testing procedure was not studied.
We recall that these previous studies are all developed in the two-sided  context, whereas our focus is on the one-sided shape constraint.

 \subsubsection{Summary of our results} 
  
In Section~\ref{sec:outliers}, we show that some minor modification of the quantile-based estimators  $\wh{\theta}$, $\wh{\sigma}$ introduced for OSC model,  can be used to estimate the null distribution to rescale the $p$-value process, and can then be suitably combined with classical multiples testing procedures: 
 
\begin{enumerate}
\item  A new ($\wh{\theta}, \wh{\sigma}$)-rescaled Benjamini-Hochberg procedure $R$ is defined and proved to enjoy the following FDR controlling property: in general model~\eqref{classicmodel_2}, for any $\pi\in \overline{\cM}_k$, with $k= \lfloor 0.9 n\rfloor$ (not anti-sparse signal),
$$
\left( \E_{\theta,\pi,\sigma}\left( \FDP(\pi, R) \right) - \frac{n_0}{n} \alpha \right)_+ \lesssim  \log(n)/n^{1/16}\ .
$$
In addition, we derive a power result showing that the power (TDP) of this procedure is close to the one of the (${\theta}, {\sigma}$)-rescaled Benjamini-Hochberg procedure (under mild conditions). The latter is an oracle benchmark that would require the exact knowledge of ${\theta}$ and ${\sigma}$. 

\item A new ($\wh{\theta}, \wh{\sigma}$)-rescaled post hoc bound $\ol{\FDP}(\cdot)$ is proposed, satisfying, for any $\pi\in \overline{\cM}_k$, with $k = 0.9 n$,
$$
\left(1-\alpha- \P_{\theta,\pi,\sigma}\left(\forall S\subset \{1,\dots,n\},\:\: \FDP(\pi,S)\leq \ol{\FDP}(S)\right)\right)_+  \lesssim 
 \log (n)/n^{1/16}\  .
$$
\end{enumerate}
  To our knowledge, these are the first theoretical results that fully validate Efron's principle of empirical null correction in a specific multiple testing problem.

For bounding the type I error rates, the technical argument used in our proof is close in spirit to recent studies \cite{LB2016,IH2017} (among others): the idea is to divide the data into two ``orthogonal" parts (small or large $Y_i$'s), the first part being used for the rescaling and the second one for testing.
For the power result, our formal argument is entirely new to our knowledge, as this kind of results is rarely met in the literature.

  \subsubsection{Application to decorrelation in multiple testing}\label{sec:deccor}

It is well known that Efron's methodology on empirical null correction can be applied to reduce the effect of correlations between the tests, as noted by Efron himself \cite{Efron2007,Efron2009} where he mentioned that ``there is a lot at stake here". Several following work supported this assertion, especially by decomposing the covariance matrix of the data into factors, see \cite{FKC2009,LS2008,Fan2012,Fan2017}. However, strong theoretical results on the corrected multiple testing procedure are still not available. 

Meanwhile, 
another branch of the literature aimed at incorporating known and unknown dependence into multiple testing procedures, for instance, by resampling-based approach \cite{WY1993,RW2005,RW2007,RSW2008,DL2008,BC2015} or by directly incorporating the known dependence structure \cite{GHS2013,DR2015b,Slope2015}.
However, as noted for instance in the discussion of \cite{Sar2008rej}, even for very simple correlation structures, no multiple testing procedure has yet been proved to control the FDR while having an optimal TDP. 

In Section~\ref{sec:equicor}, we apply our two-step procedure 
to address the multiple testing problem in the one-sided Gaussian equi-correlation case (with nonnegative equi-correlation $\rho$). This model (or its block diagonal variant) is often used as a concrete test bed in  multiple testing literature, see, e.g.,  \cite{Korn2004,DR2011,DR2016} among others. It turns out that this model can be written under the form of gOSC model~\eqref{model} with a random value of $\theta$ (the variable carrying the equi-correlation) and an unknown variance $\sigma=(1-\rho)^{1/2}$. Hence, we can directly apply our ($\wh{\theta}, \wh{\sigma}$)-rescaled Benjamini-Hochberg procedure introduced above to solve the problem: we show that the new procedure has performances close to the BH procedure under independence (and even with a slight increase of the signal to noise ratio). Even if the model is somewhat specific, this shows that correcting the dependence can be fully theoretically justified. To illustrate numerically the benefit of such an approach, Figure~\ref{fig:ROC} displays a ROC-type curve for four different versions of corrected BH procedure. A full description of the simulation setting and additional experiments are provided in Section~\ref{sec:equicor}.

\begin{figure}[h!]
\includegraphics[scale=0.5]{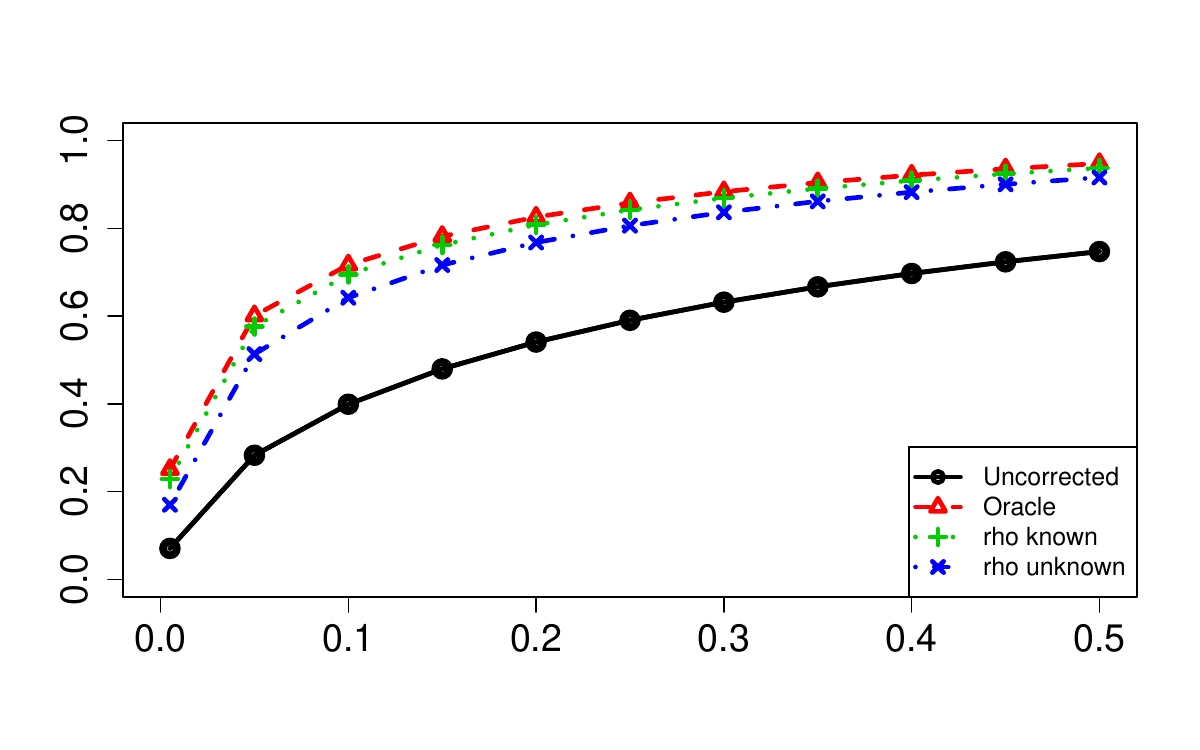}
\vspace{-1cm}
\caption{$X$-axis: targeted FDR level $\alpha\in\{0.005, 0.05, 0.1, 0.15, 0.2, 0.25, 0.3, 0.35, 0.4, 0.45, 0.5\}$, $Y$-axis: TDP (power) averaged over $100$ replications for four different procedures (see text in Section~\ref{sec:equicor}).  
The model is the one-sided Gaussian  with  equi-correlation $\rho=0.3$. The parameters used are $n=10^6$, $\Delta=2.5$, $k/n=0.1$.   
} \label{fig:ROC}
\end{figure}

\subsection{Notation}\label{sec:notation}

For $x>0$, we write $\lfloor x\rfloor^{(\log_2)}$ (resp. $\lceil x\rceil^{(\log_2)}$) for $2^{\lfloor \log_2(x)\rfloor}$ (resp.  $2^{\lceil \log_2(x)\rceil}$ ) the largest (resp. smallest) dyadic number smaller (resp. higher) than $x$. 
Similarly, $\lfloor x\rfloor_{\mathrm{even}}$ is  the largest even integer which is not higher than $x$.

For $x\in \R^n$, $x_{(k)}$ is the $k$-th smallest element of $\{x_i,1\leq i \leq n\}$. We also write $x_{(\ell:m)}$ for the $\ell$-smallest  element among $\{x_i, 1\leq i \leq m\}$, for some integer $1\leq m\leq n$.

{In the sequel, $c$, $c'$ denote numerical positive constants whose values may vary from line to line.} 
For two sequences $(u_t)_{t\in\mathcal{T}}$ and $(v_t)_{t\in\mathcal{T}}$, we write that for all $t\in\mathcal{T}$, $u_t \lesssim v_t$ (resp.~for all $t\in\mathcal{T}$, $u_t \gtrsim v_t$),
if  there exists a universal constant $c>0$ such that for all $t\in\mathcal{T}$, $u_t\leq c\: v_t$ (resp.~for all $t\in\mathcal{T}$, $u_t\geq c\: v_t$). We write $u_{t}\asymp v_{t}$ if $u_t \lesssim v_t$ and $v_t \lesssim u_t$.

For $X,Y$ two real random variables with respective cumulative distribution functions $F_X, F_Y$, we write $X\succeq Y$ 
 if for all $x\in \mathbb R$, we have
$F_X(x) \leq F_Y(x).$
We write $X\succ Y$ 
if  $X\succeq Y$ and if there exists $x\in \mathbb R$, such that 
$F_X(x) < F_Y(x).$ We also denote $P\succeq Q$ (resp.  $P\succ Q$) whenever $X\succeq Y$ (resp. $X\succ Y$) for $X\sim P$ and $Y\sim Q$.

For the standard normal distribution, we write $\Phi$ for its cumulative distribution function, $\bar \Phi = 1-\Phi$ and $\phi$ for its usual density.


\section{Estimation of $\theta$ in the gOSC model~\eqref{model}} 
\label{sec:fixedcont}

In this section, we consider the problem of estimating $\theta$ in  the Gaussian contamination model~\eqref{model} and investigate the $L_1$ minimax risk defined in \eqref{eq:definition_minimax_robust_one_sided}.
We assume throughout this section that $\sigma^2 = 1$.

\subsection{Lower bound}


\begin{thm} \label{thm:lower_one_sided}
There exists a universal constant $c>0$ such that  for any positive integer $n$ and for any integer $k\in [1,n-1]$,
\beq\label{eq:lower_one_sided}
 \cR[k,n]\geq c 
  \frac{\log^{2}\big(1+ \big(\frac{k}{n-k}\big)^{1/2}\big)}{\log^{3/2}\big(1+ (\frac{k}{\sqrt{n}})^{2/3}\big)}\ . 
\eeq
\end{thm}

The proof of this theorem is given in  Section~\ref{p:thm:lower_one_sided}.
The main tool for proving this lower bound is moment matching: we build two priors on the parameter $\gamma$ 
 that are related to two different values of $\theta$ (as far as possible) while having about $\log n$ first moments that coincide.
This is done in an implicit way by using the Hahn-Banach theorem together with properties of Chebychev polynomials, by using techniques close to \cite{Juditsky_convexity,carpentier2017adaptive}.

Let us distinguish between the three following regimes (see also Table~\ref{Tableminimax}):
\begin{itemize}
\item for $k\leq \sqrt{n}$,  the lower bound \eqref{eq:lower_one_sided} is of order $n^{-1/2}$, which is the parametric rate that would hold in the case of no contamination ({\it i.e.}, $k=0$); 
\item for $k\in (\sqrt{n}, \zeta n)$ with $\zeta\in(0,1)$, the lower bound is of the order $(k/n) \log^{-3/2}(k/\sqrt{n})$. 
In particular, in the non-sparse case $k=\lceil n/2\rceil$, we obtain  $\log^{-3/2} n$;
\item for larger $k$, e.g., $n/2\leq k \leq n-1$, the lower bound on the minimax risk is of order $\log^2(\tfrac{n}{n-k}) \log^{-3/2}(n)$. In particular, for $k=n-1$, the lower bound is of order $\log^{1/2} n$.
\end{itemize}

In the remainder of this section, we match these lower bounds by considering three different estimators of $\theta$, corresponding to the three regimes discussed above. They are then combined to derive an adaptive estimator.

\subsection{Upper bound for small and large $k$}

For small and for large values $k$ the optimal risk is achieved by simple quantile estimators. For $k\leq n^{1/2}$, we consider the empirical median defined by
\begin{equation}\label{eququantilesti}
\thetachapmed=Y_{\left(\lceil n/2\rceil\right)}\ .
\end{equation}
The following result holds for $\thetachapmed$ (note that it is stated in the more general OSC model~\eqref{classicmodel_2}).

\begin{prp}\label{prop:thetamedian}
Consider OSC model~\eqref{classicmodel_2} with $\sigma=1$.
Then there exist universal positive constants $c_1,c_2$ and a universal positive integer $n_0$ such that the following holds. For any $n\geq n_0$, any $k\leq n/10$, any $\pi\in \overline{\cM}_k$ and any $\theta\in \mathbb{R}$, we have
\beqn 
\P_{\theta,\pi}\left[|\thetachapmed- \theta|\geq \frac{3(k+1)}{2(n-k)}+ 3\frac{\sqrt{(n+1)x}}{n-k} \right]&\leq& e^{-x}\ \ , \quad \quad \text{ for all }x\leq c_1 n\ , \\
\E_{\theta,\pi}\big[|\thetachapmed- \theta| \big]&\leq& \frac{3(k+1)}{2(n-k)}+ \frac{c_2}{\sqrt{n}}\ . 
\eeqn 
\end{prp}
A proof is provided in Section~\ref{p:prelest}. A consequence if that, for $k\leq \sqrt{n}$, the empirical median $\thetachapmed$ achieves the parametric rate  $n^{-1/2}$, which turns out to be optimal in this regime, see Theorem~\ref{thm:lower_one_sided}.  
Note that in that regime $k\geq \sqrt{n}$, the empirical medial was already known to achieve this parametric rate in the more general Huber's contamination model, that allows for two sided contaminations.

\medskip 

When $k$ is really close $n$, there are very few non contaminated data. Since $\theta= \min_{i}\gamma_i$ in this model~\eqref{model}, we consider debiased  empirical minimum estimator 
\begin{equation}\label{eququantilesti_extreme}
\thetachapmin= Y_{(1)}+ \overline{\Phi}^{-1}(1/n) \ ,
\end{equation}
where we recall that  $\overline{\Phi}^{-1}(1/n)=   \sqrt{2\log(n)}+ O(1) $, see Section~\ref{sec:ineq-quantile}.
The following result holds for $\thetachapmin$ (note that it is also stated in the more general OSC model~\eqref{classicmodel_2}).

\begin{prp}\label{prp:dense}
Consider OSC model~\eqref{classicmodel_2} with $\sigma=1$.
 Then there exists some universal positive integer $n_0$ such that for any $n\geq n_0$, any $\pi\in \overline{\cM}_{n-1}$ and $\theta\in\R$, the  estimator $\thetachapmin$ satisfies 
 \beqn
 \P_{\theta,\pi}\left[|\thetachapmin- \theta|\geq 2\sqrt{2\log n} \right]\leq \frac{2}{n}\ ; \quad \E_{\theta,\pi}\big[|\thetachapmin- \theta| \big]\leq 2\sqrt{2\log n} + 1 \ .
 \eeqn 
\end{prp}
A proof  is provided in Section~\ref{p:prelest}. From Theorem~\ref{thm:lower_one_sided}, the estimator $\wh{\theta}_{\min}$ turns out to be optimal when $k$ is very close to $n$, e.g. when $k$ larger than $n - n^\epsilon$ for a fixed $\epsilon\in (0,1)$,
that is when very few samples are non-contaminated.

\subsection{Upper bound in the intermediate regime}\label{sec:intermreg}

In the previous section, we have introduced estimators that are optimal in the regimes where $k \leq \sqrt{n}$ and where $k$ is very close to $n$, respectively. The intermediate case turns out to be much more involved.


Let $q\geq 2$ be an even integer whose value will be fixed below. Let $\const= 3[1+ \log(3+2\sqrt{2})]\approx 8.29 $ and $q_{\max}= \lfloor \frac{1}{2\const }\log n\rfloor_{\mathrm{even}}-2$, where $\lfloor \cdot\rfloor_{\mathrm{even}}$ is defined in Section~\ref{sec:notation}. 
Let us also introduce two rough estimators $\thetachapup$ and $\thetachaplow{q}$ such that $\theta$ is proved to belong to $[ \thetachaplow{q},\thetachapup]$ with high probability. 
Let  $ \thetachapup=Y_{(1)}+ 2 \sqrt{\log n}$. For any positive and even integer $q$, define $\thetachaplow{q}= \wh{\theta}_{\med}- \overline{v}$  with $ \overline{v}= \pi^2/(144\: q_{\max}^{3/2})$ if $q\leq \tfrac{3}{10\const}\log n$ and  $\thetachaplow{q}=-\infty$ for larger $q$.

\medskip

To explain the intuition behind our procedure, assume for the purpose of the discussion, that we have access to the mean $\gamma_i=\theta+\mu_i$ and that instead of estimating $\theta$, we simply want to test whether  $\theta$ is greater than $u$ or not. Thus, our aim is to define a suitable function of $\gamma_i$ which is close to zero for $\gamma_i  \geq u$ and the largest possible when $\gamma_i<  u$. Since at least $n-k$'s of the $\gamma_i$'s are equal to $\theta$, a large value of this function would entail that $\theta<u$. This can be achieved by building  $g_q:\mathbb{R}\mapsto \mathbb{R}$ such that $|g_q(x)|\leq 1$ and for $x\in (-\infty,0]$ and $g_q(x)$ large for $x>0$ (assuming $u=0$, without loss of generality). If the interval $(-\infty,0]$ had been replaced by $[-1,1]$ and the function $g_q$ was restricted to be a polynomial, this would look like a  polynomial extremum  problem, which is achieved by a Chebychev polynomial (see Section~\ref{tchebysection} for some definitions and properties). To handle the non-bounded interval  $(-\infty, 0]$, we map $(-\infty, 0]$ to $(-1,1]$ using the function $x\mapsto 2e^{x}-1$ before using Chebychev polynomials of order $q$. Denoting by $T_q$ the Chebychev polynomial of degree $q$, this leads us to considering the function
\begin{equation}\label{eq:defintion_g}
g_q(x)=T_q(2 e^{x}-1)  = \sum_{j=0}^{q} a_{j,q}{e^{xj}},\:\:\:x\in\R\ ,
\end{equation}
where the coefficients $a_{j,q}$ are defined in \eqref{eq:ajq}.
It follows from the definition of Chebychev polynomials that $g_q(x)$ belongs to $[-1,1]$ for $x\leq 0$ and $g_q(x)= \cosh[q \arg\cosh(2 e^{x}-1)]$ for $x>0$.

\medskip 

Now consider, for $\lambda>0$ and $u\in \mathbb{R}$, the function $\psi_{q,\lambda}(u)$ defined by 
\begin{equation}\label{eq:definition_psi}
\psi_{q,\lambda}(u)=\frac{1}{n}\sum_{i=1}^n  g_q(\lambda(u - \gamma_i)) = \frac{1}{n}\sum_{i=1}^n  g_q(\lambda(u - \theta - \mu_i))\ . 
\end{equation}
This functions depends on the $\gamma_i$'s. Since all $\mu_i$'s are non negative, it follows from the above observation, that $|\psi_{q,\lambda}(u)|\leq 1$  for all $u\leq \theta$. Conversely, for $u\geq \theta$, $\psi_{q,\lambda}(u)$ is lower bounded as follows 
\begin{equation}\label{relationPsi}
 \psi_{q,\lambda}(u) \geq - \frac{k}{n}+ \frac{n-k}{n}g_q(\lambda(u-\theta))\ ,
\end{equation}
which is bounded away from $1$ as long as $u-\theta$ is large enough. As a consequence, 
the smallest number $u_*$ that satisfies  $\psi_{q,\lambda}(u_*)>1$ should be close (in some sense) to $\theta$. 

\medskip

Obviously, we do not have access to the function $\psi_{q,\lambda}$ as it requires the knowledge of the $\gamma_i$'s or more precisely of quantities of the form $e^{-j \lambda \gamma_i}$. Nevertheless, we can still build an unbiased estimator of such quantities relying on the empirical Laplace transform of $Y$. Given $\lambda>0$ and $u\in \mathbb{R}$ define
\begin{equation}\label{eq:laplace_empirical}
\wh{\eta}_{\lambda}(u) =
n^{-1} \sum_{i=1}^n e^{\lambda(u-Y_i) -\lambda^2/2}\ , \quad \quad \eta_{\lambda}(u) =
n^{-1} \sum_{i=1}^n e^{\lambda (u - \theta-\mu_i) }\ .
\end{equation}
Since all $Y_i$'s are independent with normal distribution of unit variance, we have $\E[\wh{\eta}_{\lambda}(u)]= \eta_{\lambda}(u)$. This leads us to considering  the statistic
\begin{equation}\label{eq:definition}
\wh{\psi}_{q,\lambda}(u) =  \sum_{j=0}^{q} a_{j,q} \wh{\eta}_{j\cdot\lambda}(u)\ ,
 \end{equation}
which is an unbiased estimator of $\psi_{q,\lambda}(u)$ for any fixed $\lambda>0$ and $u\in \mathbb{R}$. Since $\wh{\psi}_{q,\lambda}(u)$ approximates $\psi_{q,\lambda}(u)$, it is tempting to take $\wh{\theta}_q$ as the smallest value such that $\wh{\psi}_{q,\lambda}(u)$ is bounded away from $1$.

Intuitively, $\wh{\psi}_{q,\lambda}(u)$ is large compared to 1 when $u > \theta$. This is why we define $\wh{\theta}_q$ by inverting the function $\wh{\psi}_{q,\lambda}(.)$. More precisely, for an even integer $q\leq q_{\max}$, we define $\lambda_q= \sqrt{2/q}$ and the estimator $\wh{\theta}_q$ by 
\begin{equation}
 \label{eq:definition_theta_k}
 \wh{\theta}_q= \inf\bigg\{u\in \left[\thetachaplow{q},\thetachapup \right]\, :\,  \wh{\psi}_{q,\lambda_q}(u)> 1 + \frac{e^{\const q}}{\sqrt{n}}\bigg\}\ ,
\end{equation}
with the convention $\inf\{\emptyset\}= \thetachapup$.



\begin{thm}\label{th-upperbound-onesided}
Consider  gOSC model \eqref{model} with known variance $\sigma=1$. 
 There exist universal positive constants $c_1$, $c_2$, $c_3$, and  $n_0$ such that the following holds for any $n\geq n_0$, any integer $k\in \big[e^{2\const} \sqrt{n}, n - 64n^{1-1/(4\const)}\big)$,  any $\mu\in\cM_k$ and any $\theta\in \mathbb{R}$. The estimator $\wh{\theta}_{q_k}$ defined by \eqref{eq:definition_theta_k} with $q_k= \lfloor \frac{1}{\const}\log\big(\frac{k}{\sqrt{n}}\big) \rfloor_{\mathrm{even}}\wedge q_{\max} $ satisfies
\begin{equation}\label{eq:resultthetachap}\P_{\theta,\mu}\left(\wh{\theta}_{q_k} \notin \left[\theta\:,\: \theta+  c_1\frac{\log^{2}\big(1+ \sqrt{\frac{k}{n-k}}\big)}{\log^{3/2}\big(\frac{k}{\sqrt{n}}\big)}\right]\right)\leq c_3 \left(\frac{\sqrt{n}}{k}\right)^{4/3}\log^{3}\bigg(\frac{k}{\sqrt{n}}\bigg)\ ,
\end{equation}
and 
\begin{equation}\label{eq:risk_theta_tilde}
\E_{\theta,\mu}\left[|\wh{\theta}_{q_k} -\theta|\right]\leq  c_2\frac{\log^{2}\bigg(1+ \sqrt{\frac{k}{n-k}}\bigg)}{\log^{3/2}\big(\frac{k}{\sqrt{n}}\big)} \ . 
\end{equation}
\end{thm}
A proof is provided in Section~\ref{p:th-upperbound-onesided}. This result shows that $\wh{\theta}_{q_k}$ has a maximum risk of order $  \frac{k}{n}\log^{-3/2}(k/\sqrt{n})$  in the regime $k \in  [e^{2\const} \sqrt{n}, n - 64n^{1-1/(4\const)})]$. Combined with the lower bound of Theorem~\ref{thm:lower_one_sided}, we  have shown that $\wh{\theta}_{q_k}$ is minimax in the intermediate regime.

\begin{remark}
{Let us emphasize that in the regime $e^{2\const} \sqrt{n} \leq k\leq \lfloor n/2\rfloor$, the minimax risk is of order $(k/n)\log^{-3/2} (n)$, which is faster than the minimax rate $(k/n)\log^{-1/2} (n)$ that we would obtained  in a two-sided deconvolution problem, as in \cite{CJ2010} where $k/n\propto n^{-\beta}$ (and by considering the extreme case where there is no regularity assumption, that is, $\alpha=0$ with their notation). 
}
\end{remark}

\begin{remark}
 If we are only interested in a probability bound \eqref{eq:resultthetachap} and not in the moment bound \eqref{eq:risk_theta_tilde}, the preliminary estimators $\thetachaplow{q}$ and $\thetachapup$ are not needed: the estimator  could be computed by taking the minimum over $\mathbb{R}$ in  \eqref{eq:definition_theta_k}. 
\end{remark}

\subsection{Adaptative estimation}

In this section, we combine the three estimators studied in the above section to obtain an estimator that is adaptive with respect to the parameter $k$. The method relies   is a Goldenshluger-Lepski approach, see, e.g., \cite{lepski90,GL2011,LM2016}.

To unify notation, we write henceforth $\wh{\theta}_0$ for the median estimator $\wh{\theta}_{\med}$ 
and $\wh{\theta}_{q_{\max}+2}$ for the minimum estimator $\thetachapmin$. In order to obtain an adaptive procedure, we select one of the estimators $\{\wh{\theta}_q, q\in \{0,2,\ldots, q_{\max},q_{\max}+2\}\}$ as follows: 
\beq\label{eq:def_lepski}
\wh{q}= \min\big\{q\in \{0,\ldots, q_{\max}+2\}\ \text{ s.t.} \quad |\wh{\theta}_q-\wh{\theta}_{q'}| \leq \delta_{q'} \text{ for all }q'>q\big\}\ ,
\eeq
where the thresholds are chosen such that $\delta_q= 10\frac{ e^{\const (q+2)}}{ \sqrt{n}q^{3/2}}$ for $q\in \{2,\ldots, q_{\max}-2\}$,  $\delta_{q_{\max}}= \frac{25}{  q_{\max}^{3/2}}$ and  $\delta_{q_{\max}+2}= 4\sqrt{2\log n }$ (the value of $a$ being the same as in Section~\ref{sec:intermreg}).

\begin{thm}\label{thm:adaptation}
Consider gOSC model \eqref{model} with known variance $\sigma=1$. 
 There exist universal positive constants $c_1$, $c_2$, $c_3$, and $n_0$ such that the following holds. For any $n\geq n_0$, for any integer $k\in [1, n-1]$, any $\theta\in \mathbb{R}$, and any $\mu\in \cM_k$, the adaptive estimator $\wh{\theta}_{\ad}=\wh{\theta}_{\hat{q}}$  satisfies
\begin{equation}\label{eq:result_adaptation1_deviation}
\P_{\theta,\mu}\left[|\wh{\theta}_{\ad} -\theta|> c_1\frac{\log^{2}\big(1+ \sqrt{\frac{k}{n-k}}\big)}{\log^{3/2}\big(1+ (\frac{k}{\sqrt{n}})^{2/3}\big)}\right]\leq c_2 \left(\frac{\sqrt{n}}{k\vee \sqrt{n}}\right)^{4/3}\log^{3}\bigg(\frac{k\vee (2\sqrt{n}) }{\sqrt{n}}\bigg)\ ,
\end{equation}
and  
 \beq\label{eq:result_adaptation1_moment}
\E_{\theta,\mu}\Big[|\wh{\theta}_{\ad}-\theta|\Big]\leq c_3 \frac{\log^{2}\big(1+ \sqrt{\frac{k}{n-k}}\big)}{\log^{3/2}\big(1+ (\frac{k}{\sqrt{n}})^{2/3}\big)}\ . 
\eeq
\end{thm}
A proof is given in Section~\ref{p:thm:adaptation}. The risk bound in \eqref{eq:result_adaptation1_moment} matches the minimax lower bound of Theorem~\ref{thm:lower_one_sided} for all $k=1,\dots, n-1$. The estimator $\wh{\theta}_{\ad}$ is therefore minimax adaptive with respect to $k$.

\begin{remark}
Theorem~\ref{thm:adaptation} shows that the rate of estimation is not affected by the adaptation step. This is specific to our problem, for which the deviations of our estimators are very small when compared to its bias when $k \geq \sqrt{n}$, whereas a single estimator, the empirical median, has already good performances over the range $k \leq \sqrt{n}$.
\end{remark}

\section{Estimation of $\theta$ in the general OSC model}\label{sec:robust}

In this section, we study the estimation problem in the general OSC  model \eqref{classicmodel_2}. Hence, the contaminations are not assumed anymore to be Gaussian.
Throughout this section, $\sigma$ is assumed to be known and equal to $1$. Recall that the corresponding 
$L_1$ minimax risk is given by \eqref{eq:definition_minimax_robust_one_sided}.

\subsection{Lower bound}

We first show that estimating $\theta$ becomes more difficult under this model than for the  gOSC case.

\begin{thm}\label{thm:lower_mean_robust}
 There exists a universal positive constant $c$ such that for any positive integer $n$ and for any integer $k\in[1, n-1]$, 
\beq \label{eq:lower_minimax_robust}
\overline{\cR}[k,n] \geq c   \frac{\log\left(\frac{n}{n-k}\right)}{\log^{1/2}(1+ \frac{k^2}{n})}\ . 
\eeq
\end{thm}
A proof  is provided in Section~\ref{p:thm:lower_mean_robust}. 
Let us comment briefly the order of this lower bound, by going back to the three aforementioned regimes (see also Table~\ref{Tableminimax}):
\begin{itemize}
\item for $k\leq \sqrt{n}$,  the lower bound \eqref{eq:lower_minimax_robust} is of order $n^{-1/2}$, which is the parametric rate, hence is the same as for the Gaussian case; 
\item for $k\in (\sqrt{n}, \zeta n)$ with $\zeta\in(0,1)$, the lower bound is of the order $(k/n) \log^{-1/2}(k/\sqrt{n})$, so is strictly slower than with the Gaussian assumption (additional factor of order $\log(k/\sqrt{n})$). 
In particular, in the non-sparse case $k=\lceil n/2\rceil$, this gives a lower bound of order $\log^{-1/2} (n)$ (in contrast to $\log^{-3/2} (n)$ in the Gaussian model)
\item for larger $k$, e.g., $n/2\leq k \leq n-1$, the lower bound is of order $\log(n/(n-k))\log^{-1/2}(n)$. In comparison to gOSC, there is an additional factor of order $\log (n)/\log(n/(n-k))$ . Nevertheless, in the extreme case $k=n-1$, the two lower bounds are of order $\log^{1/2}(n)$.
\end{itemize}
In the next subsection, these lower bounds are proved to be sharp.

\subsection{Upper bound}\label{sec:estitheta}

In this subsection, we introduce a bias-corrected quantile estimator that matches the minimax lower bound of Theorem 
\ref{thm:lower_mean_robust}. Consider some $\pi\in \ol{\cM}_k$.
Let $\xi= (\xi_1,\ldots, \xi_n)$ denote a standard Gaussian vector.
The starting point is the following:
on the one hand, all random variables $Y_i-\theta$ stochastically dominate $\xi_i$ so that $Y_{(q)} -\theta \succeq \xi_{(q)}$. On the other hand, $Y_{(q)}$ is stochastically dominated by the $q$-th smallest observation among the {\it non-contaminated} data $Y_j$. As a consequence, we have 
\begin{equation}\label{equ-quantile-encadrement}
\xi_{(q)} \preceq Y_{(q)} -\theta  \preceq   \xi_{(q:(n-k))}\ ,
\end{equation}
where we recall that $\xi_{(q:(n-k))}$ is the $q$-th largest observation among the $n-k$ first observations of $\xi$.
Since $\xi_{(q)}$ is concentrated around $\ol{\Phi}^{-1}(q/n)$, this leads to introducing the debiased estimator 
\begin{equation}\label{estimrobust}
\wt{\theta}_q=Y_{(q)}+  \ol{\Phi}^{-1}(q/n) \:\:\:\:1\leq q\leq \lceil n/2\rceil\ .
\end{equation}
In view of \eqref{eququantilesti}, we have that $\wt{\theta}_1=\thetachapmin$ while $\wt{\theta}_{\lceil n/2\rceil}$ is almost equal to the empirical median $\thetachapmed$ (up to the additive $\ol{\Phi}^{-1}(\lceil n/2\rceil/n)$ term which is of order $1/n$ so is negligible). The following theorem bounds the error of $\widetilde{\theta}_q$ for a wide range of $q$. 
\begin{thm} \label{thm:upper_robust}
Consider  OSC  model \eqref{classicmodel_2} with known variance $\sigma=1$.
There exist universal positive constants $c_1$, $c_2$, $c'_2$,  $c_3$, $c_4$ such that the following holds.  For all positive integers $k \leq n-1$, any  $q$ such that $c_4\log n\leq q \leq (0.7(n-k)) \wedge\lceil n/2\rceil$,  any $\theta\in \mathbb{R}$ and any $\pi \in \ol{\cM}_k$, the estimator $\wt{\theta}_q$ satisfies
\beq\label{eq:result_thetat_robustes}
  \P_{\theta,\pi}\left[- c_2 \sqrt{\frac{x}{q[\log(\frac{n-k}{q})\vee 1]}} \leq \wt{\theta}_q -\theta\leq c_1 \frac{\log\left(\frac{n}{n-k}\right)}{\sqrt{\log\big(\frac{n-k}{q}\big)\vee 1}} + c_2 \sqrt{\frac{x}{q[\log(\frac{n-k}{q})\vee 1]}}\right]\geq 1 - 2e^{-x}\ , 
\eeq
for all $0<x<c_3 q$ and
\begin{equation}\label{eq:risk_theta_tilde_robust}
\E_{\theta,\pi}\left[|\wt{\theta}_{q} -\theta|\right]\leq  c_1 \frac{\log\left(\frac{n}{n-k}\right)}{\sqrt{\log\big(\frac{n-k}{q}\big)\vee 1}} + c'_2 \frac{1}{\sqrt{q[\log(\frac{n-k}{q})\vee 1]}}\ . 
\end{equation}
\end{thm}

A proof is given in Section~\ref{p:thm:upper_robust}. 
The risk bound in \eqref{eq:risk_theta_tilde_robust} exhibits a bias/variance trade-off as a function of $q$ via the quantities
$$
b(q)= \frac{\log\left(\frac{n}{n-k}\right)}{\sqrt{\log\big(\frac{n-k}{q}\big)\vee 1}}\:;\:\:\:s(q)=\frac{1}{\sqrt{q[\log(\frac{n-k}{q})\vee 1]}}\ .
$$
The quantity $s(q)$ is a deviation term that decreases with $q$ and whose minimum is of the order of $n^{-1/2}$. This minimum is achieved for 
$q=\lceil n/2\rceil$ and the  corresponding estimator is close to the empirical median.
The quantity $b(q)$ 
 is a bias term which increases slowly with $q$. Its minimum is of the order of $\log\big(\frac{n}{n-k}\big)\log^{-1/2}\big(n-k\big)$ and is achieved for $q$ constant (or of the order of $\log n$). The  corresponding estimators are extreme quantiles such as  $\widetilde{\theta}_1=\thetachapmin$.

Note also that the condition $c_4\log(n)\leq q \leq 0.7(n-k)$ cannot be met when $k$ is too  close to $n$ (i.e. $n-k< (c_4/0.7)\log(n)$). Hence, Theorem~\ref{thm:upper_robust} is silent in that regime. Nevertheless, this case is addressed by the minimum estimator  $\widetilde{\theta}_1=\thetachapmin$ already studied in Proposition~\ref{prp:dense}.

To achieve the minimax risk, it remains to suitably choose $q$ as a function of $k$. In view of $b(.)$ and $s(.)$, when $k$ is large, one should consider a smaller $q$ and therefore more extreme quantile in order to decrease the bias. More precisely, we define
\begin{equation}\label{defqkrobust}
q_k =
\left\{\begin{array}{cc} \lceil n/2\rceil &\mbox{ if $k\in [1, 4\sqrt{n})$ ;}\\ \lceil \frac{n^{5/4}}{ \:k^{1/2}}\rceil^{(\log_2)}&\mbox{ if $k\in [4\sqrt{n}, n-n^{4/5}]$ ;} \\ 1 &\mbox{ if $k\in (n-n^{4/5}, n-1]$ .}\end{array}\right.
\end{equation}
In the very sparse situation $(k\leq 4\sqrt{n})$, $\wt{\theta}_{q_k}$ corresponds to the empirical median. For $k$ increasing to $n$, $q_k$ goes smoothly to $n^{1/4}$. Finally, when $k$ is very close to $n$, we consider the minimum estimator $\wt{\theta}_1$. Other choices of $q_k$ may also lead to optimal risk bounds and the choice \eqref{defqkrobust} is made to simplify the proofs.

\begin{cor}\label{prp:upper_bound_robuste_non_adaptive}
Consider OSC  model \eqref{classicmodel_2} with known variance $\sigma=1$.
There exist universal positive  constants $c$ and $n_0$  such that the following holds. For any integer $n\geq n_0$, any integer $k\in\{1,\ldots, n-1\}$, any $\theta\in \mathbb{R}$ and any $\pi\in \overline{\cM}_k$, the estimator $\wt{\theta}_{q_k}$ satisfies 
\begin{equation}\label{eq:upper_minimax_robust}
\E_{\theta,\pi}\left[|\wt{\theta}_{q_k} -\theta|\right]\leq c   \frac{\log\left(\frac{n}{n-k}\right)}{\log^{1/2}(1+ \frac{k^2}{n})}\ .  
\end{equation}
\end{cor}
A proof is given in Section~\ref{p:thm:upper_robust}. The estimation rate of the estimator $\wt{\theta}_{q_k}$ matches the minimax lower bound given in Theorem~\ref{thm:lower_mean_robust}. However, it is not adaptive because it uses the value of $k$.

\subsection{Adaptive estimation}

 We now provide a procedure that adapts to $k$, by following a Goldenshluger-Lepski approach.
Let $\cQ$ denote the collections of values of $q_k$ when $k$ goes from $1$ to $n-1$. This collection contains $1$, $\lceil n/2\rceil$ and a dyadic sequence from $n^{1/4}$ to $n/2$ (roughly). To build an adaptive procedure, we select among the estimators $\{\widetilde{\theta}_{q},  q\in\cQ\} $ in the following way:
\beq\label{eq:def_lepski_robust}
\wh{q}= \max\big\{ q \in\cQ \text{ s.t.} \quad |\wh{\theta}_{q}-\wh{\theta}_{q'}| \leq \delta_{q'} \text{ for all }q'<q\big\}\ ,
\eeq
where
\beq\label{eq:definition_delta_q_2}
\delta_{q}= c_0 \left\{
\begin{array}{cc}
 \sqrt{\log(n)}& \text{ if }q<   \sqrt{2}n^{1/4}\ ;\\
\frac{n^{1/6}}{q^{2/3}\sqrt{\log(\frac{n}{q})\vee 1}} & \text{otherwise}\ ,
\end{array}
\right.
\eeq
where the constant $c_0$ is large enough (and depends on $c_1$ and $c'_2$ in Theorem \ref{thm:upper_robust}). Note that at most three elements in $\cQ$ are less than $\sqrt{2}n^{1/4}$. 

\begin{prp}\label{prp:adaptation_robust}
Consider OSC  model \eqref{classicmodel_2}  with known variance $\sigma=1$.
There exist universal positive  constants $c$ and $n_0$ such that the following holds. For any integer $n\geq n_0$, any integer $k\in\{1,\ldots,n-1\}$, any $\theta\in \mathbb{R}$, and any $\pi\in \overline{\cM}_k$, the estimator 
  $\widetilde{\theta}_{\ad}= \widetilde{\theta}_{{\hat{q}}}$ (see \eqref{defqkrobust} and \eqref{eq:def_lepski_robust}) satisfies
 \[
 \E_{\theta,\pi}\Big[|\wt{\theta}_{\ad}-\theta|\Big]\leq c \frac{\log\big(\frac{n}{n-k}\big)}{\log^{1/2}\big(1+ \frac{k^2}{n}\big)}\ . 
\]
\end{prp}
A proof is given in Section~\ref{p:thm:upper_robust}.  The above result shows that, as in the Gaussian case, adaptation with respect to $k$ can be achieved without any loss.

\section{Unknown variance}\label{sec:randcontsigma}

In this section, we consider OSC model \eqref{classicmodel_2} for which the noise variance $\sigma^2$ is unknown.
We derive the minimax risks and estimators for $\theta$ and $\sigma$ in that setting.

\subsection{Lower bound}

First note that, obviously, the lower bound \eqref{eq:lower_minimax_robust} for estimating $\theta$  ais also a valid lower bound for the minimax risk 
$$
 \inf_{\wt{\theta}}\sup_{\theta\in\R, \sigma>0 , \pi \in \overline{\cM}_{k} }\E_{\theta,\pi,\sigma}\Big[\frac{|\wt{\theta}-\theta|}{\sigma}\Big]\ ,
$$
corresponding to the OSC model \eqref{classicmodel_2} where $\sigma$ is unknown.

Now, let us provide a lower bound for the estimation risk of $\sigma$.
As above, it is enough to consider the case 
where $\theta$ is known and equal to zero. 
This corresponds to the minimax risk  
\beq\label{eq:definition_minimax_robust_var}
\overline{\cR}_v[k,n]= \inf_{\wt{\sigma}}\sup_{\sigma>0,   \pi \in \overline{\cM}_{k} }\E_{0,\pi,\sigma}\Big[\frac{|\wt{\sigma}-\sigma|}{\sigma}\Big]\ .
\eeq
The following theorem provides a lower bound for $\overline{\cR}_v[k,n]$ (and therefore also a lower bound on the minimax risk with arbitrary unknown $\theta$):
\begin{thm}\label{thm:lower_var_robust}
There exists a universal positive constant $c$ such that for any integer $n\geq 2$ and any $k=1,\ldots, n-2$, we have
\beq\label{eq:lower_var_robust}
\overline{\cR}_v[k,n]\geq c \frac{\log\left(\frac{n}{n-k}\right)}{\log\left(1+ \frac{k}{n^{1/2}}\right)}\ . 
\eeq
\end{thm}
A proof  is given in Section~\ref{p:thm:lower_var_robust}. 
For $k\leq \sqrt{n}$, the lower bound \eqref{eq:lower_var_robust} is of order $n^{-1/2}$. For $k\in [\sqrt{n}, \zeta n]$ (with $\zeta\in (0,1)$ fixed), the risk is of order $k/[n\log(k^2/n)]$ which is faster by a $\log^{1/2}(k^2/n)$ term than for mean estimation. When  $n-k= n^{\gamma}$ with $\gamma\in (0,1)$ (almost no uncontaminated data), the relative rate of convergence is at least constant. 

In the next section, we prove that these lower bounds on $\theta$ and $\sigma$ are all sharp (up to numerical constants).

\subsection{Upper bound}\label{sec:lbsigma} 

Since the model is translation invariant, estimating the variance can be done without knowing $\theta$. This is done by considering 
rescaled  differences of empirical quantiles. More precisely, for 
 two positive integers $1\leq q'\leq q\leq n$, let 
\beq  \label{eq:definition_sigma}
\wt{\sigma}_{q,q'}= \frac{Y_{(q)}-Y_{(q')}}{\overline{\Phi}^{-1}(q'/n)- \overline{\Phi}^{-1}(q/n)}\ , 
\eeq
with the convention $0/0=0$.  When $k= 0$ (no contamination), $Y_{(q)}$ (resp. $Y_{(q')}$) should be close to $\theta - \sigma \ol{\Phi}^{-1}(q/n)$ (resp. $\theta - \sigma\ol{\Phi}^{-1}(q'/n)$) so that, intuitively, $\wt{\sigma}_{q,q'}$ should be close to $\sigma$.  Then, to estimate $\theta$, we simply plug $\wt{\sigma}_{q,q'}$ into the quantile estimators  considered in Section~\ref{sec:estitheta}. More precisely, we consider 
\beq  \label{eq:definition_theta_unknown}
\wt{\theta}_{q,q'}  =  Y_{(q)} + \wt{\sigma}_{q,q'}\ol{\Phi}^{-1}\left(\frac{q}{n}\right) \ . 
\eeq
Given $k\in \{1,\ldots, n-1\}$,   $q_k$ is taken as in \eqref{defqkrobust} and 
\begin{equation}\label{defq'krobust}
q'_k =
\left\{\begin{array}{cc} \lceil n/3\rceil &\mbox{ if $k\in [1, 4\sqrt{n})$}\ ;\\ \lfloor \frac{n^{7/4}}{ \:k^{3/2}}\rfloor^{(\log_2)}&\text{ if }k\in [4\sqrt{n}, n-n^{4/5}]\ ; \\
1 & \text{ if } k \in ( n-n^{4/5},n-2]\ .
\end{array}
\right. 
\end{equation}
For sparse contaminations $(k\leq 4\sqrt{n})$, $\wt{\sigma}_{q_k,q'_k}$ is a rescaled difference of the empirical median and the empirical quantile of order $1/3$. For a larger  number of contaminations,  more extreme quantiles are considered. For $k\geq n-n^{4/5}$, we simply take $\wt{\sigma}_{q_k,q'_k}=0$. 

\begin{prp}\label{prp:upper_unknown}
Consider  OSC model \eqref{classicmodel_2} with unknown variance $\sigma^2$  and the quantity $q_k$ and $q_k'$ defined in \eqref{defqkrobust} and \eqref{defq'krobust}.
There exist universal positive constants $c$, $c'$, and $n_0$ such that the following holds.  For any $n\geq n_0$, 
for any positive integer $k\leq n-2$, any $\theta\in \R$, any $\sigma>0$ and any $\pi \in \ol{\cM}_{k}$, we have 
\begin{eqnarray}
 \mathbb{E}_{\theta,\pi,\sigma}\big[|\wt{\sigma}_{q_k,q'_k}-\sigma|/\sigma\big]&\leq &c\frac{\log\left(\frac{n}{n-k}\right)}{\log(1+ \frac{k}{n^{1/2}})}\  ; \label{eq:risk_upper_sigma_hat}\\
 \mathbb{E}_{\theta,\pi,\sigma}\big[|\wt{\theta}_{q_k,q'_k}-\theta|/\sigma\big]&\leq& c'\frac{\log\left(\frac{n}{n-k}\right)}{\log^{1/2}(1+ \frac{k^2}{n})}\ .
\end{eqnarray}
\end{prp}
A proof is given in Section~\ref{p:thm:upper_robust}. 

The above proposition together with the lower bounds of Section~\ref{sec:lbsigma} implies that $\wt{\sigma}_{q_k,q'_k}$ and $\wt{\theta}_{q_k,q'_k}$ are minimax estimator of $\sigma$ and $\theta$, respectively. In particular, not knowing the variance does not increase the minimax rate when estimating $\theta$.

\section{Controlled selection of outliers}\label{sec:outliers}

In this section, we focus on the general OSC model~\eqref{classicmodel_2} (with unknown $\sigma$) and now turn to the identification of the outliers. As described in Section~\ref{descriptionMT},  this can be reformulated as a multiple testing problem (see also notation therein).

\subsection{Rescaled $p$-values}\label{settingMT2} 

As already discussed in Section~\ref{descriptionMT}, ensuring good multiple testing properties   in OSC  is challenging because the scaling parameters $\theta$ and $\sigma$ are unknown.
A natural approach  is then to use the rescaled observations $Y'_i=(Y_i-\wh{\theta})/\wh{\sigma}$, $1\leq i \leq n$, where $\wh{\theta}$, $\wh{\sigma}$ are some suitable estimators of $\theta$ and $\sigma$. To formalize further this idea, let us consider the corrected $p$-values
\begin{equation}\label{equ-pvalues}
p_i(u,s) = \ol{\Phi}\left(\frac{Y_i-u}{s}\right)  , \:\: u\in\R, \:\: s>0, \:\:1\leq i \leq n\ .
\end{equation}
The perfectly corrected $p$-values thus correspond to 
\begin{equation}\label{equ-pvaluesperfect}
p^{\star}_i = p_i(\theta,\sigma)  ,\:\:1\leq i \leq n\ .
\end{equation}
These oracle $p$-values cannot be used in practice, because they depend on the unknown  parameters $\theta$ and $\sigma$.
Our general aim is to  build estimators $\wh{\theta}$, $\wh{\sigma}$ 
 such that the theoretical performance of the corrected $p$-values $p_i(\hat{\theta},\hat{\sigma})$ mimic those of the oracle $p$-values $p^{\star}_i$, when plugged into standard multiple testing or post hoc procedures. If the use of modified $p$-values and plug-in estimators has often been advocated since the seminal work of Efron~\cite{Efron2004}, proving the convergence of the behavior of the corrected $p$-values towards the oracle one is, up to our knowledge, new. 
 The challenge is to precisely quantify how the estimation error affects the FDP/TDP metrics.
For this, a key point is the following relation between $p_i(u,s)$ and $p^{\star}_i$: 
\begin{equation}\label{fromperfecttoapprox}
\{p_i(u,s)\leq t\}= \{p_i^\star\leq U_{u,s}(t)\},\:\: i\in\{1,\dots,n\}, \:t\in[0,1]\ ,
\end{equation} 
where
\begin{equation}\label{equUu}
U_{u,s}(t)= \ol{\Phi}\left( \frac{s}{\sigma}\ol{\Phi}^{-1}\left(t\right) +\frac{u-\theta}{\sigma} \right)\ ;
\:\:\:
U^{-1}_{u,s}(v)= \ol{\Phi}\left( \frac{\sigma}{s}\ol{\Phi}^{-1}\left(v\right) +\frac{\theta-u}{s} \right)\ .
\end{equation}
Furthermore, a useful property is that the order of the $p$-values does not change after rescaling.
We will denote
\begin{equation}\label{orderpvalues}
0= p_{(0)}(u,s)\leq  p_{(1)}(u,s) \leq \dots \leq p_{(n)}(u,s)\ ,
\end{equation}
the ordered elements of $\{p_i(u,s), 1\leq i\leq n\}$. We also denote $0= p_{(0:\cH_0)}(u,s)\leq  p_{(1:\cH_0)}(u,s) \leq \dots \leq p_{(n_0:\cH_0)}(u,s)$ the ordered elements of the subset $\{p_i(u,s), i\in\cH_0\}$, that is, of the $p$-value set  corresponding to false outliers (or, equivalently, true null hypotheses).

\subsection{Upper-biased estimators}

This section provides estimators $\wt{\theta}_{+}$, $\wt{\sigma}_{+} $ that will be suitable to make the $p$-value rescaling. They are similar to the estimators introduced in Sections~\ref{sec:estitheta} and~\ref{sec:lbsigma}. However, since minimax estimation and false outliers control do not use the same risk metrics, we need to slightly modify these estimators, especially by making them upper-biased (which roughly means that the null hypotheses are favored).

For $q_n=\lfloor n^{3/4}\rfloor$ and $q'_n=\lfloor n^{1/4}\rfloor$, let us consider
\begin{equation}\label{def:estimator:forMT}
\left\{
\begin{array}{l}
\wt{\theta}_{+} =  Y_{(q_n)} + \wt{\sigma}_{+} \:\ol{\Phi}^{-1}\left(\frac{q_n}{n}\right) \ ;\\
\wt{\sigma}_{+} = \frac{Y_{(q_n)}-Y_{(q'_n)}}{\overline{\Phi}^{-1}(q'_n/(n-k_0))- \overline{\Phi}^{-1}(q_n/n)}\ ,
\end{array}
\right.
\end{equation}
for some parameter $k_0\leq \lfloor 0.9 n\rfloor$. The key difference with the estimators $\wt{\theta}_{q,q'},\wt{\sigma}_{q,q'}$ of Section~\ref{sec:randcontsigma} is the quantity $k_0$ in the denominator of $\wt{\sigma}_{+}$. 
The following result holds.
\begin{prp}\label{rem:estimator}
Consider  OSC model \eqref{classicmodel_2} with unknown variance $\sigma^2$.
Then there exists two universal positive  constants $c$, $c'$ such that the following holds for any positive integer $n$, for any $\theta\in\R$, $\sigma>0$, for any $\pi\in \overline{\cM}_{k}$ with  $k= \lfloor 0.9 n\rfloor$. Choosing $k_0$ such that $n_1(\pi)\leq k_0\leq  \lfloor 0.9 n\rfloor$ within the estimators $\wt{\theta}_{+}$, $\wt{\sigma}_{+}$ \eqref{def:estimator:forMT}, we have
\begin{align}
\P_{\theta,\pi,\sigma}(\wt{\theta}_{+}- \theta   \leq  -  \sigma n^{-1/16})  &\leq c  /n\  ;\label{inequthetachapsigmachap}\\
\P_{\theta,\pi,\sigma}(\wt{\sigma}_{+}- \sigma   \leq  -  \sigma n^{-1/16})  &\leq c  /n\ ; \label{inequthetachapsigmachapbis}\\
\P_{\theta,\pi,\sigma}\left(|\wt{\theta}_{+}- \theta|   \geq    \sigma \big( c'(k_0/n)\log^{-1/2} (n) + n^{-1/16} \big) \right)  &\leq  c /n \  ;\label{inequthetachapsigmachap2}\\ 
\P_{\theta,\pi,\sigma}\left(|\wt{\sigma}_{+}- \sigma|   \geq    \sigma \big( c'(k_0/n)\log^{-1}(n) + n^{-1/16} \big)\right)  &\leq c/n\ .
\label{inequthetachapsigmachap2bis}
\end{align}
\end{prp}

Proposition~\ref{rem:estimator} is proved in Section~\ref{p:thm:upper_robust}.  It is strongly related to Theorem~\ref{thm:upper_robust} above, although the statement is slightly different because of the introduced bias in the estimators.  
Inequalities (\ref{inequthetachapsigmachap}--\ref{inequthetachapsigmachapbis}) entail that the estimators are (with high probability) above the targeted quantity minus a polynomial term, which will be particularly suitable for obtaining a control on the false positives (FDR control and post hoc bounds). Inequalities (\ref{inequthetachapsigmachap2}--\ref{inequthetachapsigmachap2bis}) are two-sided, which is useful for studying the power of the rescaled procedures: there is an additional error term of order $(k_0/n) \log^{-a} (n)$, $a\in\{1/2,1\}$, where $k_0$ corresponds to a known upper bound of the  number of contaminated coordinates in $\pi$. 

The assumption $n_1(\pi)\leq k_0\leq  \lfloor 0.9 n\rfloor$ in Proposition~\ref{rem:estimator} is not very restrictive: it means that the number of outliers is bounded by above by some quantities $k_0$, which is used in the definition of the estimators \eqref{def:estimator:forMT}.  For instance, taking $k_0= \lfloor 0.7 n\rfloor$ means that we assume that there is no more than $70\%$ of outliers in the data, which is fair.

Finally note that our multiple testing analysis will only rely on the deviation bounds (\ref{inequthetachapsigmachap}--\ref{inequthetachapsigmachap2bis}). As a consequence, other estimators satisfying these properties can be used for scaling.  

\subsection{FDR control for selected outliers}
\label{sec:BH}

The Benjamini-Hochberg (BH) procedure is probably the most famous and widely used multiple testing procedure since its introduction in \cite{BH1995}. 
Here, the rescaled BH procedure (of nominal level $\alpha$), denoted $\BH_\alpha(u,s)$ is defined from the $p$-value family $p_i(u,s),1\leq i \leq n$, as follows:
\begin{itemize}
\item Order the $p$-values as in \eqref{orderpvalues};
\item Consider $\hat{\ell}_\alpha(u,s)=\max\{\ell\in\{0,1,\dots,n\}\::\: p_{(\ell)}(u,s)\leq \alpha \ell/n\}$;
\item Reject $H_{0,i}$ for any $i$ such that $p_i(u,s)\leq \wh{t}_{\alpha}(u,s)\}$, for $
\wh{t}_{\alpha}(u,s) = \alpha \wh{\ell}_\alpha(u,s)/n$.
 \end{itemize} 
The procedure, identified to set of the selected outliers, is then given by
\begin{equation}\label{equ:BHprocedure}
 \BH_\alpha(u,s)=\{ 1\leq i \leq n\::\: p_i(u,s)\leq \wh{t}_{\alpha}(u,s)\}\ .
 \end{equation}
The famous FDR-controlling result of \cite{BH1995,BY2001} can be re-interpreted as follows: the BH procedure using the perfectly corrected $p$-values \eqref{equ-pvaluesperfect}, that is, $\BH^\star_\alpha=\BH_\alpha(\theta,\sigma)$, satisfies
$$
\E_{\theta,\pi,\sigma}\left( \FDP(\pi, \BH^\star_\alpha) \right) = \frac{n_0}{n} \alpha \leq \alpha,\:\: \mbox{ for all $\theta,\pi,\sigma$.}
$$
This comes from the fact that the perfectly corrected $p$-values \eqref{equ-pvaluesperfect}
are independent, with uniform marginal distributions under the null hypothesis.

Recall the estimators $\wt{\theta}_{+}$ and $\wt{\sigma}_{+}$ defined in \eqref{def:estimator:forMT} with the tuning parameter $k_0$. 
The next result gives the behavior of the rescaled procedure $\BH_\alpha(\wt{\theta}_{+},\wt{\sigma}_{+})$ both in terms of FDP and TDP. 

\begin{thm} \label{th:FDPquantile}
Consider  OSC model \eqref{classicmodel_2} with unknown variance $\sigma^2$.
Then, there exists a universal positive constant $c$ such that the following holds. 
For any $\alpha\in(0,0.4)$,  $\theta\in \R$, $\sigma>0$, and any $\pi\in \overline{\cM}_{ \lfloor 0.9 n\rfloor}$ such that $n_1(\pi)\leq k_0\leq \lfloor 0.9 n\rfloor$, we have
\begin{equation}\label{FDRcontrol}
\left( \E_{\theta,\pi,\sigma}\left( \FDP(\pi, \BH_\alpha(\wt{\theta}_{+},\wt{\sigma}_{+})) \right) - \frac{n_0}{n} \alpha \right)_+ \leq c\: 
\log(n)/n^{1/16}\ .
\end{equation}
Additionally, for any sequences $\epsilon_n\in (0,1)$  tending to zero with $\epsilon_n \gg \log^{-1/2}(n)$, for any sequence $\pi=\pi_n$, $k_0=k_{0,n}$ with $n_1(\pi_n)\leq k_{0,n}\leq \lfloor 0.9 n\rfloor$ and $n_1(\pi_n)/n\asymp  k_{0,n}/n$, we have  for all $\theta,\sigma$,
\begin{equation}\label{Powercontrol}
\limsup_n\left\{ \E_{\theta,\pi,\sigma}\left( \TDP(\pi_n, \BH_\alpha^\star) \right) - \E_{\theta,\pi,\sigma}\left( \TDP(\pi_n, \BH_{\alpha(1+\epsilon_n)}(\wt{\theta}_{+},\wt{\sigma}_{+})) \right) \right\} \leq 0\ .
\end{equation}
\end{thm}
In a nutshell, inequalities \eqref{FDRcontrol} and \eqref{Powercontrol} show that the procedure $ \BH_\alpha(\wt{\theta}_{+},\wt{\sigma}_{+})$ behaves similarly to the oracle procedure  $\BH_\alpha^\star$, both in terms of false discovery rate control and power. These can be seen as a first validation of Efron's theory on empirical null distribution estimation for FDR control.

The proof of this theorem is given in Section~\ref{sec:proofth:FDPquantile}. Compared to the usual FDR proofs of the existing literature, 
there are two additional difficulties: 
first the independence assumption between the corrected $p$-values is not satisfied anymore, because the correction terms are random; second, the quantity $\FDP(\pi, \BH_\alpha(\wt{\theta}_{+},\wt{\sigma}_{+}))$ is not monotone in the estimators $\wt{\theta}_{+}$, $\wt{\sigma}_{+}$, because of the denominator in the FDP.
However, the specific properties of $\wt{\theta}_{+}$, $\wt{\sigma}_{+}$ given in Proposition~\ref{rem:estimator} will be enough to get our result: first, these estimators are biased upwards with an error term vanishing at a polynomial rate $n^{-1/16}$,  which is enough for false positive control. As for the power, we should consider the bias downwards, which is of order $(k_0/n) \log^{-a}(n)$, $a\in\{1/2,1\}$.
It turns out that the error term induced in the power is of order 
$(k_0/n) \log^{-a} (n) \log(n/n_1)$, which tends to $0$ when $ k_{0}/n\asymp n_1/n$ both in the sparse and non-sparse cases.

\subsection{Post hoc bound for selected outliers}

We now turn to the problem of finding a post hoc bound, that is, a confidence bound for $\FDP(\pi, S)$ which is valid uniformly over all $S\subset \{1,\dots,n\}$.
In \cite{GW2006,GS2010}, the authors showed that post hoc bound can be derived from a simple inequality, called the Simes inequality.
This inequality has a long history since the original work of Simes \cite{Sim1986} and is still a very active research area, see, e.g., \cite{Bodnar2017,Finner2017}.

Specifically, the following  property holds for the perfectly corrected $p$-values  $p_i^\star$:
\begin{align}\label{Simesperfect}
\P_{\theta,\pi,\sigma}\left[ \exists  \ell \in \{1,\dots,n_0\}\::\: p^\star_{(\ell:\cH_0)}\leq \alpha \ell/n\right] \leq \alpha, \:\:\alpha\in(0,1)\ ,
\end{align}
where $p^\star_{(1:\cH_0)}\leq \dots\leq p^\star_{(n_0:\cH_0)}$ are the ordered elements of $\{p_i^\star, i\in\cH_0\}$.

When replacing the perfect $p$-values by the estimated ones, the next result shows that Simes inequality is approximately valid.

\begin{thm}\label{th:simescorrected}
Consider  OSC model \eqref{classicmodel_2} with unknown variance $\sigma^2$. 
Then, there exists a universal positive constant $c$ such that the following holds. 
For any $\alpha\in(0,0.4)$,  $\theta\in \R$, $\sigma>0$, and any $\pi\in \overline{\cM}_{ \lfloor 0.9 n\rfloor}$ such that $n_1(\pi)\leq k_0\leq \lfloor 0.9 n\rfloor$, we have
\begin{align}\label{Simescorrected}
\left(\P_{\theta,\pi,\sigma}\left[ \exists  \ell \in \{1,\dots,n_0\}\::\: p_{(\ell:\cH_0)}(\wt{\theta}_{+},\wt{\sigma}_{+})\leq \alpha \ell/n\right] - \alpha \right)_+ \leq c\: \log(n)/n^{1/16}\ . 
\end{align}
\end{thm}
The proof  is given in Section~\ref{p:th:simescorrected}. It uses that $\wt{\theta}_{+}$, $\wt{\sigma}_{+}$ are bias upwards thanks to Proposition~\ref{rem:estimator}, together with the monotonicity of the criterion.
Let us now define the following data-driven quantity
\begin{align}\label{posthocbound}
\ol{\FDP}(S; u,s)= 1\wedge\left\{\min_{\ell\in\{1,\dots,n\}}\left(\sum_{i\in S} \mathds{1}\{p_i(u,s) >\alpha \l/n\}+\l-1\right)/(|S|\vee 1)\right\}, \:\: S\subset\{1,\dots,n\}\ .
\end{align}
The next result shows that \eqref{posthocbound} is an upper-bound for the FDP, uniformly valid over all the possible selection sets $S$.

\begin{cor}\label{cor:Simescorrected}
There exists a numerical constant $c$ such that the following holds. 
Under the conditions of Theorem~\ref{th:simescorrected}, the bound $\ol{\FDP}(\cdot; \wt{\theta}_{+}, \wt{\sigma}_{+})$ defined by \eqref{posthocbound} 
satisfies the following property: 
\begin{align}\label{corposthocbound}
\left(1-\alpha- \P_{\theta,\pi,\sigma}\left(\forall S\subset \{1,\dots,n\},\:\: \FDP(\pi,S)\leq \ol{\FDP}(S; \wt{\theta}_{+},\wt{\sigma}_{+})\right)\right)_+  \leq c\:
 \log(n)/n^{1/16}\ .
\end{align}
\end{cor}
The proof is standard from \cite{GW2006,GS2010} and is given in Section~\ref{p:cor:Simescorrected} for completeness.
From Corollary \ref{cor:Simescorrected}, we deduce, for any selection procedure $\wh{S}$ possibly depending on the data  in an arbitrary way, that the  quantity $\ol{\FDP}(\wh{S}; \wt{\theta}_{+},\wt{\sigma}_{+})$ is a valid confidence bound for  $\FDP(\pi,\wh{S})$. Hence, this provides a statistical guarantee on the proportion of false outliers in  any data-driven set.

\subsection{Application to decorrelation}\label{sec:equicor}

In this section, we consider a multiple testing issue in the Gaussian equi-correlated model, that corresponds to observe
\begin{equation}\label{model-equi}
Y \sim \mathcal{N}\left(m,\Gamma\right), \:\:\: m\in \R^n,\:\:\: \Gamma=\left(\begin{array}{cccc}1&\rho&\dots&\rho\\\rho&1&\ddots&\vdots\\\vdots&\ddots&\ddots&\rho\\\rho&\dots&\rho&1\end{array}\right)\ ,
\end{equation}
for some unknown $\rho\in[0,1)$ and some unknown $m_i$, $1\leq i \leq n$. 
The probability (resp. expectation) in this model is denoted $\P_{m,\rho}$ (resp. $\E_{m,\rho}$). 
We consider the problem of testing simultaneously: 
\begin{equation}\label{hypomodel}
\mbox{$H_{0,i}: ``m_i=0"$ against $H_{1,i}: ``m_i>0"$, for all $1\leq i \leq n$.}
\end{equation}
Classically, this model can be rewritten as follows
\begin{equation}\label{equirealization}
Y_i = m_i+\rho^{1/2} W+ (1-\rho)^{1/2}\zeta_i, \:\:\: 1\leq i\leq n,
\end{equation}
for $W$, $\zeta_i$, $1\leq i\leq n$,  i.i.d. $\mathcal{N}(0,1)$.
This model shares strong similarities with the gOSC model (and therefore also OSC model), because, conditionally on $W$, the $Y_i$'s  follows  model \eqref{model}-\eqref{decompgamma} with $\gamma_i=m_i+\rho^{1/2} W$, $\theta=\rho^{1/2} W$, $\mu_i=m_i$, $\xi_i=\zeta_i$ and $\sigma^2=1-\rho>0$.
In particular, the multiple testing problem \eqref{hypomodel} is the same as \eqref{hypomu} and we can define $\cH_0(m)$, $n_0(m)$, $\cH_1(m)$, $n_1(m)$, $\FDP(m,\cdot)$ and $\TDP(m,\cdot)$ accordingly, see Section~\ref{settingMT}.
Whereas the classical $p$-values  $p_i=\ol{\Phi}(Y_i)$, $1\leq i\leq n$ lead to some pessimistic behaviour, we can use the empirically re-scaled $p$-values  to get the following result.

\begin{cor}\label{prp:FDP}
Consider the model \eqref{model-equi}. There exists a universal positive constant $c$ such that the following holds.
For any $\alpha \in (0,0.4)$, any $\rho\in (0,1]$, any mean $m$ satisfying $n_1(m)\leq k_0\leq  0.9 n$ (that is, $m$ has at most $0.9n$ non-zero coordinates), we have
\begin{itemize}
\item[(i)] the procedure $\BH_\alpha(\wt{\theta}_{+},\wt{\sigma}_{+})$ defined in Section~\ref{sec:BH} satisfies  
$$
\left( \E_{m,\rho}\left( \FDP(m, \BH_\alpha(\wt{\theta}_{+}, \wt{\sigma}_{+}) )\right) - \frac{n_0}{n} \alpha \right)_+ \leq c \log (n)/n^{1/16} \  .
$$
Moreover, for any sequence $\epsilon_n\in (0,1)$  tending to zero with $\epsilon_n \gg \log^{-1/2} (n)$,  for any sequences $m_n$ and $k_0=k_{0,n}$ with $n_1(m_n)\leq k_{0,n}$ and $n_1(m_n)/n\asymp  k_{0,n}/n$, we have 
$$ 
\limsup_n\left\{ \E'\left( \TDP(m_n, \BH_\alpha)\right)-\E_{m_n,\rho}\left( \TDP(m_n, \BH_{\alpha(1+\epsilon_n)}(\wt{\theta}_{+}, \wt{\sigma}_{+}) \right)   \right\} \leq 0\ ,
$$
where $\E'$ denotes the expectation in the model \eqref{model-equi} in which $\rho'=0$ (independence) and $m'_n=m_n (1-\rho)^{-1/2}$.
\item[(ii)]  the bound $\ol{\FDP}(\cdot; \wt{\theta}_{+}, \wt{\sigma}_{+})$ defined by \eqref{posthocbound} satisfies  
\begin{align}\label{propposthocboundequi}
\left(1-\alpha- \P_{m,\rho}\left(\forall S\subset \{1,\dots,n\},\:\: \FDP(m,S)\leq \ol{\FDP}(S; (\wt{\theta}_{+}, \wt{\sigma}_{+})\right)\right)_+  \leq 
 c\log( n)/n^{1/16}\ .
\end{align}
\end{itemize}
\end{cor}

This result is a direct consequence of Theorems~\ref{th:FDPquantile} and~\ref{th:simescorrected} by integrating w.r.t.~$W$, so the proof is omitted.

In a nutshell, these results indicate that the analysis of the FDR under independence can be extendeds in the one-sided equi-correlated model, with an additional improvement due to variance reduction by a factor $(1-\rho)$, which can be significant under strong dependence.
The condition $n_1(\pi)\leq k_0\leq \lfloor 0.9 n\rfloor$ is not very restrictive as we can always choose $k_0=\lfloor 0.9 n\rfloor$. However, this choice leads to a conservative estimator $\wt{\sigma}_{+}$ and it is obviously better to choose $k_0$ as close as possible to $n_1$, the number of non zero coordinates of $m$. The power result indicates that choosing $k_{0}/n\asymp  n_1/n$ is enough from an asymptotic point of view.

Compared to the state of the art, and especially the work aiming at correcting the dependencies by estimating factors \cite{FKC2009,LS2008,Fan2012,Fan2017}, 
this result is to our knowledge the first one that shows that the corrected procedures is rigorously controlling the desired multiple testing criteria, with some power optimality. In particular, our result shows that the sparsity is not required to make a rigorous dependency correction: the correction is also theoretically valid when $n_1=\lfloor 3n/4\rfloor$ for instance. This is due to the one-sided structure of our model and would certainly be not true in the two-sided case.
Overall, while our setting is admittedly somewhat narrowed, we argue that this is an important "proof of concept" that supports previous studies and may pave the way for future developments in the area of dependence correction via principal factor approximation.

\subsection{Numerical experiments}

In this section, we illustrate Corollary~\ref{prp:FDP} with numerical experiments. We consider the equi-correlation model \eqref{model-equi}, with  a mean $m$ taken of the form
$
m_i = \Delta ,
$ 
$1\leq i \leq n_1$ and $m_i=0$ otherwise. 
The results of our experiments are qualitatively the same for other kinds of alternative means, see Section~\ref{supp:simu}.

In our simulations, we consider four types of rescaled $p$-values $p_i(u,s) $, $1\leq i \leq n$, see \eqref{equ-pvalues}:
\begin{itemize}
\item Uncorrected: $u=0$, $s=1$;
\item Oracle: $u=\theta=\rho^{1/2} W$ and $s=\sigma=(1-\rho)^{1/2}$;
\item Correlation $\rho$ known: 
\begin{equation}\label{def:estimator:forMTsemi}
u=\wt{\theta}_{+}(\sigma) =  Y_{(q_n)} + \sigma \:\ol{\Phi}^{-1}\left(\frac{q_n}{n}\right), \:\:q_n=\lfloor n^{3/4}\rfloor\ ,
\end{equation}
and $s=\sigma=(1-\rho)^{1/2}$;
\item Correlation $\rho$ unknown: $u=\wt{\theta}_{+}$ and $s=\wt{\sigma}_{+}$  for the estimators defined by \eqref{def:estimator:forMT} with $k_0=n_1(m)$ (a value of $k_0$ avoiding a too biased estimation of $\sigma$).
\end{itemize}
Each of these rescaled $p$-values are used either for FDR control via the Benjamini-Hochberg procedure  $\BH_\alpha(u,s)$ \eqref{equ:BHprocedure} (Figure~\ref{fig:boxplot}),  or to get post hoc bound via the Simes bound $\ol{\FDP}(\cdot; u,s)$ \eqref{posthocbound} (Figure~\ref{fig:equi:posthoc}). 

\paragraph{New corrected BH procedure}

Figure~\ref{fig:boxplot} displays the performances of the rescaled BH procedure. 
As we can see, the result of this experiment corroborates Proposition~\ref{prp:FDP}: while the FDR control is maintained, the power of the new procedure is greatly improved. More importantly, estimating the variable $W$ substantially stabilizes the picture and make the FDP/TDP much less variant. This figure also allows to feel the price to pay for estimating $\rho$, as the procedure using the true value of $\rho$ is closer to the oracle and less variant. 

\begin{figure}[h!]
\begin{tabular}{ccc}
&FDP& TDP \\
\rotatebox{90}{\hspace{2cm}$\Delta=2$}&
\includegraphics[scale=0.27]{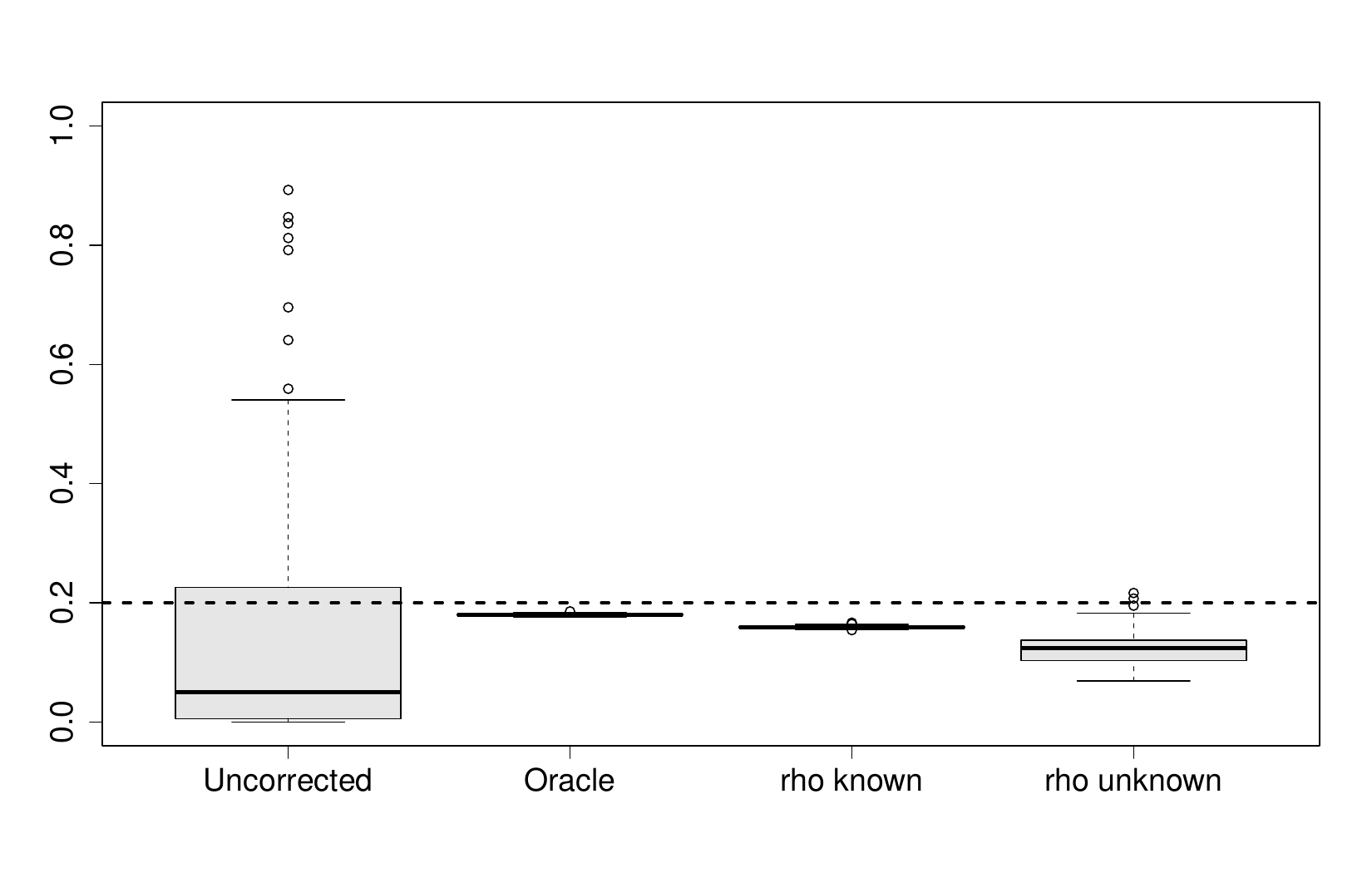}&
\includegraphics[scale=0.27]{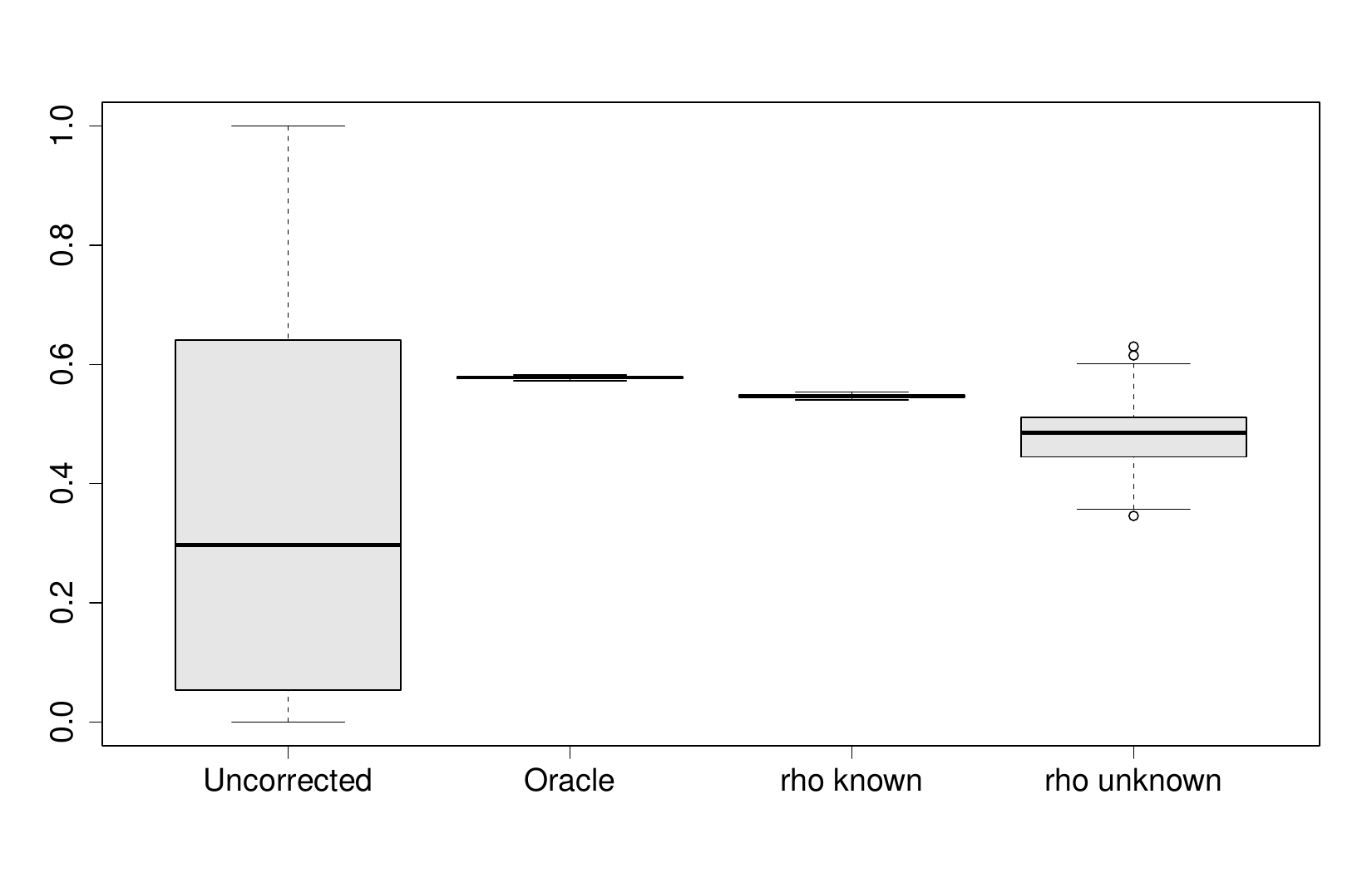}\\
\rotatebox{90}{\hspace{2cm}$\Delta=3$}&
\includegraphics[scale=0.27]{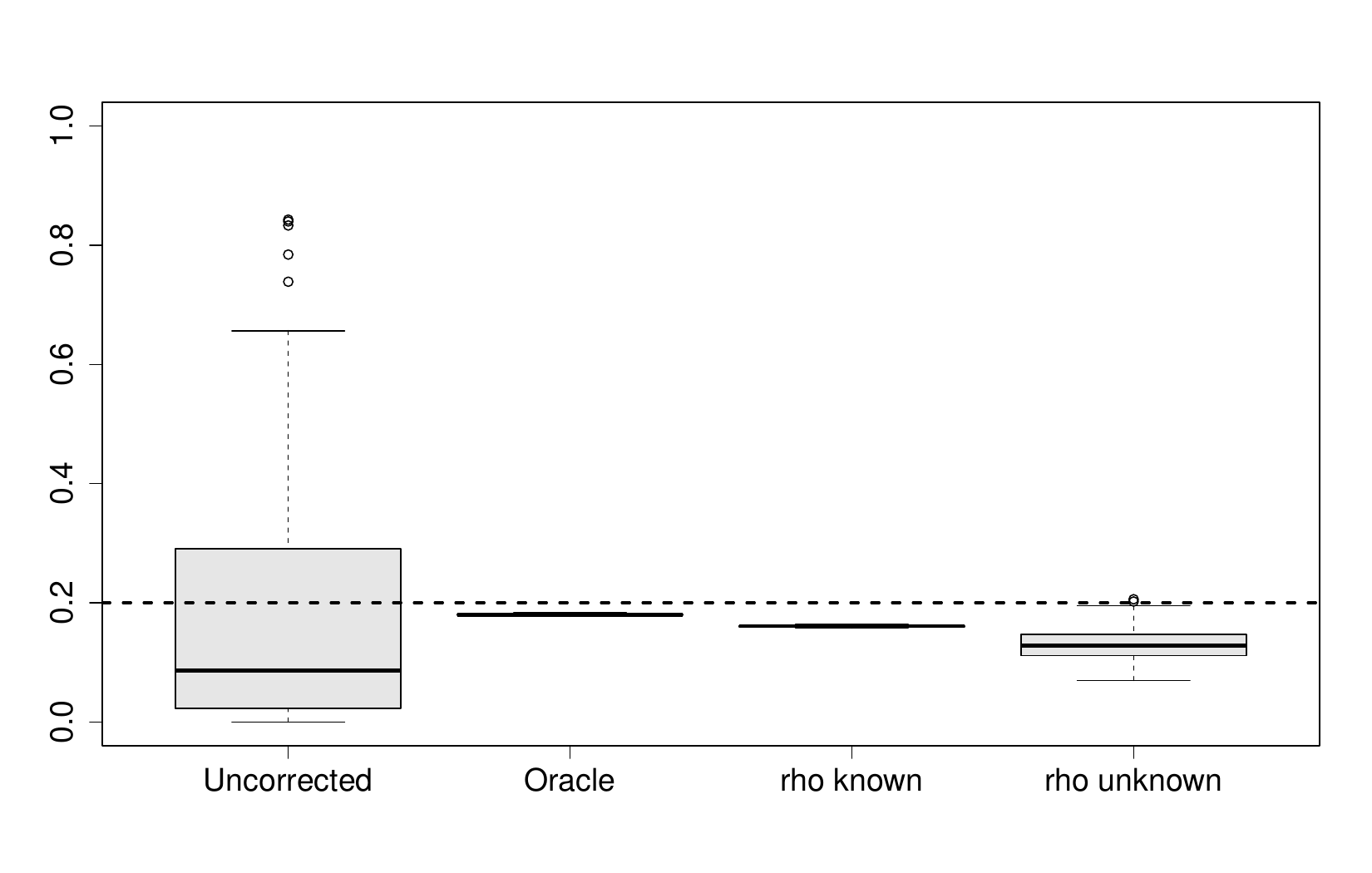}&
\includegraphics[scale=0.27]{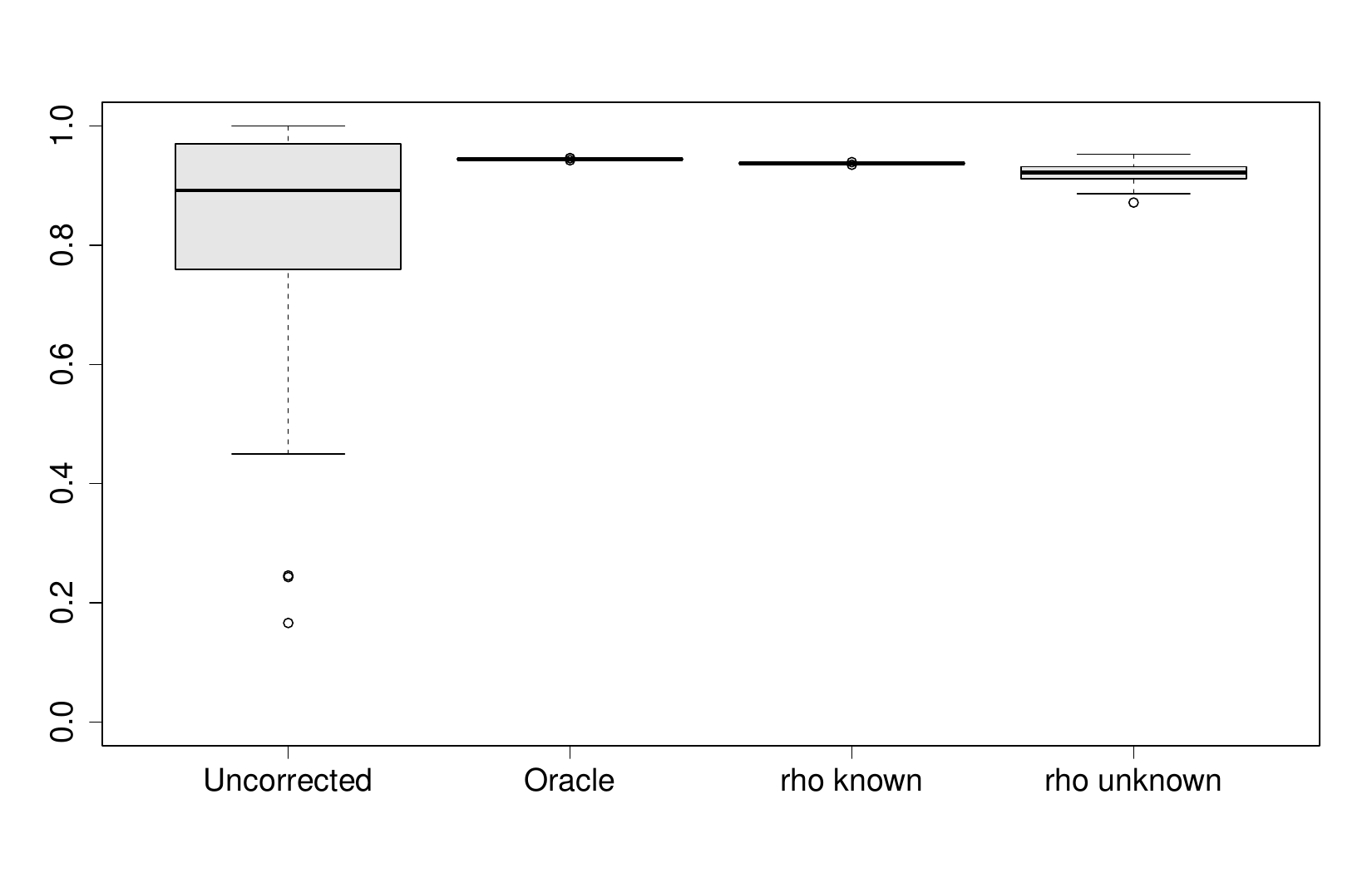}
\end{tabular}
\caption{Boxplot of the FDP and TDP of the BH procedure, with different type of $p$-value rescaling, see text. 
The values of the parameters are $n=10^6$, $\rho=0.3$, $k/n=0.1$, $\alpha=0.2$, $100$ replications are done to evaluate the FDP/TDP.} \label{fig:boxplot}
\end{figure}

\paragraph{Post selection bound}

Here, we evaluate the quality of the rescaled Simes post hoc bounds. Since these bounds are meant to be uniform over all the possible selection sets $S$, there is no obvious choice for the set $S$ on which these bounds should be computed. A possible choice, in line with the recent work \cite{GMKS2016}, is to choose a ``typical" family of sets $(S_t)_t$ and to look at the quality of the so-derived confidence envelope  $t\mapsto \ol{\FDP}(S_t)$ for $t\mapsto {\FDP}(S_t)$.  Here, the subset family $\{S_t, 1\leq t\leq 200 \}$ is built as follows: each $S_t$ is composed of the indices of the $t$ largest values of $\{m_i, 1\leq i \leq n\}$. Hence, we have simply here $S_t=\{ 1,\dots, t\}$ and ${\FDP}(S_t)=(t- n_1)_+/t$.

Figure~\ref{fig:equi:posthoc} reports the values of the obtained confidence envelopes, for the four types of $p$-value rescaling described above. Each time, the confidence envelope should be above the true value of the FDP (bold red line), with high probability, and the closer the bound to this quantity, the sharper the bound.
The conclusion is similar to the FDP/TDP case.  We see less variability for the corrected bounds, especially around the ``neck" of the curves ($s\approx n_1=50$), which is a point of interest.
 
\begin{figure}[h!]
\begin{tabular}{cc}
\vspace{-5mm}
Uncorrected
& 
Oracle
\\
\includegraphics[scale=0.4]{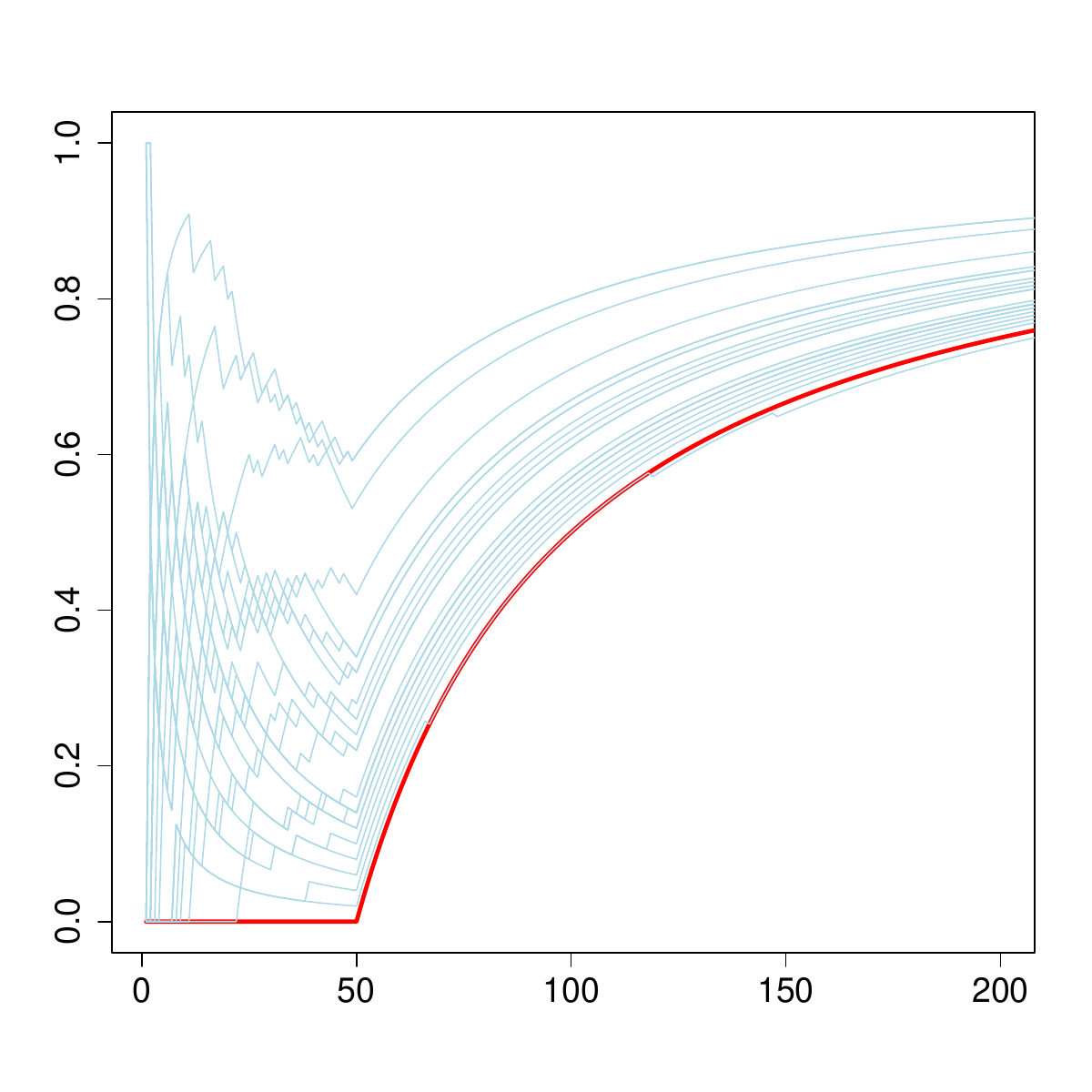}&\includegraphics[scale=0.4]{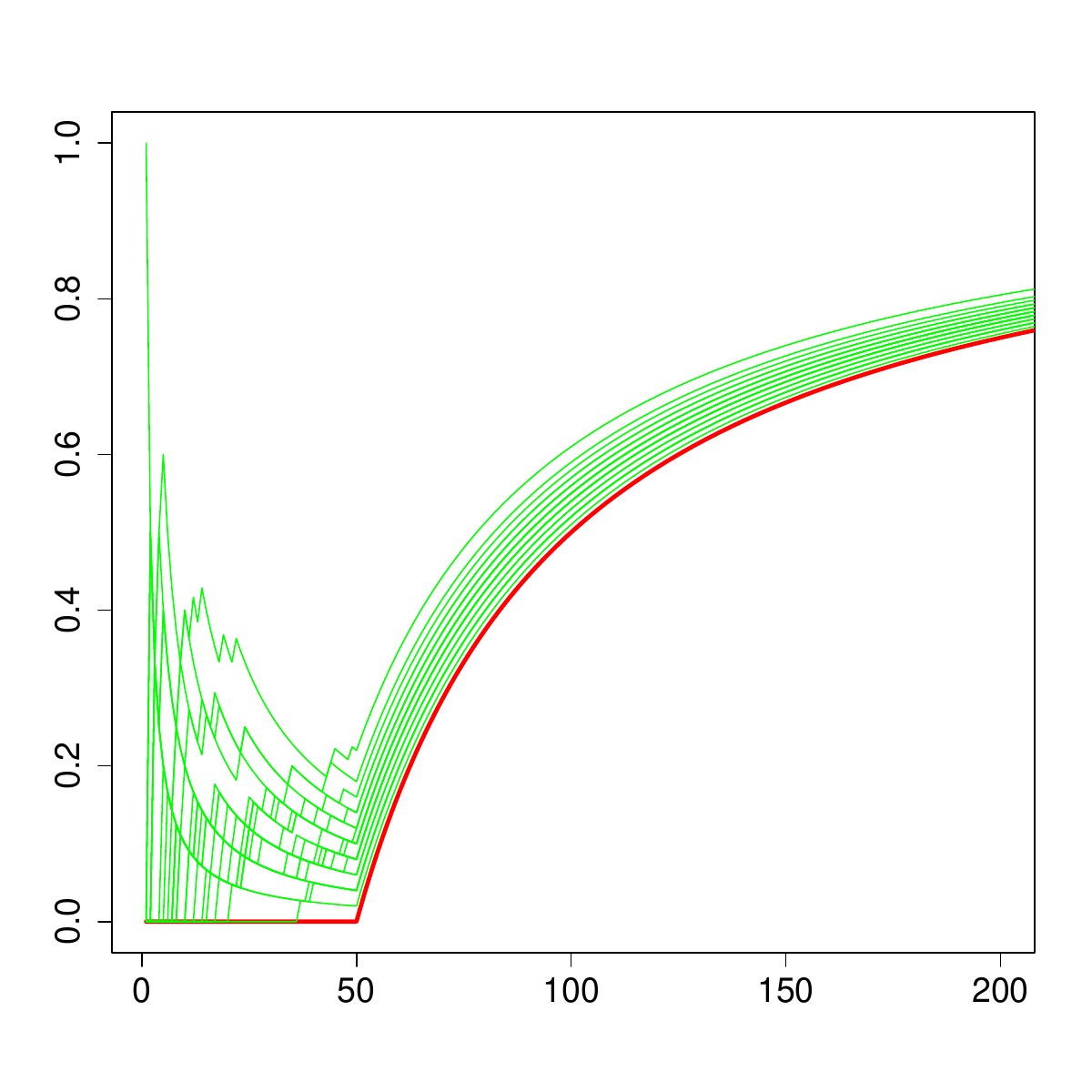}\\
\vspace{-5mm}
Correlation known
&
Correlation unknown
\\
\includegraphics[scale=0.4]{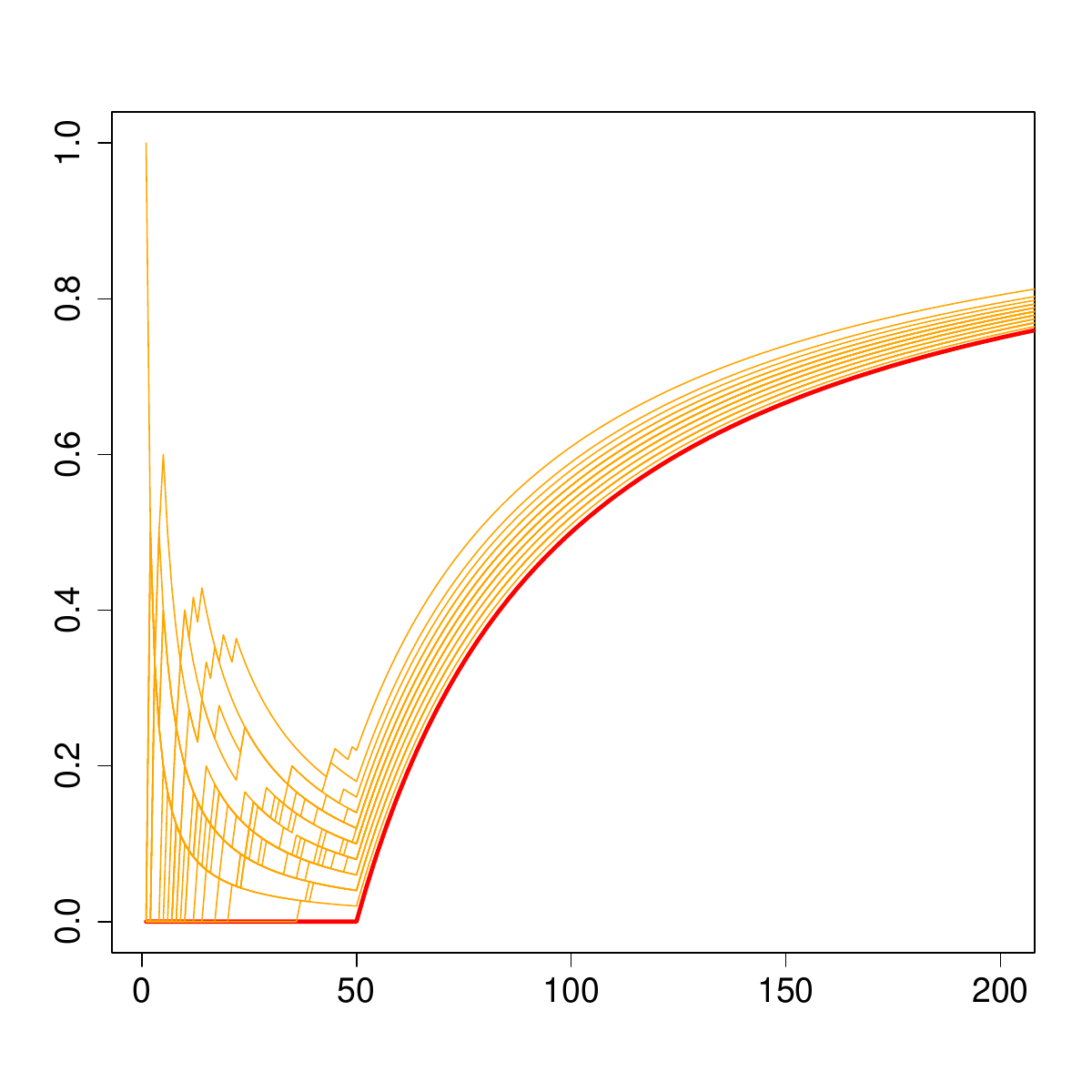}&\includegraphics[scale=0.4]{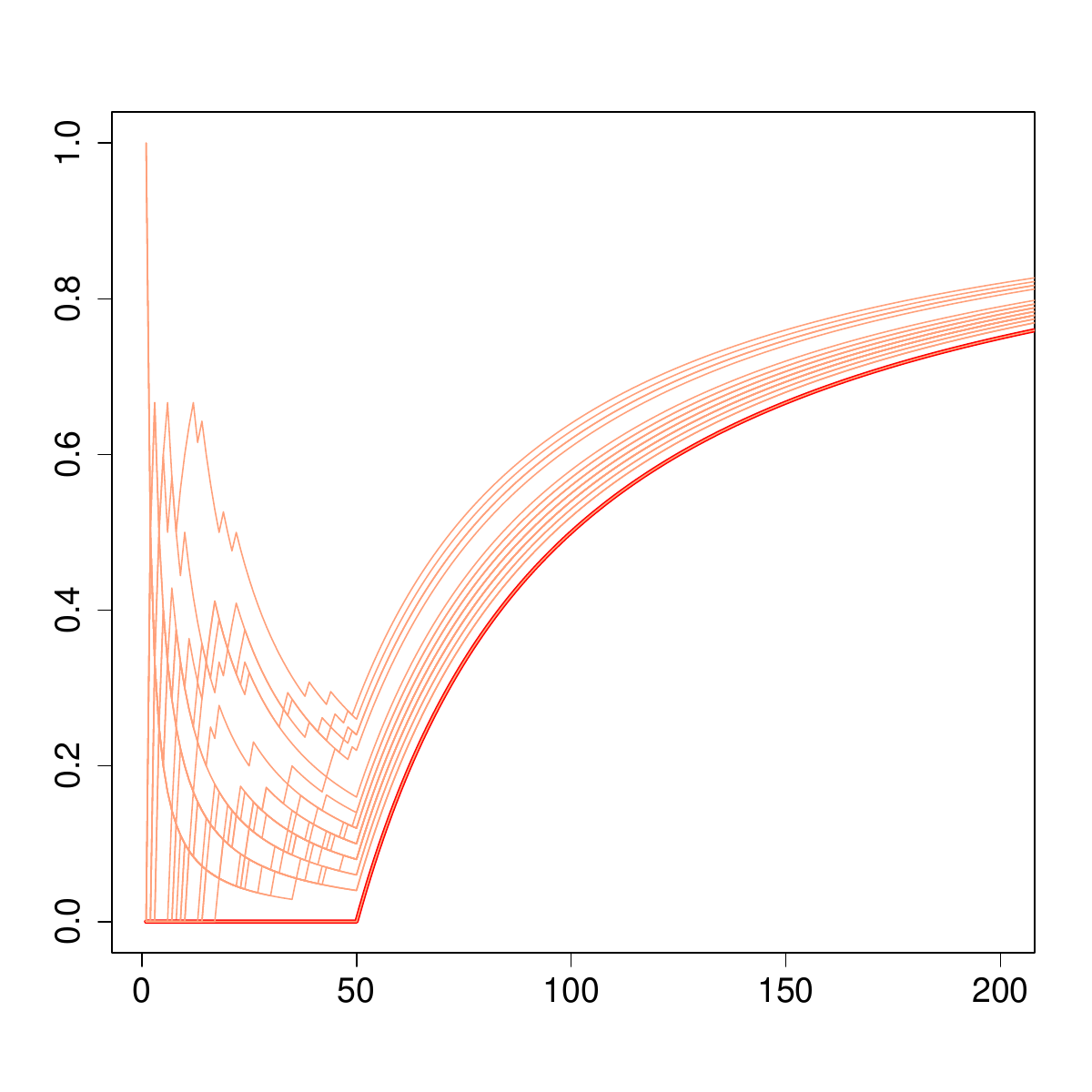}
\end{tabular}
\caption{
Plot of $25$ trajectories of different post hoc bounds $t\mapsto \ol{\FDP}(S_t; \cdot,\cdot)$ for $S_t$ corresponding to the top-$t$ largest means (see \eqref{posthocbound} for the definition of $\ol{FDP}$ and the text for the exact definition of $S_t$). From top-left to bottom-right: uncorrected bound $t\mapsto\ol{\FDP}(S_t; 0,1)$ ; oracle bound $t\mapsto\ol{\FDP}(S_t; \theta,\sigma)$; bound with $\sigma=(1-\rho)^{1/2}$ known $t\mapsto\ol{\FDP}(S_t; \widetilde{\theta}_+(\sigma),\sigma)$; bound with $\sigma$ unknown $t\mapsto\ol{\FDP}(S_t; \widetilde{\theta}_+,\widetilde{\sigma}_+$).
The estimators $\widetilde{\theta}_+,\widetilde{\sigma}_+$ (resp. the estimator $\wt{\theta}_{+}(\sigma) $)  are given by \eqref{def:estimator:forMT} (resp. \ref{def:estimator:forMTsemi}).
In each of these pictures, the unknown value of  $t\mapsto {\FDP}(S_t)$ is displayed with the red bold line.
  The values of the parameters are $n=10^3$, $\rho=0.3$, $k/n=0.05$, $\alpha=0.2$ and $\Delta=4$.
  }\label{fig:equi:posthoc}
\end{figure}

\paragraph{Acknowledgements.} The work of A. Carpentier is partially supported by the Deutsche Forschungsgemeinschaft (DFG) Emmy Noether grant MuSyAD (CA 1488/1-1), by the DFG - 314838170, GRK 2297 MathCoRe, by the DFG GRK 2433 DAEDALUS, by the DFG CRC 1294 'Data Assimilation', Project A03, and by the UFA-DFH through the French-German Doktorandenkolleg CDFA 01-18.
This work has also been supported by ANR-16-CE40-0019 (SansSouci) and ANR-17-CE40-0001 (BASICS).

 \appendix
 
 \section{Proofs of minimax lower bounds}

\subsection{Proof of Theorem \ref{thm:lower_one_sided} (gOSC)}\label{p:thm:lower_one_sided}

\subsubsection{Extreme cases}

First,  for $k=0$, estimating $\theta$ amounts to estimating the mean of a Gaussian random variable based on $n$ observation. For this problem, the minimax risk is widely known to be of order $1/\sqrt{n}$ (see standard statistical textbooks such as \cite{MR2724359}). Since $\cR(k,n)$ is nondecreasing with respect to $k$ if follows that, for all integers $k$,  $\cR(k,n)\geq cn^{-1/2}$ for some universal constant $c>0$. 
For $k'= n-\lfloor n^{1/4}\rfloor$, we shall prove prove that $\cR[k',n]\geq c' \sqrt{\log(n)}$, which in turn implies $\cR[k,n]\geq c' \sqrt{\log(n)}$ for all $k\geq k'$.
\begin{lem}\label{lem:ultra_dense}
There exists a constant $c'>0$ such that for $k'= n-\lfloor n^{1/4}\rfloor$, we have  $\cR[k',n]\geq c' \sqrt{\log(n)}$. 
\end{lem}
This lower bound straightforwardly follows from a reduction of the  estimation problem to a detection problem for which minimax separation distance have already been derived.
The proof will also serve as a warm-up for the more challenging case  $k\in [\sqrt{n}, n - n^{1/4}]$.

\begin{proof}[Proof of Lemma \ref{lem:ultra_dense}]

To prove this lemma we reduce the problem of estimating $\theta$ to a signal detection problem. Write $\underline{k}= n-k'=\lfloor n^{1/4}\rfloor$ for short. Let $a\in (0,1)$ be a constant that will be fixed later. Given any $k$, denote  $\cP[k,n]$, the collection of subset of $\{1,\ldots,n\}$ of size $k$. Define $\theta_0= -a \sqrt{\log(n)}$ and for any $S\in \cP[\underline{k},n]$ let $\mu_S$ denote the vector such that $(\mu_S)_i= |\theta_0|$ if $i\notin S$ and 0 otherwise. Note that 
$\mu_S$ is $k'$-sparse, that is, $\mu_S\in \mathcal{M}_{k'}$.
It follows from the definition of the minimax risk
\beqn 
 \cR[k',n]&\geq &\frac{1}{2} \inf_{\wh{\theta}} \left[\E_{0,0}[|\wh{\theta}|]+ \max_{S\in \cP[\underline{k},n]} \E_{\theta_0,\mu_S}[|\wh{\theta}-\theta_0|]  \right]\\
 &\geq&\frac{a}{4}\sqrt{\log(n)}  \Big[   \P_{0,0}[\wh{T}=0]+ \max_{S\in \cP[\underline{k},n]} \P_{\theta_0,\mu_S}[\wh{T}=1\big]\Big].
 \eeqn
by letting $\wh{T}=\ind_{\wh{\theta}\geq \theta_0/2}$ and by using $(|\theta_0|/2) \ind_{\wh{T}=0} \leq |\wh{\theta}|$ and 
$(|\theta_0|/2) \ind_{\wh{T}=1} \leq |\wh{\theta}-\theta_0|$. Thus
\beqn  
  \cR[k',n] &\geq & \frac{a}{4}\sqrt{\log(n)} \inf_{\wh{T}} \Big[   \P_{0,0}[\wh{T}=0]+ \max_{S\in \cP[\underline{k},n]} \P_{\theta_0,\mu_S}[\wh{T}=1\big]\Big].
 \eeqn
As a consequence, if $a$ is chosen small enough such that no test is able to decipher reliably between $\P_{0,0}$ and $\{\P_{\theta_0,\mu_S},\ S\in \cP[\underline{k},n]\}$, then the minimax risk is of the order of $a\sqrt{\log(n)}$. Note that this problem amounts to testing in a simple Gaussian white noise model $\mathcal{N}(\gamma,I_n)$ whether the mean vector $\gamma$ is zero or if $\gamma=\theta_0+\mu_S$ is $\underline{k}$-sparse with negative non-zero values that are all equal to $-a\sqrt{\log(n)}$. Quantifying the difficulty of this problem is classical in the statistical literature and has be done for instance in \cite{baraud02} (for non-necessarily positive values). For the sake of completeness, we shall provide exhaustive arguments. Denote $\overline{\P}_{\underline{k}}= |\cP[\underline{k},n]|^{-1}\sum_{S}\P_{\theta_0,\mu_S}$ the mixture measure when $S$ is sampled uniformly in $\cP[\underline{k},n]$. Since the supremum is larger than the mean, we obtain
 \begin{eqnarray}\nonumber 
  \cR[k',n]&\geq & \frac{a}{4}\sqrt{\log(n)} \inf_{\wh{T}} \big[   \P_{0,0}[\wh{T}=0]+  \overline{\P}_{\underline{k}}[\wh{T}=1]\big] \nonumber \\
 &\geq & \frac{a}{4}\sqrt{\log(n)}(1 -  \| \P_{0,0}-  \overline{\P}_{\underline{k}}\|_{TV})\geq \frac{a}{4}\sqrt{\log(n)} \left(1- \sqrt{\chi^2( \overline{\P}_{\underline{k}}; \P_{0,0}) }\right) \ ,\label{eq:lower_ultra_dense}
\end{eqnarray}
since the $\chi^2$ discrepancy between distributions dominates the square of the total variation distance $\|.\|_{TV}$, see Section~2.4 in \cite{MR2724359}. Writing $L$ the likelihood ratio of $\overline{\P}_{\underline{k}}$ over $\P_{0,0}$ and, for $S\in \cP[\underline{k},n]$, $L_S$ the likelihood ratio of $\P_{\theta_0,\mu_S}$ over $\P_{0,0}$, and $\pi$ the uniform measure over $\cP[\underline{k},n]$, we have
\beqn 
1+ {\chi^2( \overline{\P}_{\underline{k}}; \P_{0,0})} &=& \E_{0,0}[L^2]= \E_{0,0}[\pi^{\otimes 2}( L_S L_{S'})]\\
& =& \pi^{\otimes 2}\Big[\E_{0,0}( L_S L_{S'})\Big]\\
& =& \pi^{\otimes 2}\Big[e^{\theta_0^2 |S\cap S'|}\Big]\ , 
\eeqn 
where the last equation follows from simple computations for normal distribution. Note that $|S\cap S'|$ is distributed as  an hypergeometric random variable $Z$ with parameter $(\underline{k},\underline{k},n)$. It has been observed in~\cite{aldous85} (see the proof of Proposition~20.6 therein) that there exists $\sigma$-field $\cB$ and a random variable $W$ following a Binomial distribution with parameter $(\underline{k},\underline{k}/n)$ such that $Z= \E[W|\cB]$. Then, it follows from Jensen inequality, that
\beqn
 1+ {\chi^2( \overline{\P}_{\underline{k}}; \P_{0,0})}&\leq&\E[\exp[\theta_0^2 W]]= \left[1+\frac{\underline{k}}{n}\left(e^{\theta_0^2}-1\right)\right]^{\underline{k}}\\
 &\leq & \exp\left[\frac{\underline{k}^2}{n}e^{\theta_0^2}\right]\leq {\exp\left[n^{-1/2+a^2}\right]} \ , 
\eeqn
by definition of $\theta_0$ and $\underline{k}$. Fixing {$a=1/2$}, 
and coming back to \eqref{eq:lower_ultra_dense}, we obtain 
\[
  \cR[k',n]\geq   c'\sqrt{\log(n)}\left(1- \sqrt{\exp(n^{-1/4}) - 1}\right) \ , 
\]
which is larger than some $c\sqrt{\log(n)}$ for $n\geq 5$. The result is also valid for $n< 5$ by considering  a constant $c$ small enough.

\end{proof}

We now turn to the case $k\in [\sqrt{n}, n - n^{1/4}]$. As in the previous proof, we shall first reduce the problem to a two point testing problem and then compute an upper bound of the total variation distance. However, the reasoning here is much more involved. 

\subsubsection{Step $1$: Two point reduction}

Given any two distributions $\nu_0$ and $\nu_1$ on $\mathbb{R}_+$ and any $\theta_0<\theta_1$ in $\mathbb{R}$, we denote, for $i=0,1$ $\mathbf{P}_i= \int \P_{\theta_i,\mu}\nu_i^{\otimes n}(d\mu)$ the mixture distribution where the components of $\mu$ are i.i.d. sampled according to the distribution $\nu_i$. We start with the following general reduction lemma.

\begin{lem}[Reduction]\label{lem:reduction_general}
For any $\theta_0$, $\theta_1$, $\nu_0$ and $\nu_1$, we have 
\beq\label{eq:lower_reduction_general}
\cR[k,n]\geq \frac{|\theta_0-\theta_1|}{4} \left[ 1- \|\mathbf{P}_0- \mathbf{P}_1\|_{TV}  - \sum_{i=0,1} \nu_i^{\otimes n}(\mu\notin \cM_k) \right]\ . 
\eeq
\end{lem}

\begin{proof}[Proof of Lemma \ref{lem:reduction_general}]
 For short, we write $\cB_k=\{\|\mu\|_0> k\}$. Starting from the definition of the minimax risk, we have  
\beqn 
\cR[k,n]&\geq &\frac{1}{2} \inf_{\wh{\theta}}\left( \sup_{\mu\in \cM_k} \E_{\theta_0,\mu}[|\wh{\theta}-  \theta_0|]+ \sup_{\mu\in \cM_k} \E_{\theta_1,\mu}[|\wh{\theta}-  \theta_1|]\right)\\
&\geq & \frac{|\theta_0-\theta_1|}{4} \inf_{\wh{\theta}}\left( \sup_{\mu\in \cM_k} \P_{\theta_0,\mu}\left[\wh{\theta}> \frac{\theta_0 + \theta_1}{2}\right]+  \sup_{\mu\in \cM_k} \P_{\theta_1,\mu}\left[\wh{\theta}\leq \frac{\theta_0 + \theta_1}{2}\right]\right)\\
&\stackrel{(i)}{\geq} & \frac{|\theta_0-\theta_1|}{4} \inf_{\wh{\theta}}\left[\mathbf{P}_0\left(\wh{\theta}> \frac{\theta_0 + \theta_1}{2}\right) - \nu_0^{\otimes n}(\cB_k) + \mathbf{P}_1\left(\wh{\theta}\leq  \frac{\theta_0 + \theta_1}{2}\right)  - \nu_1^{\otimes n}(\cB_k) \right]\\
&\geq & \frac{|\theta_0-\theta_1|}{4} \left[ 1-  \sup_{\cA}|\mathbf{P}_0(\cA)- \mathbf{P}_1(\cA)|- \nu_0^{\otimes n}(\cB_k)  - \nu_1^{\otimes n}(\cB_k) \right]\\
&= & \frac{|\theta_0-\theta_1|}{4} \left[ 1- \|\mathbf{P}_0- \mathbf{P}_1\|_{TV}- \nu_0^{\otimes n}(\cB_k)  - \nu_1^{\otimes n}(\cB_k) \right]\ , 
\eeqn 
 where we use in ($i$) that both $\mathbf{P}_0\left[\wh{\theta}> \frac{\theta_0 + \theta_1}{2}\right] \leq  \nu_0^{\otimes n}(\cB_k) 
 + \sup_{\mu\in \cM_k} \P_{\theta_0,\mu}\left[\wh{\theta}> \frac{\theta_0 + \theta_1}{2}\right]$ and $\mathbf{P}_1\left[\wh{\theta}\leq \frac{\theta_0 + \theta_1}{2}\right] \leq \nu_0^{\otimes n}(\cB_k)  +\sup_{\mu\in \cM_k} \P_{\theta_1,\mu}\left[\wh{\theta}\leq \frac{\theta_0 + \theta_1}{2}\right] $. 

\end{proof}

As a consequence, we will  carefully choose the $\theta_i$'s and $\nu_i$'s in such a way that $|\theta_0-\theta_1|$ is the largest possible while the total variation distance $\|\mathbf{P}_0- \mathbf{P}_1\|_{TV}$ is bounded away from one and the measures $\nu_i^{\otimes n}$ are concentrated on $\cM_k$. To this end, we consider in the sequel some measures with an extra mass on $0$, but in a light fashion so that the convergence rate is preserved.

In the sequel, we define $k_0= \max(k/2, n- 3(n-k)/2)<k$ and we assume 
\beq\label{eq:assumption_nu_souligne}
\nu_i(\{0\})\geq  1 - \frac{k_0}{n}\ , \text{ for }i=0,1.
\eeq
As a consequence, under $\nu_i^{\otimes n}$, $\|\mu\|_0$ is stochastically dominated by a Binomial distribution with parameters $(n,k_0/n)$. By Chebychev inequality, we have 
\[
 \nu_i^{\otimes n}(\mu \notin \cM_k)\leq  \frac{k_0 (n-k_0)}{n\min^2[k/2, (n-k)/2]}\leq \frac{2\min(k ,n-k)}{\min^2[k/2, (n-k)/2]}\leq \frac{8}{k\wedge (n-k)}\ , 
\]
which is smaller than $8n^{-1/4}$ since $k\in [\sqrt{n}, n - n^{1/4}]$. 
For $n$ large enough, this is smaller than $0.45$ and together with
 Lemma \ref{lem:reduction_general}, we obtain 
\beq\label{eq:general_lower}
\cR[k,n]\geq \frac{|\theta_0-\theta_1|}{4} \left[ 0.55- \|\mathbf{P}_0- \mathbf{P}_1\|_{TV} \right]\ . 
\eeq

In view of \eqref{eq:general_lower}, the challenging part is to control the total variation distance between the mixture distributions $\mathbf{P}_0$ and $\mathbf{P}_1$. Contrary to the situation we dealt with in Lemma \ref{lem:ultra_dense}, one cannot easily derive a closed form formula for the $\chi^2$ distance between two mixtures. 
Instead, we shall rely on a general upper bound for mixture of normal distributions. Remember that $\phi$ denotes the standard normal measure and that, given a real probability measure $\pi$, we write $\pi* \phi$ for the corresponding convolution measure.

\begin{lem}\label{lem:tvmatching}
For two real probability measures $\pi_0$ and $\pi_1$, assume that $\pi_1(\{0\})>0$ and that the supports of $\pi_0$ and $\pi_1$ are bounded. Then we have 
\[
\|(\pi_0*\phi)^{\otimes n} - (\pi_1*\phi)^{\otimes n}\|_{TV}^2\leq  \frac{n}{\pi_1(\{0\})} \sum_{\ell\geq 1} \left( \int x^{\ell}(d\pi_0(x)- d\pi_1(x)) \right)^2  /\ell! \ . 
\]
\end{lem}

\begin{proof}[Proof of Lemma \ref{lem:tvmatching}]
  Let  $f_{0}$ (resp. $f_{1}$) the density associated to $\pi_0*\phi$ (resp. $\pi_1*\phi$) and let $Z_0$ (resp. $Z_1$) be a random variable distributed according to $f_0$ (resp. $f_1$). It follows from Le Cam's inequalities and tensorization identities for Hellinger distances~\cite[][Section 2.4]{MR2724359} that
   \begin{align}
\|(\pi_0*\phi)^{\otimes n} - (\pi_1*\phi)^{\otimes n}\|_{TV}^2
&\leq   \int \dots \int \left(\prod_{i=1}^n  f_0^{1/2}(y_i)  - \prod_{i=1}^n  f_1^{1/2}(y_i)\right)^2 dy_1 \dots dy_n\nonumber\\
&\leq     n \int \left( f_0^{1/2}(y)- f^{1/2}_1(y)\right)^2 dy\ . \label{equinterm0}
 \end{align}
Obviously, we have $f_{1}(y)\geq \pi_1(\{0\})\phi(y)$, which enforces 
 \begin{align}
\int \left( f_{0}^{1/2}(y)- f_{1}^{1/2}(y)\right)^2 dy &=  \int \frac{\left( f_{0}(y)- f_{1}(y)\right)^2}{\left( f_{0}^{1/2}(y)+ f_{1}^{1/2}(y)\right)^2}  dy  \nonumber \\ & \leq  \pi^{-1}_1(\{0\}) \int  \left( f_{0}(y)- f_{1}(y)\right)^2/\phi(y)  dy\ .  \label{equ_interm2}
\end{align}

Next, we use Hermite's polynomials $(H_k(\cdot)/(k!)^{1/2})_{k\geq 0}$ as an Hilbert basis of $L^2(\mathbb{R},\phi)$ (the space of square integrable function with respect to the normal measure) and the relation $\frac{\phi(y-x) }{\phi(y)} = 1+\sum_{\ell\geq 1} H_\ell(y) x^\ell/\ell!$ (see, e.g., (1.1) in \cite{Foa1981}) to obtain that the rhs of \eqref{equ_interm2} is upper-bounded by
 \begin{align}
&  \pi^{-1}_1(\{0\}) \int \left( \int \frac{\phi(y - x)}{\phi(y)} d\pi_0(x)-  \int \frac{\phi(y - x)}{\phi(y)}d\pi_1(x) \right)^2   \phi(y) dy \nonumber\\
 =&\: \pi^{-1}_1(\{0\})  \int \left( \sum_{\ell\geq 1} \frac{H_\ell(y)}{\ell!}\left( \int x^\ell d\pi_0(x) -\int x^\ell d\pi_1(x) \right) \right)^2   \phi(y) dy 
 \nonumber\\
 = &\: \pi^{-1}_1(\{0\}) \sum_{\ell\geq 1}   \frac{1}{\ell!}\left( \int x^{\ell} (d\pi_0(x)-d\pi_1(x)) \right)^2  ,
 \end{align}
 where we used in the last line the orthonormality of the Hermite polynomials. This concludes the proof. 
 \end{proof}

Now, if we take $\theta_1=0$ and we define  $\pi_i=\delta_{\theta_i}*\nu_i$ (where $\delta_{x}$ is the Dirac measure), we have $\mathbf{P}_i= (\pi_i*\phi)^{\otimes n}$ and we are in position to apply Lemma \ref{lem:tvmatching}.
If we further assume that, for some integer $m>2$ and some $M>0$ the support of $\pi_0$ and $\pi_1$ is included in $[-M,M]$ and that their $m$ first moments are matching
\beq\label{eq:com}
\int x^\ell d\pi_0(x)= \int x^\ell d\pi_1(x),\quad \forall \ell= 1,\ldots, m\ ,
\eeq
we can derive from Lemma \ref{lem:tvmatching} and \eqref{eq:assumption_nu_souligne} that 
\beqn 
\|\mathbf{P}_0 - \mathbf{P}_1\|_{TV}^2&\leq & \frac{n}{1- k_0/n} \sum_{\ell>m } \left( \int x^{\ell}(d\pi_0(x)- d\pi_1(x)) \right)^2  /\ell! \\ 
&\leq &  \frac{n^2}{n-k_0}\sum_{\ell>m }\frac{1}{\ell !}\left[ \frac{2k_0}{n }M^\ell + |\theta_0|^\ell\right]^2\\
&\leq & \frac{2n^2}{n-k_0}\sum_{\ell>m }\left[ \frac{4k^2_0}{ n^2 }\left(\frac{eM^2}{\ell}\right)^{\ell} + \left(\frac{e|\theta_0|^2}{\ell}\right)^{\ell} \right]\ .  
\eeqn 
Then, if $|\theta_0|$, $m$, and $M$ are such that $m\geq 2e(M^2\vee \theta^2_0)$, we obtain
\beqn 
\|\mathbf{P}_0 - \mathbf{P}_1\|_{TV}^2\leq \frac{8 k_0^2 }{n-k_0}2^{-m} + 2\frac{n^2}{n-k_0}\left(\frac{e\theta_0^2}{m}\right)^m.
\eeqn 
If $m$ is large enough this will imply that $\|\mathbf{P}_0 - \mathbf{P}_1\|_{TV}\leq 0.5$. Putting everything together and coming back to \eqref{eq:general_lower}, we conclude that 
\beq\label{eq:general_lower_2}
\cR[k,n]\geq \frac{|\theta_0|}{80}\ ,
\eeq
if there exists $\theta_0$, $\pi_0$ and $\pi_1$ such that for $m_0= \lceil \log(64 k_0^2/(n-k_0))  /\log(2)  \rceil $ and  $M_0= \sqrt{m_0/2e}$
\beq\label{eq:condition_pi_0_pi_1}
\left\{\begin{array}{l}
\min(\pi_0(\{\theta_0\}),\pi_1(\{0\}))\geq 1- \frac{k_0}{n}\ ;\\        
\pi_0 \text{ (resp. } \pi_1\text{)}\text{ is supported on }  [\theta_0,M_0] \text{ (resp. }[0,M_0] \text{)}\ ;\\
\int x^{\ell}d\pi_0(x)=\int x^{\ell}d\pi_1(x) ,\quad \forall \ell= 1,\ldots, m_0\ ;\\
\theta_0^2 \leq \frac{m_0}{e} \left[\left(\frac{n-k_0}{16n^2}\right)^{1/m_0}\wedge \frac{1}{2}\right]\ .
       \end{array}
\right.
\eeq
The remainder of the proof is devoted to demonstrate the existence of such $\pi_0$ and $\pi_1$, for $|\theta_0|$ taken as large as possible.

\subsubsection{Step $3$: Existence of $\pi_0$ and $\pi_1$}

  \begin{lem}\label{lemlb3}
  For any positive numbers $M>0$, $\eta>0$ and any positive integer $m$, there exist two probability measures $\pi_0$ and $\pi_1$ respectively supported on $[-\eta M ,M]$ and $[0,M]$, whose $m$ first moments are matching and such that 
 \beqn 
 \min(\pi_0(\{-\eta M\}), \pi_1(\{0\}))\geq  \left[1+ \eta m^2 e^{2\sqrt{\eta}m}\right]^{-1}.
 \eeqn 
  \end{lem}

 \begin{proof}[Proof of Lemma \ref{lemlb3}]
Consider the space  $\cC_0$ of continuous functions from $[0,1]$ to $\R$, endowed with the supremum norm. Let  $\mathcal{P}_k$ be the subset $\cC_0$ made  of  polynomials of degree at most $m$. 
Consider the linear map $\Lambda : \mathcal{P}_N \rightarrow \R$ defined by $\Lambda(P)=P(-\eta)-P(0)$.
 Then, by the Hahn-Banach theorem, $\Lambda$ can be extended into a linear map $\wt{\Lambda}$ on the whole subspace $\cC_0$ without increasing its operator norm, that is 
 $$
\sup_{\substack{f \in \cC_0\\ \|f\|_\infty\leq 1}}|\wt{\Lambda}(f)| =  \sup_{\substack{f \in \mathcal{P}_m\\ \|f\|_\infty\leq 1}}|\Lambda(f)|\ .
 $$
Since, for $f\in \cP_m$, $|f(-\eta)-f(0)|\leq \eta \sup_{x\in [-\eta,0]}|f'(x)|$, we derive from Markov's theorem (Lemma~\ref{lem:optimalTN} (ii) and (i)) that 
\beq\label{equcetoilemaj}
\sup_{\substack{f \in \cC_0\\ \|f\|_\infty\leq 1}}|\wt{\Lambda}(f)|\leq  \frac{\eta m^2}{1+\eta}\cosh[ m\ \arccosh(1+2\eta)]\leq \eta m^2 e^{2\eta\sqrt{m}} =: a^{\star}(m,\eta)\ ,  
\eeq
where we used that $\cosh(x)\leq e^{x}$ and $\arccosh(1+x)\leq \sqrt{2x}$. 
Next, by Riesz representation theorem, there exists a signed measure $\underline{\pi}$ such that  $\wt{\Lambda}(f)= \int f d\pi$ for all $f\in \cC_0$. Decomposing $\underline{\pi}= \underline{\pi}^+ - \underline{\pi}^-$ as a difference of positive measure supported on $[0,1]$, if follows from the values of $\Lambda(x^{\ell})$ for $\ell=0,\ldots, m$, that $|\underline{\pi}^+|=|\underline{\pi}^-|$ and $\int x^{\ell}(d\underline{\pi}^+(x) - d\underline{\pi}^-(x)) = (-\eta)^\ell$. Besides, the total variation norm  $\|\underline{\pi}\|_{TV}= 2|\underline{\pi}^+|$ is upper bounded by $a^{\star}(m,\eta)$.  Let now define the two probability measures 
 \begin{align*}
  \pi_{1} &= \frac{1}{1+\|\pi\|_{TV}/2}\left( \delta_0 + \underline{\pi}^+(\cdot M) \right)\ ;\, \quad \quad  \pi_{0} = \frac{1}{1+\|\pi\|_{TV}/2}\left( \delta_{-\eta} + \underline{\pi}^-(\cdot M) \right)\ . 
 \end{align*}
 Obviously, the $m$ first moments of $\pi_0$ and $\pi_1$ are matching and $\min(\pi_0(\{-\eta M\}), \pi_1(\{0\}))\geq (1+a^\star (m,\eta))^{-1}$. 
 \end{proof}

Note that, for any $x,t>0$, if $x\leq 0.5\log(1+t)$, we have $xe^{x}\leq t$. Applying Lemma \ref{lemlb3} with $m_0$, $M_0$ and any $\eta$ such that
 \[
\eta \leq \eta_0=\frac{\log^2\left(1+ \sqrt{\frac{k_0}{n-k_0}}\right)}{4m^2_0}\ , 
 \]
we conclude that $\min(\pi_0(\{-\eta M\}), \pi_1(\{0\}))\geq 1-k_0/n$. As a consequence, if we choose $\theta_0$ negative with
\[
|\theta_0|\leq  \frac{M_0\log^2\left(1+ \sqrt{\frac{k_0}{n-k_0}}\right)}{4m^2_0}\bigwedge \left(\sqrt{\frac{m_0}{2e}} \left(\frac{n-k_0}{16n^2}\right)^{1/(2m_0)}\right)\ ,
\]
then, there exist $\pi_0$ and $\pi_1$ satisfying Conditions \eqref{eq:condition_pi_0_pi_1}. From \eqref{eq:general_lower_2}, we conclude that 
\[
\cR[k,n]\geq c \left[\frac{\log^2\left(1+ \sqrt{\frac{k_0}{n-k_0}}\right)}{m^{3/2}_0}\bigwedge \left(m_0^{1/2} \left(\frac{n-k_0}{16n^2}\right)^{1/(2m_0)}\right)\right]\ . 
\]
In fact, the second expression in the rhs is always larger (up to a numerical constant) than the first one. Indeed, for $m_0=1$ (which corresponds to $k_0\lesssim \sqrt{n}$), this expression is of order $1/\sqrt{n}$. When $k_0\leq n^{1/3}$ and $m_0\geq 2$, this expression is higher than $n^{-1/4}$ which is again higher than the first term. For $k_0\in (n^{1/3},n-1]$, the expression in the rhs is higher than $\sqrt{\log(n)}$ which is again no less than the left expression. In view of the definition of $k_0$, we have proved that 
\[
 \cR[k,n]\geq c \frac{\log^2\left(1+ \sqrt{\frac{k}{n-k}}\right)}{\log^{3/2}\left(1+ \frac{k}{\sqrt{n}}\right)}\ ,
\]
 which concludes the proof of Theorem \ref{thm:lower_one_sided}

 \subsection{Proof of Theorem \ref{thm:lower_mean_robust} (OSC)}\label{p:thm:lower_mean_robust}

Note that, for $k\leq 4\sqrt{n}$ and for $n-k\leq 2\sqrt{n}$, the minimax lower bound \eqref{eq:lower_minimax_robust} is a consequence of the lower bound of Theorem \ref{thm:lower_one_sided} for the specific gOSC model. As a consequence, we only have to show the result for $k\in (2\sqrt{n}, n-2\sqrt{n})$. Consider one such $k$. As in the proof of Theorem \ref{thm:lower_one_sided}, we define $k_0= \max(k/2, n-3(n-k)/2)<k$.

 \medskip 
 
 As for the  proof of Theorem \ref{thm:lower_one_sided}, we shall rely on a two point Le Cam's approach but the construction of the distribution is quite different. 
 Let us denote $\P_x$ for the distribution $\mathcal{N}(x,1)$ and $\phi_x$ for its usual density. 
 Let $\theta>0$ be a positive number whose value will be fixed later and define $\epsilon=k_0/n$. We shall introduce below two probability measures $\mu_0$ and $\mu_1$ that stochastically dominate $\P_{-\theta}$ and $\P_{\theta}$. Consider the mixture distribution $\vartheta_0= (1-\epsilon)\P_{-\theta}+ \epsilon \mu_0 $ and  $\vartheta_1=(1-\epsilon)\P_{\theta}+ \epsilon \mu_1$ and $\mathbf{P}_0= \vartheta_0^{\otimes n}$ and $\mathbf{P}_1= \vartheta_1^{\otimes n}$. 
 Under $\mathbf{P}_0$, all variables $Y_i$ are sampled independently and with probability $(1-\epsilon)$ follow the normal distribution $\P_{-\theta}$ and with probability $\epsilon$ follows the stochastically larger distribution $\mu_0$. Let $Z$ be a binomial variable with parameters $(n,\epsilon)$.  Under $\mathbf{P}_0$, $Z$ of the observations have been sampled according to $\mu_0$ and the $n-Z$ remaining observations have been sampled according to $\P_{-\theta}$. Thus, up to an event of probability $\P[\mathcal{B}(n,\epsilon)\leq k]$, $\mathbf{P}_0$ is a mixture of distributions in $\overline{\cM}_k$ whose corresponding functional is  $-\theta$. The measure $\mathbf{P}_1$ satisfies the same property with $-\theta$ replaced by $\theta$.

  Arguing as in Lemma \ref{lem:reduction_general}, we therefore obtain
\[
\overline{\cR}[k,n]\geq \frac{\theta}{2} \left[ 1- \|\mathbf{P}_0- \mathbf{P}_1\|_{TV}  - 2\P[Z>k]	 \right]\ .
\]
The probability $\epsilon=k_0/n$ has been chosen small enough  that $\P[Z>k]$ is vanishing for $n$ large enough (see the proof of Theorem \ref{thm:lower_one_sided}, Step $1$) so that for such $n$, we obtain $\overline{\cR}[k,n]\geq \frac{\theta}{2} [ 0.55- \|\mathbf{P}_0- \mathbf{P}_1\|_{TV}]$ (see \eqref{eq:general_lower}) and thus 
\beq\label{eq:lower_risk_robust_general}
\overline{\cR}[k,n]\geq \frac{\theta}{40}\quad \text{ if }\quad \|\mathbf{P}_0- \mathbf{P}_1\|_{TV}\leq 0.5\ .
\eeq
 In the sequel, we fix 
\beq\label{eq:choice_theta_proof_robust}
\theta= \frac{\log\left(\frac{1}{1-\epsilon}\right)}{8\sqrt{\log(n\epsilon^2)}} \ 
\eeq
and we will shall build two measure $\mu_0$ and $\mu_1$ that enforce $\|\mathbf{P}_0- \mathbf{P}_1\|_{TV}\leq 0.5$. In view of \eqref{eq:lower_minimax_robust} and  \eqref{eq:lower_risk_robust_general}, this will conclude the proof.

\bigskip 
Define
\beq\label{eq:definition_t_theta}
t_{\theta}= \frac{1}{2\theta}\log\left(\frac{1}{1-\epsilon}\right)\ . 
\eeq
We shall pick $\mu_0$ and $\mu_1$ in such a way that the densities of $\vartheta_0$ and $\vartheta_1$ are matching on the widest interval possible. Define $\mu_0$ and $\mu_1$ by their respective densities $f_0$ and $f_1$
\[
 f_0(x) = g_0 (x) +  h(x)\ , \quad \quad f_1(x) =  g_1 (x) +  h(x)\ ,
\]
where 
\beqn 
g_0(x)&=& \frac{1-\epsilon}{\epsilon} [\phi_{\theta}(x)- \phi_{-\theta}(x)]\text{ if }x\in (0,t_{\theta}]\quad \text{ and }g_0(x)=0 \text{ else}\\
g_1(x)&=& \frac{1-\epsilon}{\epsilon} [\phi_{-\theta}(x)- \phi_{\theta}(x)]\text{ if }x\in [-t_{\theta};0)\quad \text{ and }g_1(x)=0 \text{ else}\ ,
\eeqn 
and  $h(x)= \mathbf{1}_{x> u} a \phi_{2\theta}(x)$ where $u>\max(n+\theta, t_{\theta})$  and $a\geq 1$ are taken such that $\int (g_0(x)+ h(x))dx=1$. To ensure the existence of $h$ (that is, of such $a$ and $u$), we need to prove that $\int g_0(x)dx= \int g_1(x)dx<1$. By definition \eqref{eq:definition_t_theta} of $t_{\theta}$, we have $\phi_{\theta}(x)/\phi_{-\theta}(x)< (1-\epsilon)^{-1}$ for all $x\in (0,t_{\theta})$. This implies that  $g_0(x) < \phi_{-\theta}(x)$ for all $x\in (0,t_{\theta})$ and $g_1(x) < \phi_{\theta}(x)$ for all $x\in (-t_{\theta},0)$, which entails $\int g_0(x)dx= \int g_1(x)dx<1$.

Also, the two measures $\mu_0$ and $\mu_1$ respectively satisfy  $\mu_0 \succeq \P_{-\theta}$ and $\mu_1 \succeq  \P_{\theta}$. Let us only prove the second inequality, the first one being simpler. Consider any $t\in \mathbb{R}$. Then, it follows from the definition of $f_1$ that
\[
\int_{-\infty}^{t}f_1(x)dx = \left\{
\begin{array}{cc}
0 & \text{ if } t<-t_{\theta}\ ;\\
\int_{-t_\theta}^{t}g_1(x)dx& \text{ if }t\in [-t_{\theta},0]\ ; \\
\int_{-t_\theta}^{0}g_1(x)dx& \text{ if }t\in (0,u]\ ; \\
1- a \int_{t}^{\infty}\phi_{2\theta}(x)dx& \text{ if }t>u\ .\\
\end{array}
\right.
\]
Since $g_1(x)< \phi_{\theta}(x)$ for all $x\in (-t_{\theta},0)$, we readily obtain $\int_{-\infty}^{t}f_1(x)dx< \int_{-\infty}^{t}\phi_{\theta}(x)dx$ for all $t\leq u$. For $t> u$, we have $\int_{t}^{\infty}f_1(x)dx=  a\int_{t}^{\infty}\phi_{2\theta}(x)dx \geq \int_{t}^{\infty}\phi_{\theta}(x)dx$, which implies $\mu_1 \succeq  \P_{\theta}$.

\bigskip

It remains to prove that  $\|\mathbf{P}_0-\mathbf{P}_1\|_{TV}\leq 0.5$. Denote  $ H^2(\mathbf{P}_0; \mathbf{P}_1)$ the square Hellinger distance. As in the previous proof, it follows from Le Cam's inequalities and tensorization identities for Hellinger distances~\cite[][Section 2.4]{MR2724359} that
\beqn 
\|\mathbf{P}_0- \mathbf{P}_1\|^2_{TV}&\leq& H^2(\mathbf{P}_0; \mathbf{P}_1)= 2\left[1- \left(1- \frac{H^2(\vartheta_0;\vartheta_1)}{2}\right)^n\right]\\
&\leq &   nH^2(\vartheta_0;\vartheta_1)\ . 
\eeqn 
As a consequence, for $n$ large enough, one has  $\|\mathbf{P}_0-\mathbf{P}_1\|_{TV}\leq 0.5$ as long as  $H^2(\vartheta_0;\vartheta_1)\leq (2n)^{-1}$. It remains to prove this last inequality. 
Write $v_0= (1-\epsilon)\phi_{-\theta}+ \epsilon (g_0+ h)$ the density of $\vartheta_0$ and $v_1$ the density of $\vartheta_1$. 
$g_0$ and $g_1$ have been chosen in such a way that $v_0$ and $v_1$ are matching on the interval $[-t_{\theta}; t_{\theta}]$. 
\beqn 
\frac{1}{2}H^2(\vartheta_0; \vartheta_1)&= & 1- \int \sqrt{v_0(x)v_1(x)}dx 
 = \int (v_0(x) - \sqrt{v_0v_1}(x))dx \\
& =& (1-\epsilon)\left[ \int_{(-\infty;-t_{\theta}]\cup [t_{\theta},\infty)}\big[ \phi_{-\theta}(x) - \sqrt{ \phi_{-\theta}(x) \phi_{\theta}(x)}\big]dx\right]\\
&&+ \int_{(u,\infty)} \Big[v_0(x) - \sqrt{v_1v_0}(x)- (1-\epsilon) \big[ \phi_{-\theta}(x) - \sqrt{ \phi_{-\theta}(x) \phi_{\theta}(x)}\big]\Big]dx\\
&\leq& (1-\epsilon)\left[ \int_{(-\infty;-t_{\theta}]\cup [t_{\theta},\infty)}\big[ \phi_{-\theta}(x) - \sqrt{ \phi_{-\theta}(x) \phi_{\theta}(x)}\big]dx\right]+  e^{-n^2/2}\ ,
\eeqn 
where we used $v_0(x)\leq v_1(x)$ for $x>u$. This leads us to 
\beq\label{eq:upper_hellinger}
H^2(\vartheta_0; \vartheta_1) \leq   2e^{-n^2/2}+ 2(1-\epsilon)\left[\overline{\Phi}(t_{\theta}+ \theta) +  \overline{\Phi}(t_{\theta}- \theta) - e^{-\theta^2/2} 2\overline{\Phi}(t_{\theta})\right].
\eeq
Since $k\geq 4\sqrt{n}$, we have $16\log(n\epsilon^2) \geq \log(1/(1-\epsilon))$ which entails $\theta \leq t_\theta/2$. 
This leads us to
\beqn 
\overline{\Phi}(t_{\theta}+ \theta) +  \overline{\Phi}(t_{\theta}- \theta) - 2\overline{\Phi}(t_{\theta})
&\leq&  \theta^2  \sup_{x\in [t_{\theta}-\theta; t_{\theta}+ \theta]}|\phi'(x) |\\
&\leq &\frac{\theta^2}{\sqrt{2\pi}}\frac{3t_{\theta}}{2}e^{-t_{\theta}^2/8} \\
&\leq & \frac{3}{ 32 \sqrt{2\pi} \log(n\epsilon^2) (n\epsilon^2)^2} \log^2\left(\frac{1}{1-\epsilon}\right)\\
& \leq & \frac{1}{8 \sqrt{2\pi} n^2 \epsilon^2 (1-\epsilon)^2}\ , 
\eeqn 
where we used the definitions \eqref{eq:choice_theta_proof_robust} and \eqref{eq:definition_t_theta} of $\theta$ and $t_{\theta}$ and $\log(n\epsilon^2)\geq 1$. 
Similarly, one has 
\beqn 
\overline{\Phi}(t_{\theta})(1- e^{-\theta^2/2})&\leq&  e^{-t^{2}_{\theta}/2} \frac{\theta^2}{2} \leq  \frac{1}{128\log(n\epsilon^2)(n\epsilon^2)^2} \log^2\left(\frac{1}{1-\epsilon}\right) 
\leq  \frac{1}{128 n^2 \epsilon^2 (1-\epsilon)^2}\ . 
\eeqn 
Coming back to \eqref{eq:upper_hellinger}, we conclude that 
\[
H^2(\vartheta_0; \vartheta_1)\leq 2e^{-n^2/2}+ \frac{1}{8n^2\epsilon^2 (1-\epsilon)^2} \leq 2 e^{-n^2/2}+ \frac{1}{8n(1-cn^{-1/2})}\ , 
\]
since $\epsilon\geq 2n^{-1/2}$ and $1-\epsilon\geq 3n^{-1/2}$. For  $n$ large enough, we obtain $H^2(\vartheta_0; \vartheta_1)\leq 1/(2n)$, which  concludes the proof.

 \subsection{Proof of Theorem \ref{thm:lower_var_robust} (OSC)}\label{p:thm:lower_var_robust}

 This proof proceeds from the same approach as that of Theorem \ref{thm:lower_mean_robust} but the construction of the prior distributions  are quite different.  For any fixed numerical constant $c_0>0$ and any  $k\leq c_0 \sqrt{n}$, the lower bound in the  theorem is parametric and is easily proved in a model without contamination. We assume henceforth that $k> c_0\sqrt{n}$ and we will fix the value of $c_0$ at the end of the proof. 
 Also for $n-k\leq 2\sqrt{n}$, the lower bound in Theorem \ref{thm:lower_var_robust} is of the order of a constant, so that we only have to prove the result for $n-k> 2\sqrt{n}$, so we also assume henceforth that $n-k\geq 2\sqrt{n}$. 

As in the previous proof, we define $k_0= \max(k/2, n-3(n-k)/2)<k$.
Let us denote $\underline{\P}_{0,y}$ for the distribution $\mathcal{N}(0,y^2)$ and $\phi_{0,y}$ for its usual density.
Let $\sigma>1$ be a positive quantity that will be fixed later and let  $\epsilon=k_0/n$. We shall introduce below two probability measures $\mu_0$ and $\mu_1$ that stochastically dominate $\P_{0,1}$ and $\P_{0,\sigma}$. Consider the mixture distribution 
$\vartheta_0= (1-\epsilon)\underline{\P}_{0,1}+ \epsilon \mu_0 $ and  $\vartheta_1=(1-\epsilon)\underline{\P}_{0,\sigma}+ \epsilon \mu_1$ and $\mathbf{P}_0= \vartheta_0^{\otimes n}$ and $\mathbf{P}_1= \vartheta_1^{\otimes n}$. 
Arguing as in the proof of Theorem \ref{thm:lower_mean_robust} (and Lemma \ref{lem:reduction_general}), we obtain that 
\beq\label{eq:lower_risk_robust_sigma_general}
\overline{\cR}_v[k,n]\geq \frac{\sigma-1 }{80}\quad \text{ if }\quad \|\mathbf{P}_0- \mathbf{P}_1\|_{TV}\leq 0.5\ .
\eeq
 In the sequel, we fix 
\beq\label{eq:choice_sigma_proof_robust}
\sigma= 1+ \frac{\log\left(\frac{1}{1-\epsilon}\right)}{6\log(n\epsilon^2)}\wedge 1 \ . 
\eeq
and we will shall build two measure $\mu_0$ and $\mu_1$ that enforce $\|\mathbf{P}_0- \mathbf{P}_1\|_{TV}\leq 0.5$. In view of the two previous inequalities this will conclude the proof.

We shall pick $\mu_0$ and $\mu_1$ in such a way that the densities of $\vartheta_0$ and $\vartheta_1$ are matching on the widest interval possible. Denote $\mu_0$ and $\mu_1$ by their respective densities $f_0$ and $f_1$. In principle, we would like to take $f_0=(1-\epsilon)/\epsilon [\phi_{0,\sigma}- \phi_{0,1}]_+$ and $f_1
=(1-\epsilon)/\epsilon [\phi_{0,1}- \phi_{0,\sigma}]_+$ as this would enforce $\|\mathbf{P}_{0}-\mathbf{P}_1\|_{TV}=0$. Unfortunately, such a choice is not possible as the corresponding measure $\mu_0$ would not be a probability measure (and would not either dominate $\P_{0,1}$). The actual construction is a bit more involved. First, define
\beq\label{eq:definition_u_v_sigma}
v_{\sigma}= \sqrt{\frac{2\sigma^2}{\sigma^2 -1}\log(\sigma)},\ \quad w_{\sigma}=\sqrt{\frac{2\sigma^2}{\sigma^2 -1}\log\left(\frac{\sigma}{1-\epsilon}\right)}\ . 
\eeq
We have $\phi_{0,\sigma}(t)\geq \phi_{0,1}(t)$ if and only if $|t|\geq v_{\sigma}$ and $(1-\epsilon)\phi_{0,\sigma}(t)\geq \phi_{0,1}(t)$ for all $|t|\geq w_{\sigma}$. This implies 
\[
 \int_{0}^{v_{\sigma}}\phi_{0,1}(x)-\phi_{0,\sigma}(x)dx> \int_{v_{\sigma}}^{w_{\sigma}}\phi_{0,\sigma}(x)-\phi_{0,1}(x)dx\ .
\]
Thus, we can define $u_{\sigma}\in (0,v_{\sigma})$ in such a way that 
\beq\label{eq:definition_u_sigma}
\int_{u_{\sigma}}^{v_{\sigma}} \phi_{0,1}(x)-\phi_{0,\sigma}(x)dx=  \int_{v_{\sigma}}^{w_{\sigma}}\phi_{0,\sigma}(x)-\phi_{0,1}(x)dx\ .
\eeq
Then, we take
\[
 f_0(x) = g_0 (x) +  h(x)\ , \quad \quad f_1(x) =  g_1 (x) +  h(x)\ ,
\]
where 
\beqn 
g_0(x)&=& \frac{1-\epsilon}{\epsilon} [\phi_{0,\sigma}(x)- \phi_{0,1}(x)]\text{ if }|x|\in [v_{\sigma}, w_{\sigma}]\,\quad  \text{ and }
g_0(x)=0~~~\text{else.}\\
g_1(x)&=& 
\frac{1-\epsilon}{\epsilon} [\phi_{0,1}(x)- \phi_{0,\sigma}(x)] \text{ if }|x|\in [u_{\sigma}, v_{\sigma}]\, \quad \text{ and }
g_1(x)=0~~~\text{else.}\
\eeqn 
By definition of $v_{\sigma}$ and $w_{\sigma}$, $g_0$ is nonnegative and is smaller or equal to $\phi_{0,1}$. 
As a consequence, $\int g_0(x) < \int \phi_{0,1}(x)dx \leq 1$. Besides, $u_{\sigma}$ has been chosen in such a way that $\int g_0(x)= \int g_1(x)dx$. Finally, we  define   $h(x)= \mathbf{1}_{x> s} a \phi_{0,\sigma}(x)$ where $s>n\sigma+w_{\sigma}$  and $a\geq 1$ are taken such that $\int (g_0(x)+ h(x))dx=1$. 

Since we assume that $\sigma(1-\epsilon)<1$, observe that $(1-\epsilon)\phi_{0,1}(t) \leq \phi_{0,\sigma}(t)$ for all $t\in \mathbb{R}$, which in turn implies that  $g_1\leq \phi_{0,\sigma}$. Since $g_0\leq \phi_{0,1}$, it follows that the two measures $\mu_0$ and $\mu_1$ respectively satisfy  $\mu_0\succeq \P_{0,1}$ and $\mu_1\succeq  \P_{0,\sigma}$.

It remains to prove that  $\|\mathbf{P}_0-\mathbf{P}_1\|_{TV}\leq 0.5$. As in the previous proof, we have 
$\|\mathbf{P}_0- \mathbf{P}_1\|^2_{TV} \leq nH^2(\vartheta_0;\vartheta_1)$ and we only have to prove that   $H^2(\vartheta_0;\vartheta_1)\leq (2n)^{-1}$ for $n$ large enough. 

To compute this Hellinger distance, we first observe that the densities $v_0$ and $v_1$ associated to $\vartheta_0$  and  $\vartheta_1$ are matching in $[-w_{\sigma},u_{\sigma}]\cup [u_{\sigma},w_{\sigma}]$. Together with the definition of $f_0$ and $f_1$ this leads us to 
\beqn 
\frac{1}{2}H^2(\vartheta_0; \vartheta_1)&= & 1- \int \sqrt{v_0(x)v_1(x)}dx 
 = \int (v_0(x) - \sqrt{v_0v_1}(x))dx \\
& =& 2(1-\epsilon) \int_{[0,u_\sigma]\cup[w_{\sigma},\infty)} \left[ \phi_{0,1}(x) - \sqrt{\phi_{0,1}(x)\phi_{0,\sigma}(x)}\right]dx\\
&&+ \int_{(s,\infty)}\left[ v_0(x) - \sqrt{v_1v_0}(x)+ (1-\epsilon)\left(\sqrt{\phi_{0,1}(x)\phi_{0,\sigma}(x)}-\phi_{0,1}(x)\right)\right]dx\ .
\eeqn 
Since $v_{0}(x)\leq v_1(x)$ for $x> s$, the last term is less or equal to $\int_{s}^{\infty}\phi_{0,\sigma}(x)dx\leq e^{n^2/2}$. 
It follows from \eqref{eq:definition_u_sigma} that 
\[
\int_{0}^{u_{\sigma}} (\phi_{0,1}(x)-\phi_{0,\sigma})(x)dx=  \int_{w_{\sigma}}^{\infty} (\phi_{0,\sigma}- \phi_{0,1})(x)dx\ .
\]
This leads us to 
\beq\label{eq:upper_HH_sigma}
\frac{1}{2}H^2(\vartheta_0; \vartheta_1)\leq   e^{-n^2/2} + 
(1-\epsilon) \int_{[0,u_{\sigma}]\cup [w_{\sigma},\infty)}  \phi_{0,1}(x) + \phi_{0,\sigma}(x) - 2\sqrt{\phi_{0,1}(x)\phi_{0,\sigma}(x)}dx\ . 
\eeq
For $\delta\in [-1/2,1]$, Taylor's formula leads to $(2+\delta- 2\sqrt{1+\delta})\leq \delta^2/\sqrt{2}$. As a consequence, for any $x$ such that $\phi_{0,\sigma}(x)/\phi_{0,1}(x)\in [1/2,2]$, we have
\beq\label{eq:upper_diff_phi}
 \phi_{0,1}(x)+ \phi_{0,\sigma}(x) - 2\sqrt{\phi_{0,1}(x)\phi_{0,\sigma}(x)}\leq  \frac{(\phi_{0,1}(x)-\phi_{0,\sigma}(x))^2}{\sqrt{2}\phi_{0,1}(x)}
= \frac{\phi_{0,1}(x)}{\sqrt{2}}\left[ \frac{1}{\sigma}\exp\left(\frac{x^2(\sigma^2-1)}{2\sigma^2}\right)-1 \right]^2\ . 
\eeq
 Define $z_{\sigma}=\sqrt{\frac{2\sigma^2}{\sigma^2 -1}[\log(2\sigma)\wedge  1]}$.  For any $|x|\leq z_{\sigma}$, we have $\phi_{0,\sigma}(x)\leq 2 \phi_{0,1}(x)$. From the previous inequality, we derive that, for  $|x|\in (w_{\sigma};w_{\sigma}\vee z_{\sigma})$,
\beqn 
 \phi_{0,1}(x)+ \phi_{0,\sigma}(x) - 2\sqrt{\phi_{0,1}(x)\phi_{0,\sigma}(x)}&\leq&  \frac{\phi_{0,1}(x)}{\sqrt{2}}\left[ \frac{1}{\sigma}\exp\left(\frac{x^2(\sigma^2-1)}{2\sigma^2}\right)-1 \right]^2\\
 &\leq & \frac{\phi_{0,1}(x)}{\sqrt{2}} \left[2\frac{x^2(\sigma^2-1) }{2\sigma^2}+ \frac{1}{\sigma}-1 \right]_+^2\\
 &\leq & 3\phi_{0,1}(x) x^4 (\sigma-1)^2\ . 
\eeqn 
Since $\sigma\leq 2$, we have $\phi_{0,\sigma}(x)\geq  \phi_{0,1}(x)/2$ for all $x$. As a consequence of \eqref{eq:upper_diff_phi}, we obtain that, for $|x|\leq u_{\sigma}$,
\[
\phi_{0,1}(x)+ \phi_{0,\sigma}(x) - 2\sqrt{\phi_{0,1}(x)\phi_{0,\sigma}(x)}\leq  \frac{(\phi_{0,1}(x)-\phi_{0,\sigma}(x))^2}{\sqrt{2}\phi_{0,1}(x)}\leq  \big[\phi_{0,1}(x)-\phi_{0,\sigma}(x)\big] \frac{\sigma -1}{\sqrt{2}\sigma}\ . 
\]
Coming back to \eqref{eq:upper_HH_sigma}, we obtain 
\begin{eqnarray}
&&\frac{\frac{1}{2}H^2(\vartheta_0; \vartheta_1)-e^{-n^2/2}}{1-\epsilon} \nonumber\\
&\leq & \frac{\sigma -1 }{\sqrt{2}\sigma}\int_{0}^{u_{\sigma}}(\phi_{0,1}(x)-\phi_{0,\sigma}(x))dx+ 3(\sigma-1)^2\int_{w_{\sigma}}^{w_{\sigma}\vee z_{\sigma}}x^4\phi_{0,1}(x)dx+  \phi_{0,1}\big(\frac{w_{\sigma}\vee z_{\sigma}}{\sigma}\big)  \nonumber \\ 
&\leq &\frac{\sigma -1 }{\sqrt{2}\sigma}\int_{w_{\sigma}}^{\infty}(\phi_{0,\sigma}(x)-\phi_{0,1}(x))dx+ 3(\sigma-1)^2\int_{w_{\sigma}}^{\infty}x^4\phi_{0,1}(x)dx+  \phi_{0,1}\big(\frac{w_{\sigma}\vee z_{\sigma}}{\sigma}\big) \nonumber\\ 
&\leq & \frac{\sigma -1 }{\sqrt{2}\sigma}\int_{w_{\sigma}/\sigma}^{w_{\sigma}}\phi_{0,1}(x)dx+ 3(\sigma-1)^2[w_{\sigma}^3+6w_{\sigma}]\phi_{0,1}(w_{\sigma})+  \phi_{0,1}\big(\frac{ w_{\sigma}\vee z_{\sigma}}{\sigma}\big) \nonumber \\
&\leq &w_{\sigma}\frac{(\sigma -1)^2 }{\sqrt{2}\sigma^2}\phi_{0,1}(\frac{w_{\sigma}}{\sigma})+ 3(\sigma-1)^2[w_{\sigma}^3+6w_{\sigma}]\phi_{0,1}(w_{\sigma})+  \phi_{0,1}\big(\frac{ w_{\sigma}\vee z_{\sigma}}{\sigma}\big) \nonumber\\
&\leq & (\sigma-1)^2\left[3w_{\sigma}^3 + 7w_{\sigma} \right]+ \phi_{0,1}\big(\frac{ w_{\sigma}\vee z_{\sigma}}{\sigma}\big) \ , \label{eq:upper_HH_sigma_2}
 \end{eqnarray}
where we used the definition \eqref{eq:definition_u_sigma} of $u_{\sigma}$ in the third line. To conclude, we come back to the definitions of $w_{\sigma}$, $z_{\sigma}$ and $\sigma$
\beqn 
\frac{w^2_{\sigma}}{2\sigma^2}&= &\frac{1}{\sigma^2-1}\log\left(\frac{
\sigma}{1-\epsilon}\right)\geq 2\log(n\epsilon^2)\ ;\\
\frac{z^2_{\sigma}}{2\sigma^2}& = & \frac{\log(2\sigma)\wedge 1}{\sigma^2-1}\geq \frac{\log(2)}{3(\sigma-1)}= \frac{2\log(2)\log(n\epsilon^2)}{\log\big(\frac{1}{1-\epsilon}\big)} \ ;  \\
\sigma - 1 &\leq &  \frac{\log(\frac{1}{1-\epsilon})}{\log(n\epsilon^2)} \ ;\\
w^2_{\sigma}&\leq& \frac{2}{\sigma -1}\log\left(\frac{\sigma}{1-\epsilon}\right) \leq 2+ \frac{2}{\sigma -1}\log\left(\frac{1}{1-\epsilon}\right)\leq 2 + [12\log(n\epsilon^2)]\vee [\log(\frac{1}{1-\epsilon})] \ . 
\eeqn 
This implies that 
\beqn 
 \phi_{0,1}\left(\frac{z_{\sigma}\vee w_{\sigma}}{\sigma}\right)
\leq  \exp\left[- 2\log(n\epsilon^2)\left(1\vee \frac{\log(2)}{\log\big(\frac{1}{1-\epsilon}\big)}\right) \right]\ , 
\eeqn 
which is less than $16 n^{-2}$ since the maximum    over $\epsilon \in [e^2/\sqrt{n}, 1-1/\sqrt{n}]$  is achieved at $\epsilon=1/2$.

Coming back to \eqref{eq:upper_HH_sigma_2}, we conclude that 
\beqn 
\frac{1}{2}H^2(\vartheta_0; \vartheta_1)&\leq &e^{-n^2/2}+ \frac{c'}{n^2}+ c(1-\epsilon)  \frac{\log^2\left(\frac{1}{1-\epsilon}\right)\log(n\epsilon^2)+ \log^5\left(\frac{1}{1-\epsilon}\right) }{n^2 \epsilon^4}  
\\ &\leq & e^{-n^2/2}+\frac{c'}{n^2}+\left\{\begin{array}{cc}
\frac{c}{n}\frac{\log(n\epsilon^2)}{n\epsilon^2} &\text{ if }   \epsilon\leq 1/2\ ;\\
\frac{c}{n^2}&\text{ if }   \epsilon> 1/2\ .\\
                                           \end{array}\right.
\eeqn 
This last expression is less than $1/(2n)$ as soon as long as  $n$ is large enough and $n\epsilon^2$ is large enough, which is ensured if the constant $c_0$ introduced at the beginning of the proof is large enough. This concludes the proof.

\section{Proofs of upper bounds}

\subsection{Proofs for the preliminary estimators : Propositions~\ref{prop:thetamedian} and~\ref{prp:dense} (OSC)} \label{p:prelest}

\begin{proof}[Proof of Proposition~\ref{prop:thetamedian} ]
We prove this result in the OSC model. Consider any $\mu\in \overline{\cM_k}$. As argued in Section \ref{sec:robust}, we have the stochastic bounds
\begin{equation}\label{equ-med}
\xi_{(\lceil n/2\rceil)}\preceq  \thetachapmed- \theta\preceq \xi_{(\lceil n/2\rceil:n-k)}\ ,
 \end{equation}
 where $\xi=(\xi_1,\ldots, \xi_n)$ is a standard Gaussian vector. Hence, we only have to control the deviations of $\xi_{(\lceil n/2\rceil)}$ and of $\xi_{(\lceil n/2\rceil:n-k)}$. 
Then, we apply Lemma \ref{lem:quantile_empirique} with $q=\lceil n/2\rceil $ to obtain
\[
 \P_{\theta,\pi}\Big[ \thetachapmed -\theta +  \overline{\Phi}^{-1}(\tfrac{\lceil n/2\rceil}{n}) \leq  - 3\frac{\sqrt{(n+1) x}}{n}\Big]\leq e^{-x}\ , 
\]
for all $x\leq cn$ (where $c$ is some universal constant). As for the right deviations of $\widehat{\theta}_{med}$, we apply the deviation inequality \eqref{equ1Nico}  to  $\xi_{(\lceil n/2\rceil:n-k)}$ as $k\leq n/10$.  This leads us to
\[
 \P_{\theta,\pi}\Big[ \thetachapmed -\theta +  \overline{\Phi}^{-1}(\tfrac{\lceil n/2\rceil}{n-k}) \geq   3\frac{\sqrt{(n+1) x}}{n-k}\Big]\leq e^{-x}\ ,  
\]
for all $x\leq cn$. Then, Lemma \ref{lem:difference_quantile} ensures that 
\[
|\overline{\Phi}^{-1}(\tfrac{\lceil n/2\rceil}{n-k})|=  |\overline{\Phi}^{-1}(\tfrac{\lceil n/2\rceil}{n-k})- \overline{\Phi}^{-1}(1/2)|\leq \frac{3(k+1)}{2(n-k)}\ .
\]
Similarly, we have $|\overline{\Phi}^{-1}(\tfrac{\lceil n/2\rceil}{n})|\leq 3/(2n)$. We have proved the first result. 

Let us now turn to the moment bound. Starting from \eqref{equ-med}, we get
\beqn
\E_{\theta, \pi }\big[ |\thetachapmed-\theta  |\big]&\leq& \E_{\theta, \pi }\big[ (\thetachapmed-\theta)_+\big]+ \E_{\theta, \pi }\big[ (\theta-  \thetachapmed)_+\big]\\
&\leq & \E\big[ (\xi_{(\lceil n/2\rceil:n-k )})_+\big]+ \E\big[(\xi_{(\lceil n/2 \rceil)})_- \big]\ .
\eeqn 
We have proved above deviation inequalities for these two random variables for probabilities larger than $e^{-c'n}$ (where $c'$ is some universal constant). Write $Z_1= \xi_{(\lceil n/2\rceil:n-k )}+  \overline{\Phi}^{-1}(\frac{\lceil n/2\rceil }{n-k})$ and $Z_{2}= \xi_{(\lceil n/2 \rceil)}+ \overline{\Phi}^{-1}(\frac{\lceil n/2\rceil }{n})$. We deduce from the previous deviation inequalities that 
\beqn 
\E\big[ (Z_1)_{+}\ind_{Z_1\leq c'}]\leq \frac{3\sqrt{\pi (n+1)}}{\sqrt{2}(n-k)} \ ,\quad 
 \E\big[ (Z_2)_{-}\ind_{Z_2\geq -c'}\big]\leq  \frac{3\sqrt{\pi (n+1)}}{\sqrt{2}n}\ .
\eeqn 
It remains to control $\E\big[ (Z_1)_{+}\ind_{Z_1> c'}]$ and $\E\big[ (Z_1)_{+}\ind_{Z_1\leq -c'}]$. Since these two random variables are Lipschitz functions of $\xi$, they follow the Gaussian concentration theorem. In particular, their variance is less than $1$. Also, $Z_1$ and $Z_2$ concentrate well around their medians and around 0 (previous deviation inequality). Thus, their first moments is smaller than a constant. Cauchy-Schwarz inequality then  yields 
\[
 \E\big[ (Z_1)_{+}\ind_{Z_1> c'}]\leq \P^{1/2}[Z_1\geq c']\E^{1/2}[(Z_1)_+^2]\leq ce^{-c''n}\  .
\]
Similarly, we have $\E\big[ (Z_2)_{-}\ind_{Z_2\leq -c'}\big]\leq  ce^{-c''n}$. This concludes the proof. 

\end{proof}

\begin{proof}[Proof of Proposition \ref{prp:dense}]
For any $\theta$ and any $\pi\in \overline{\cM}_{n-1}$, we have 
\[
\xi_{(1)}\preceq Y_{(1)} - \theta \preceq   \xi_1\ ,
\]
where $\xi=(\xi_1,\ldots, \xi_n)$ is a standard normal vector. 
Using the Gaussian tail bound, we derive that 
\[
 \P_{\theta,\pi}\left[(\thetachapmin - \theta) \in [\overline{\Phi}^{-1}(n^{-1})- \overline{\Phi}^{-1}(n^{-2})  ,  2\overline{\Phi}^{-1}(n^{-1})]\right]\geq 1- \frac{2}{n}\ .
\]
By Lemma \ref{lem:quantile}, $2\overline{\Phi}^{-1}(n^{-1})\leq 2\sqrt{2\log(n)}$, whereas Lemma \ref{lem:difference_quantile} ensures that 
$\overline{\Phi}^{-1}(n^{-2}) - \overline{\Phi}^{-1}(n^{-1})\leq \sqrt{\log(n)/2}+O(1)$. Thus, the desired deviation bound holds for $n$ large enough. 
Turning to the moment bound, we have the following decomposition
\beqn 
 \E_{\theta,\pi}\left[|\thetachapmin- \theta| \right]&\leq& \E\left[(\thetachapmin- \theta)_{+}+ ( \theta- \thetachapmin)_{+} \right]\\
 &\leq & 2\sqrt{2\log(n)}+ \E\left[ (\xi_{1}+ \overline{\Phi}^{-1}(1/n))_+\ind_{\xi_{1}+ \overline{\Phi}^{-1}(1/n)\geq 2\sqrt{2\log(n)}}\right]\\&& + \E\left[ (\xi_{(1)}+ \overline{\Phi}^{-1}(1/n) )_-\ind_{ \xi_{(1)}+ \overline{\Phi}^{-1}(1/n) \leq -2\sqrt{2\log(n)}}\right]\ . 
\eeqn 
Let us focus on the first expectation, the second expectation being handled similarly. As a consequence of Cauchy-Schwarz inequality, we have 
\beqn
 \E\left[ (\xi_{1}+ \overline{\Phi}^{-1}(1/n))_+1_{\xi_{1}+ \overline{\Phi}^{-1}(1/n)\geq 2\sqrt{2\log(n)}}\right]&\leq& c\sqrt{\log(n)}\P^{1/2}[\xi_{1}+ \overline{\Phi}^{-1}(1/n)\geq 2\sqrt{2\log(n)}]\\ &\leq& c\sqrt{\frac{\log(n)}{n}}\ ,
\eeqn 
where we used the above deviation inequality. We obtain 
\[
  \E_{\theta,\pi}\left[|\thetachapmin- \theta| \right]\leq 2\sqrt{2\log(n)}+ c\sqrt{\frac{\log(n)}{n}}\ ,
\]
which concludes the proof. 

\end{proof}

\subsection{Range for $\theta$: analysis of  $\thetachapup$ and $\thetachaplow{q_k}$ (gOSC)}

As a preliminary step for the proof of Theorem \ref{th-upperbound-onesided}, we control the deviations of the rough estimators $\thetachapup$ and $\thetachaplow{q_k}$. Recall that the tuning parameter $q_k$ is defined in Theorem \ref{th-upperbound-onesided}.

\begin{lem}[Control of $\thetachaplow{q_k}$]\label{lem:theta_min_q}
There exist an universal constants  $n_0\geq 1$ and $c>0$ such that the following holds for all $n\geq n_0$. For all  $k\in [1,n]$ such that $q_k\leq \frac{3}{10\const}\log n$, $\mu\in \cM_{k}$ and all $\theta\in\R$,   the estimator $\thetachaplow{q_k}$ satisfies 
 \[
 \P_{\theta,\mu}\Big[\thetachaplow{q_k}\geq \theta \Big] \leq\: \frac{1}{n} \ ;\quad \quad 
\E_{\theta,\mu}\Big[(\theta -\thetachaplow{q_k})\ind_{\thetachaplow{q_k}\leq \theta - 2\overline{v}} \Big]\leq  \frac{c}{\sqrt{n}}\ .
 \]
\end{lem}

\begin{proof}[Proof of Lemma \ref{lem:theta_min_q}]
Recall
$q_k= \lfloor \frac{1}{\const}\log\big(\frac{k}{\sqrt{n}}\big) \rfloor_{\mathrm{even}}\wedge q_{\max}$  so that $k \leq e^{2a} n^{4/5}$.
By definition, we have $\thetachaplow{q_k}= \wh{\theta}_{\mathrm{med}}- \overline{v}$, which, thanks to Proposition \ref{prop:thetamedian}, implies that 
\begin{align*}
\P_{\theta,\mu}\left(\thetachaplow{q_k}\geq \theta \right)&\leq \P_{\theta,\mu}(\thetachapmed- \theta \geq \overline{v}) \\
&\leq e^{- \left(\frac{n-k}{3}\overline{v}- \frac{k+1}{2} \right)^2/(n+1)}\leq e^{-c n \overline{v}^2 } \leq e^{-c' n /\log^3 (n)}\ ,
\end{align*}
for some constant $c'>0$ and $n$ large enough (above, we have used that  $k/(n-k)=O(n^{-1/5})$). The first bound follows.
Let us turn to proving the second bound. From \eqref{equ-med}, one has $\thetachaplow{q_k}-\theta \succeq  \varepsilon_{(\lceil n/2\rceil)} -\overline{v}$. As a consequence, 
\begin{align*}
\E_{\theta,\mu}\Big[(\theta -\thetachaplow{q})\ind_{\thetachaplow{q}\leq \theta - 2\overline{v}} \Big]
&\leq \E_{\theta,\mu}\Big[(- \varepsilon_{(\lceil n/2\rceil)}+ \overline{v} )\ind_{\varepsilon_{(\lceil n/2\rceil)}\leq - \overline{v} } \Big]\\
&\leq 2 \: \E_{\theta,\mu}\Big[(-\varepsilon_{(\lceil n/2\rceil)})_+\Big] \leq 2c_1/\sqrt{n}\ ,
\end{align*}
where the last bound is for instance a consequence of Proposition~\ref{prop:thetamedian} for $k=0$.
\end{proof}

\begin{lem}[Control of $\thetachapup$]\label{lem:theta_max}
There exists an universal integer $n_0\geq 1$ such that for any $n\geq n_0$, $\mu\in \cM_{n-1}$ and any $\theta\in\R$, the estimatorv  $\thetachapup$ satisfies 
 \begin{align*} 
\P_{\theta,\mu}\big[\thetachapup < \theta \big]&\leq  \frac{1}{n}\ ;\quad
 \E_{\theta,\mu}\big[(\theta- \thetachapup)_+\big]  \leq  \frac{1}{n}\ ; \quad 
\E_{\theta,\mu}\big[\big(\thetachapup-\theta\big)_+\ind_{\thetachapup-\theta\geq 4\sqrt{\log(n)}}\big]\leq \frac{1}{n^2}\ . 
 \end{align*}

 \end{lem}

\begin{proof}[Proof of Lemma \ref{lem:theta_max}]
The first bound is a slight variation of Proposition~\ref{prp:dense}. As in the proof of that proposition, we start from 
$
 \thetachapup - \theta \succeq \epsilon_{(1)}+2\sqrt{\log(n)}
$
which implies 
\begin{equation*}
\P_{\theta,\mu}[\thetachapup < \theta ] \leq  \P\left[ \epsilon_{(1)}< -  2\sqrt{\log(n)}\right] \leq n \overline{\Phi}(2\sqrt{\log(n)})\leq 1/n\ . 
 \end{equation*}
where we used an union bound  and \eqref{eq:maj-fonctionrepgauss}.

Second, we start from  $\thetachapup-\theta \geq \varepsilon_{(1)}+ 2\sqrt{\log  n}$ and obtain
\begin{align*}
\E_{\theta,\mu}\big[(\theta- \thetachapup)_+\big] &= \int_0^\infty \P_{\theta,\mu}(\theta- \thetachapup\geq t)dt\\
&\leq \int_0^\infty \P(\varepsilon_{(1)} \leq -(t+2\sqrt{\log  n})) dt= n \int_{0}^{\infty}\overline{\Phi}(t+2\sqrt{\log(n)})dt\\
&\leq \frac{n}{2\sqrt{\log  n}} \int_0^\infty \phi(t+2\sqrt{\log  n}) dt = \frac{n}{2\sqrt{\log  n}} \ol{\Phi}(2\sqrt{\log  n})\leq 1/n\ ,
\end{align*}
where we used \eqref{eq:maj-fonctionrepgauss} and $n$ large enough in the last line. 

Finally, since at least one $\mu_i$ is zero, we may assume  without loss of generality that $\mu_1=0$, which implies  $\thetachapup-\theta\leq \eps_1+ 2\sqrt{\log n }$.
 This leads us to
\begin{align*}
\E_{\theta,\mu}\big[\big(\thetachapup-\theta\big)_+\ind_{\thetachapup-\theta\geq 4\sqrt{\log n }}\big] & \leq  \E_{\theta,\mu}\big[\big[\eps_{1}+ 2\sqrt{\log n}\big]\ind_{\epsilon_{1}\geq 2\sqrt{\log n}}\big]\\
&\leq 2\:\E_{\theta,\mu}\big[\eps_1\ind_{\epsilon_{1}\geq 2\sqrt{\log n}}\big] = 2 \phi(2\sqrt{\log n}) \leq \frac{1}{n^2}\ , 
\end{align*}
by integration.
\end{proof}

\subsection{Proof of Theorem \ref{th-upperbound-onesided} (gOSC)}\label{p:th-upperbound-onesided}

In order to ease the notation, we write $\lambda_k$ for $\lambda_{q_k}$. 
We first prove the probability bound for $\wh{\theta}_{q_k}$  and then turn to the moment bound. First, recall that the population function $\psi_{q,\lambda}(u)$ has been defined in such a way that $\psi_{q,\lambda}(u)\in [-1,1]$ 
 for all $u\leq \theta$ and is larger than $1$  for $u$ sufficiently large.
The following lemma quantifies this phenomenon by providing a lower bound for $\psi_{q,\lambda_q}(\theta+v_k^*)$ with 
some $v^*_k$ defined by 
\beq \label{eq:definition_v_k}
v^*_k: = \frac{5k }{(n-k)\lambda_kq_k^2} \text{ if }q_k< q_{\max}\ , \: \text{ and } v^*_k: = \frac{2}{q_k^2 \lambda_k} \log^2\left(\frac{8n}{n-k}\right) \text{ if }q_k= q_{\max} \ . 
\eeq
When $q_k< q_{\max}$, we have $k<e^{-2a} n$.
As a consequence, we easily check that, for any $k\in [e^{2\const} \sqrt{n}, n - 64n^{1-1/(4\const)})$, one has
\beq\label{eq:equiv_v_*}
 v^*_k \leq c   \frac{\log^2\left(1+ \sqrt{\frac{k}{n-k}}\right)}{\log^{3/2}\left(\frac{k}{\sqrt{n}}\right)}\ ,
\eeq
in both cases, where $c$ is a positive universal constant. 

\begin{lem}\label{lem:bias}
Let us consider the function $\psi_{q,\lambda}$ defined by \eqref{eq:definition_psi} and any integer $k\in [e^{2\const} \sqrt{n}, n - 64n^{1-1/(4\const)})$ and $v^*_k$ defined by \eqref{eq:definition_v_k}. Assume $\mu\in \cM_k$. Then, we have 
\beq\label{eq:control_bias}
\psi_{q_k,\lambda_k}(\theta+ v^*_k)>1 + \frac{k}{n}(1+ e^{q_k\lambda_kv^*_k})\ .
\eeq 
Besides, if $q_k< q_{\max}$, we have for any $\omega \geq 1$,
\[
\psi_{q_k,\lambda_k}(\theta+\omega v^*_k)> 1 + \omega\frac{k}{n}(1+ e^{q_k\lambda_kv^*_k})\ .
\]
\end{lem}

The second lemma controls  the simultaneous deviations of the statistics $\wh{\psi}_{q_k,\lambda_k}(u)$, $u\in\R$, around their expectations.
 
\begin{lem}\label{lem:concentration_psi}
Let us consider the functions $\psi_{q,\lambda}$ and $\wh{\psi}_{q,\lambda}$ defined by \eqref{eq:definition_psi} and \eqref{eq:definition}, respectively, for some arbitrary $\mu\in \cM$. Fix any $t>0$, any $\lambda>0$ and any positive even integer $q$. Then, with probability higher than $1-1/t^2$, we have 
 \beqn 
|\wh{\psi}_{q,\lambda}(u) - \psi_{q,\lambda}(u)| \leq  \frac{t}{\sqrt{n}}q^{3/2}\exp\Big[\lambda^2 \frac{q^2}{2}+ q\log(3+2\sqrt{2}) - \lambda (\theta - u)_{+} + \lambda q (u-\theta)_+\Big]\ ,
\eeqn 
simultaneously over all $u\in \mathbb{R}$. 
\end{lem}
Let us now define
\beq\label{eq:equation_tk}
t_k= \frac{e^{\const q_k}}{q_k^{3/2}}e^{-q_k^2\lambda_k^2/2- q_k\log(3+2\sqrt{2})}=\frac{e^{2\const q_k/3 }}{q_k^{3/2}} \geq \frac{e^{-8\const/3}}{q_k^{3/2}}\bigg(\frac{k}{\sqrt{n}}\bigg)^{2/3} \ , 
\eeq
by definition of $\const$, and because  
 $\lambda_k=(2/q_k)^{1/2}$ and $q_k\geq \const^{-1}\log(k/\sqrt{n}) - 4$. 
It readily follows from Lemma \ref{lem:concentration_psi} that, with probability higher than $1-t_k^{-2}$, we have 
\beq\label{eq:control_underestimation}
 \sup_{u\leq \theta} \wh{\psi}_{q,\lambda_k}(u)\leq \sup_{u\leq \theta} \psi_{q,\lambda_k}(u) + \frac{e^{\const q_k}}{n^{1/2}}\leq 1 + \frac{e^{\const q_k}}{n^{1/2}}\ .
\eeq
Together with Lemma \ref{lem:bias}, this also  leads to  (on the same event)
\beq\label{eq:control_overestimation}
 \wh{\psi}_{q_k,\lambda_k}(v^*_k+ \theta) \geq \psi_{q_k,\lambda_k}(v^*_k+ \theta) - \frac{e^{\const q_k}}{n^{1/2}}e^{q_k\lambda_k v^*_k}> 1 +  \frac{e^{\const q_k}}{n^{1/2}}\ , 
\eeq
since $\const q_k\leq \log(k/\sqrt{n})$. 
Thanks to  \eqref{eq:control_underestimation} and \eqref{eq:control_overestimation},  we have proved  that  
\beq \label{eq:control_deviation_theta_q_k}
\P_{\theta,\mu}\Big[\wh{\theta}_{q_k}\in  [\theta, \theta + v^*_k]\Big]\geq 
 1 - t_{k}^{-2} - \P_{\theta,\mu}\Big[\thetachaplow{q_k}\geq \theta + v_k^*\Big] -  \P_{\theta,\mu}\Big[\thetachapup<\theta \Big]\ . 
\eeq
Note that, for the probability bound, the preliminary estimators $\thetachaplow{q_k}$ and $\thetachapup$ do not help at all and we would have obtained a similar result  had we simply taken $\wh{\theta}_{\min}=-\infty$ and $\thetachapup=+\infty$ in which case the two last terms in the above bound would be equal to zero. With our choice of preliminary estimators, Lemmas \ref{lem:theta_min_q} and \ref{lem:theta_max} ensure that the two probabilities in the right hand side of \eqref{eq:control_deviation_theta_q_k} are small compared to $1/n$. We have proved that 
\beq \label{eq:deviation_importante}
\P_{\theta,\mu}\Big[\wh{\theta}_{q_k}\in  [\theta, \theta + v^*_k]\Big]\leq c_3 \left(\frac{k}{\sqrt{n}}\right)^{-4/3}\log^{3}\bigg(\frac{k}{\sqrt{n}}\bigg) \ , 
\eeq
for some constant $c_3$, which in view of the bound \eqref{eq:equiv_v_*} of $v^*_k$ leads to the desired  probability bound \eqref{eq:resultthetachap}.

\bigskip

 \noindent 
Let us turn to prove the moment bound \eqref{eq:risk_theta_tilde}. We consider separately $\E_{\theta,\mu}[(\wh{\theta}_{q_k}-\theta)_+]$ and $\E_{\theta,\mu}[(\theta- \wh{\theta}_{q_k})]_+$.\\

\noindent 
 {\bf Step 1: Control of $\E_{\theta,\mu}[(\wh{\theta}_{q_k}-\theta)_+]$}. The analysis is divided into two cases, depending on the value of $k$.

 \noindent 
{\it Case 1: $q_k\geq 0.3 \const^{-1}\log n$ which implies $k\geq  n^{4/5}$}.
Since $\wh{\theta}_{q_k}\leq \thetachapup$, we have   the following risk  decomposition. 
\beqn 
   \E_{\theta,\mu}[(\wh{\theta}_{q_k}-\theta)_+] & \leq& v^*_k +  \E_{\theta,\mu}\big[\big(\thetachapup-\theta\big)_+\ind_{\wh{\theta}_{q_k}-\theta \geq v^*_k}\big]\\
   & \leq & v^*_k + 4\sqrt{\log n}\:\P_{\theta,\mu}\big[\wh{\theta}_{q_k}-\theta \geq v^*_k\big]+  \E_{\theta,\mu}\big[\big(\thetachapup-\theta\big)_+\ind_{\thetachapup-\theta\geq 4\sqrt{\log(n)}}\big]\ .
 \eeqn  
 The second term is less than $\log^{7/2}(n)(k/\sqrt{n})^{-4/3}$ which is small in front of $v_k^*$ since $k\geq  n^{4/5}$. Finally, the last term is small compared to $1/n$ by Lemma \ref{lem:theta_max}. We have  proved that $\E_{\theta,\mu}[(\wh{\theta}_{q_k}-\theta)_+]  \lesssim v^*_k$ (for $n$ large enough).

 \medskip 
 
 \noindent 
{\it Case 2: $q_k< 0.3 \const^{-1}\log n$.} Define the event $\cA= \{\thetachaplow{q_k}\leq \theta\}$.  Since $\wh{\theta}_{q_k}\leq \thetachapup$, we have   the following decomposition. 
\beqn
  \E_{\theta,\mu}[(\wh{\theta}_{q_k}-\theta)_+]  &\leq &\E_{\theta,\mu}\big[\big(\wh{\theta}_{q_k}-\theta\big)_+\ind_{\cA}\big] + \E_{\theta,\mu}\big[\big(\thetachapup-\theta\big)_+\ind_{\cA^c}\big]  \nonumber \\
  &\leq & \E_{\theta,\mu}\big[\big(\wh{\theta}_{q_k}-\theta\big)_+\ind_{\cA}\big] + 4\sqrt{\log n}\:\P_{\theta,\mu}[\cA^c] +  \E_{\theta,\mu}\big[\big(\thetachapup-\theta\big)_+\ind_{\thetachapup-\theta\geq 4\sqrt{\log n}}\big]  \ . 
 \eeqn 
 By Lemma \ref{lem:theta_max}, the third term in the rhs has been proved to be small compared $1/n$. By Lemma \ref{lem:theta_min_q}, $\sqrt{\log n}\:\P_{\theta,\mu}[\cA^c]$ is small compared to $\sqrt{\log n}/n$ which in turn is smaller than $v^*_k$. Hence, we only need to prove that $\E_{\theta,\mu}\big[\big(\wh{\theta}_{q_k}-\theta\big)_+\ind_{\cA}\big]$ is of order at most $v^*_k$. 
 By integration, it suffices to prove that, for all $\omega\geq 1$,
\beq\label{eq:deviation_faible}
\P_{\theta,\mu}[(\wh{\theta}_{q_k} -\theta)\ind_{\cA} >  \omega v^*_k \big] \leq (\omega t_k)^{-2}\ ,
\eeq
Fix some $\omega\geq 1$. Since  $q_k < q_{\max}- 2$,   Lemma \ref{lem:bias} ensures that 
\[
 \psi_{q_k,\lambda_k}(\theta+ \omega v^*_k)  >  1+ \omega \frac{k}{n}(1+ e^{q_k\lambda_kv^*_k})\ .  
\]
From Lemma \ref{lem:concentration_psi} with $t= \omega   t_k$, we deduce as in \eqref{eq:control_overestimation} that, with probability higher than $1-(\omega t_k)^{-2}$
\[
\wh{\psi}_{q_k,\lambda_k}(\theta + \omega v^*_k) \geq \psi_{q_k,\lambda_k}(\theta+ \omega v^*_k) - \omega \frac{k}{n}e^{q_k\lambda_k v^*_k}> 1 +  \frac{k}{n}\geq 1 + \frac{e^{\const q_k}}{\sqrt{n}}\ ,  
\]
Together with $\cA$, this event enforces that $\wh{\theta}_{q_k}\leq \theta + \omega v^*_k$. We have proved \eqref{eq:deviation_faible}. This entails $
\E_{\theta,\mu}[(\wh{\theta}_{q_k}-\theta)_+]\leq c v^*_k$ for some universal constant $c>0$.
  
 \bigskip 
 
 \noindent 
 {\bf Step 2: Control of $\E_{\theta,\mu}[(\theta- \wh{\theta}_{q_k})_+]$}. Define the estimator 
\begin{equation*}
 \widetilde{\theta}_q= \inf\bigg\{u\in [\thetachaplow{q},+\infty)\, :\,  \wh{\psi}_{q,\lambda_q}(u)> 1 + \frac{e^{\const q}}{\sqrt{n}}\bigg\}\ .
\end{equation*}
It follows from this definition that  $\wh{\theta}_q\geq \widetilde{\theta}_q\wedge  \thetachapup$ and
\[
  \E_{\theta,\mu}[(\theta - \wh{\theta}_{q_k})_+]\leq  \E_{\theta,\mu}\big[(\theta - \thetachapup)_+\big] +   \E_{\theta,\mu}\big[\big(\theta- \widetilde{\theta}_{q_k}\big)_+\big]\ .
  \]
By Lemma \ref{lem:theta_max}, the first term in the rhs is small compared to $1/n$ and we focus on the second expectation.

 \medskip 
 
 \noindent 
 {\it Case 1: $q_k> 0.3 \const^{-1}\log(n)$ which implies $k\geq n^{4/5}$}. Fix any $\omega\geq 1$ and define $v_{\omega}= \log(\omega)/\lambda_k$. It follows from Lemma \ref{lem:concentration_psi}, that, with probability higher than $1-(t_k\omega)^{-2}$, we have simultaneously  over all $u\geq v_{\omega}$, 
 \[
  \wh{\psi}_{q_k,\lambda_k}(\theta-u) \leq \psi_{q_k,\lambda_k}(\theta-u) + \omega \frac{e^{\const q_k}}{\sqrt{n}} e^{-\lambda_k u}\leq 1 +  \omega \frac{e^{\const q_k}}{\sqrt{n}}e^{-\lambda_k v_{\omega}}= 1+ \frac{e^{\const q_k}}{\sqrt{n}} \ ,
 \]
where we used $|\psi_{q_k,\lambda_k}(\theta-u)|\leq 1$ for all $u>0$. With probability higher than $1-(t_k\omega)^{-2}$,  $\widetilde{\theta}_{q_k}$ is therefore higher than $\theta - \log(\omega)/\lambda_k$. 
Integrating this last bound leads to 
\beq \label{eq:upper_moment_lower_grand_k}
 \E_{\theta,\mu}[(\theta - \widetilde{\theta}_{q_k})_+]\leq  \frac{1 }{2 t_k^2 \lambda_k} \lesssim v_k^*\ ,
\eeq
since  $k \geq  n^{4/5}$ (for $n$ large enough).

 \medskip 

\noindent 
{\it Case 2:  $q_k< 0.3 \const^{-1}\log(n)$}.  Define the event  $\cB_k= \{\thetachaplow{q_k}\geq \theta - v_{\min,k} \}$ with $v_{\min,k}= \sqrt{2} \frac{\pi^2}{72\lambda_kq_k^2} \geq 2 \overline{v}$ (where $\ol{v}$ is defined along with $\thetachaplow{q}$). Since $\widetilde{\theta}_{q_k}\geq \thetachaplow{q_k}$, we have the following decomposition. 
\beq
  \E_{\theta,\mu}[(\theta - \widetilde{\theta}_{q_k})_+]\leq 
   \E_{\theta,\mu}\big[\big(\theta- \thetachaplow{q_k}\big)_+\ind_{\mathcal{B}_k^c}\big] +  \E_{\theta,\mu}\big[\big(\theta- \widetilde{\theta}_{q_k}\big)_+\ind_{\mathcal{B}_k}\big]
  \label{eq:decomposition_moment_lower}\ .
\eeq
We start by considering the first term in the right hand side.  Since $\overline{v}\leq v_{\min,k}/2$,  we rely on Lemma \ref{lem:theta_min_q} to derive that
\beq\label{eq:upper_moment_lower_petit_k2}
 \E_{\theta,\mu}\left[(\theta-\thetachaplow{q_k})_+\ind_{ \cB_k^{c}}\right]\leq\E_{\theta,\mu}\left[(\theta-\thetachaplow{q_k})_+\ind_{ \thetachaplow{q_k}< \theta - 2\overline{v} }\right]\leq c/\sqrt{n}\ . 
\eeq
 
 We now turn to $\E_{\theta,\mu}\big[\big(\theta- \widetilde{\theta}_{q_k}\big)_+\ind_{\mathcal{B}_k}\big]$ in 
\eqref{eq:decomposition_moment_lower}. In comparison to the previous case, this bound is slightly more involved and we rely on the explicit expression of Chebychev Polynomials. For any $v>0$, we have 
\beqn
\psi_{q_k,\lambda_k}(\theta- v) \leq \frac{k}{n}+ \bigg(1-\frac{k}{n}\bigg) \cos( q_k \arg\cos(2e^{-\lambda_k v}-1))\ . 
\eeqn 
 Observe that 
$1- e^{-t}\in [t/2,t]$ for $t\in [0,\log 2]$,  $\cos t\leq 1-t^2/4$ for all $t\in [0,\pi/2]$ and $\arg\cos(1-t)\in [\sqrt{2t},2\sqrt{t}]$ for $t\in [0,1]$.  As a consequence, for $v\leq v_{\min,k}$,  one has 
\[
\psi_{q_k,\lambda_k}(\theta- v) \leq 1 - \bigg(1-\frac{k}{n}\bigg)  \frac{q_k^2\lambda_k v}{2}.
\]
Fix any $\omega>1$. Thanks to deviation bound in Lemma \ref{lem:concentration_psi} we derive that, with probability higher than $1-(\omega t_k)^{-2}$, we have 
\[
\wh{\psi}_{q_k,\lambda_k}(\theta- v) \leq \psi_{q_k,\lambda_k}(\theta- v) +\omega \frac{e^{\const q_k}}{\sqrt{n}}e^{-\lambda_k v}\leq 1+ \frac{e^{\const q_k}}{\sqrt{n}} + (\omega-1)\frac{k}{n}
- \bigg(1-\frac{k}{n}\bigg)  \frac{q_k^2\lambda_k v}{2}\ ,  
\]
simultaneously for all $v\in [0, v_{\min,k}]$. In the second inequality, we used that $\const q_k\leq \log(k/\sqrt{n})$. As a consequence, $\wh{\psi}_{q_k,\lambda_k}(\theta- v)\leq 1+ e^{\const q_k}/\sqrt{n}$ for all $v$ in the (possibly empty) interval 
\[
v\in \left[ \frac{2(\omega-1)k}{(n-k)q^2_k\lambda_k},v_{\min,k}\right]\ . 
\]
Since we work under the event $\cB_k=\{\thetachaplow{q_k}\geq \theta - v_{\min,k} \}$, this implies that 
\[
 \P_{\theta,\mu}\Big[(\theta - \widetilde{\theta}_{q_k})_+\ind_{\cB_k} > \frac{2(\omega-1)k}{(n-k)q^2_k\lambda_k} \Big]\leq \frac{1}{\omega^2 t_k^2}\ , 
\]
for all $\omega\geq 1$. Integrating this deviation bound, we conclude that  
\[
\E_{\theta,\mu}\big[(\theta - \widetilde{\theta}_{q_k})_+\ind_{\cB_k}\big] \leq 2 
 \frac{k}{(n-k)t_k^2q^2_k\lambda_k}\lesssim v_k^*\ .  
\]
since $t_k\gtrsim 1$.

Together with \eqref{eq:upper_moment_lower_grand_k} and \eqref{eq:upper_moment_lower_petit_k2}, we have proved that  $\E_{\theta,\mu}[(\theta - \wh{\theta}_{q_k})_+]\lesssim v_k^*$, which  concludes the proof of the theorem.

\begin{proof}[Proof of Lemma \ref{lem:bias}]
Let us first prove the following inequality:
\beq \label{eq:lower-g_q}
g_q(t)= \cosh(q \arg \cosh( 2e^{t}-1)\geq \max\left[1+ q^2 t , \frac{1}{2}e^{q\sqrt{2t}} \right].
\eeq
Since $\cosh(t)\leq e^{t^2/2}$ for all $t>0$ (compare the power expansions), we have  $\cosh(\sqrt{2t})\leq e^{t}\leq 2e^{t}-1$, implying that 
$g_q(t)\geq \cosh(q \sqrt{2t})$. Then, we use that $\cosh(t)\geq 1+t^2/2$ and $\cosh(t)\geq e^{t}/2$ to conclude.

For a $k$-sparse vector $\mu$, we have already observed in \eqref{relationPsi} 
 that, for all $t>0$, 
\[
  \psi_{q,\lambda}(\theta+t) \geq - \frac{k}{n}+ \frac{n-k}{n}g_q(\lambda t)\ .
\]
The analysis is divided into two cases, depending on the value of $k$.

\medskip 

\noindent 
{\it Case 1:  $q_k< q_{\max}$}. 
For any $t>0$, it follows from \eqref{eq:lower-g_q} that 
$g_q(t)\geq 1+ q^2 t$
\[
  \psi_{q_k,\lambda_k}(\theta+ t)\geq -\frac{k}{n}+ \frac{n-k}{n}(1+ q_k^2\lambda_kt) = 1 - 2\frac{k}{n} + \frac{k}{n} \frac{(n-k)q_k^2\lambda_k t}{k} \ .
\]
If we choose $t= \omega v^*_k$ with $\omega \geq 1$, we have  
\[
  \psi_{q_k,\lambda_k}(\theta+ t)\geq 1 + (5\omega -2) \frac{k}{n}\geq 1 + 3\omega \frac{k}{n}\ . 
\]
Finally, we have  \[
\exp(\lambda_k v^*_k q_k)= \exp\bigg[\frac{5k}{(n-k)q_k}\bigg]\leq \exp\bigg[\frac{5k}{2(n-k)}\bigg] < 2\ ,
                  \]
where we used in the last inequality the fact that $q_k< q_{\max}$ which implies $k
<e^{-2a} n\leq n/6$. 
This concludes the first part of the proof. 

\medskip 

\noindent 
{\it Case 2: $q_k= q_{\max}$, that is  $k\in [\sqrt{n}e^{\const q_{\max}},n- 64n^{1-1/(4\const)})$}. Recall $v^*_k = \frac{2}{q_k^2 \lambda_k} \log^2(8n/(n-k))$. Together with \eqref{eq:lower-g_q}, this yields 
\beqn 
\psi_{q_k,\lambda_k}(\theta+v^*_k)&\geq &-\frac{k}{n} + \frac{n-k}{2n}\exp\bigg[q_k \sqrt{2\lambda_k v^*_k}\bigg] \geq -1+   \frac{n-k}{2n}\exp\bigg[q_k \sqrt{2\lambda_k v^*_k}\bigg]\\
&\geq & -1 + 4
\exp\bigg[\frac{1}{2}q_k \sqrt{2\lambda_k v^*_k}\bigg]\geq 2 +\exp\bigg[\frac{1}{2}q_k \sqrt{2\lambda_k v^*_k}\bigg] \ , 
\eeqn 
because $\exp\big[\frac{1}{2}q_k \sqrt{2\lambda_k v^*_k}\big] \geq 8n /(n-k)$ by definition of $v^*_k$.
Next, note that 
\[
  q_k= q_{\max}= \lfloor \frac{1}{2\const} \log(n)\rfloor_{even} -2 \geq \frac{1}{2\const} \log(n) -4\ .
\]
Furthermore, the condition $k\leq n- 64n^{1-1/(4\const)}$ has been chosen in such that a way that 
\[
 q_k \geq 2 \log\bigg(\frac{8n}{n-k}\bigg)\ , 
\]
 which implies that $\lambda_k v^*_k\leq 1/2$. This allows to conclude that 
 \[
\psi_{q_k,\lambda_k}(\theta+v^*_k)>
1 + \frac{k}{n}(1+ e^{q_k \lambda_k v^*_k})\ .
 \]
\end{proof}

\begin{proof}[Proof of Lemma~\ref{lem:concentration_psi}]
At $u=\theta$, simple computations lead to  $\var{\wh{\eta}_{\lambda}(\theta)}\leq  e^{\lambda^2}/n$. It then follows from Chebychev's inequality that 
\[
 \P_{\theta,\mu}\Big[|\wh{\eta}_{\lambda}(\theta) - \eta_{\lambda}(\theta)|\geq \frac{t}{\sqrt{n}}e^{\lambda^2/2} \Big]\leq \frac{1}{t^2}\ , 
\]
for all $t>0$. For a general $u\in\R$, observe that ${\eta}_\lambda(u)$ (resp. $\wh{\eta}_\lambda(u)$) is a simple transformation of $\eta_{\lambda}(\theta)$ (resp. $\wh{\eta}_{\lambda}(\theta)$):
$$\wh{\eta}_\lambda(u) = \wh{\eta}_\lambda(\theta)e^{\lambda(u-\theta)}~~\text{and}~~ \eta_\lambda(u) =  \eta_\lambda(\theta)e^{\lambda(u-\theta)}\ . $$
This entails
\[
 \P_{\theta,\mu}\Big[\exists u\in\R,\:|\wh{\eta}_{\lambda}(u) - \eta_{\lambda}(u)|\geq \frac{t}{\sqrt{n}}e^{\lambda^2/2}e^{\lambda(u-\theta)} \Big]\leq \frac{1}{t^2}\ .
\]

Then, taking an union bound over all $j=1,\ldots, q$, we obtain that, for any $t>0$, we have
\beq\label{eq:control_eta_hat}
 |\wh{\eta}_{\lambda j}(u) - \eta_{\lambda j}(u)|\leq t\sqrt{\frac{q}{n}}e^{(\lambda j)^2/2 + \lambda j(u - \theta)}  \ , \quad \text{ for all }u\in \mathbb{R}\text{ and }j=1,\ldots, q\ ,
\eeq
with probability higher than $1-1/t^2$. Then, we rely on the upper bound \eqref{eq:upper_a_j_q} of the coefficients $|a_{j,q}|$ in the definition of $\wh{\psi}_{q,\lambda}$ to obtain 
\beqn 
|\wh{\psi}_{q,\lambda}(u) - \psi_{q,\lambda}(u)|\leq \frac{t}{\sqrt{n}}q^{3/2}\exp\Big[\lambda^2 \frac{q^2}{2}+ q\log(3+2\sqrt{2}) - \lambda (\theta - u)_{+} + \lambda q (u-\theta)_+\Big]\ ,
\eeqn 
simultaneously over all $u\in \mathbb{R}$ with probability higher than $1-1/t^2$. 
\end{proof}

\subsection{Proof of Theorem \ref{thm:adaptation} (gOSC)}\label{p:thm:adaptation}
Consider any $(\theta,\mu)\in \R\times \cM$ and denote $k= \|\mu\|_0$, which is such that $k\leq n-1$ by assumption. Let us denote
$$
q_*=\left\{\begin{array}{ll}
0 & \mbox{ if $k< e^{2\const}\sqrt{n}$\ ;}\\
 q_k & \mbox{  if $k\in [e^{2\const}\sqrt{n}, n-64n^{1-1/4\const})$\ ;}\\
 q_{\max+2} & \mbox{ else}\ .
\end{array}\right.
$$
We call $\wh{\theta}_{q_*}$ the oracle estimator 
because $\wh{\theta}_{q_*}$ has been shown to achieve the desired risk bounds (see  Propositions~\ref{prop:thetamedian} and~\ref{prp:dense} and Theorem~\ref{th-upperbound-onesided}). Let us also underline that 
$q \in \{0,\dots,q_{\max}+2\} \mapsto\delta_q$ is increasing, because $x\in [1,\infty)\mapsto \const x - (3/2)\log x$ is also increasing.

We start by proving the probability bound \eqref{eq:result_adaptation1_deviation}. We first assume that $q_*< q_{\max}$ and consider afterwards the case $q_*\geq q_{\max}$. 
Consider the event $\cA=\cap_{q\geq  q_*}\{|\wh{\theta}_q-\theta|\leq \delta_{q}/2\}$. Under the event $\cA$, it follows from triangular inequality and the fact that the sequence $\delta_q$ is increasing that $\wh{q}\leq q_*$. Relying again on triangular inequality and the definition of $\wh{q}$, we obtain
\[
 |\wh{\theta}_{\wh{q}} - \theta|\leq |\wh{\theta}_{\wh{q}} - \wh{\theta}_{q_*}|+ |\wh{\theta}_{q_*} - \theta|\leq \frac{3}{2}\delta_{q_*}.
\]
We deduce 
\[
\P_{\theta,\mu}\bigg[|\wh{\theta}_{\wh{q}} - \theta|> \frac{3}{2}\delta_{q_*} \bigg]\leq \P_{\theta,\mu}[\cA^c] \leq   \sum_{q\geq  q_*} \P_{\theta,\mu}\Big[|\wh{\theta}_q-\theta|> \delta_{q}/2\Big]\ . 
\]
Since for any $k < k'$, a $k$-sparse vector is also a $k'$-sparse vector, we can apply the deviation bounds \eqref{eq:control_deviation_theta_q_k} and   \eqref{eq:deviation_importante} in the 
proof of Theorem \ref{th-upperbound-onesided}
(and the definition \eqref{eq:equation_tk} of $t_k$)  to all estimators $\wh{\theta}_q$ with $q=q_*,\ldots, q_{\max}$. For such $q$, we obtain 
 $ \P_{\theta,\mu}[|\wh{\theta}_q-\theta|> \delta_{q}/2]\lesssim e^{-4\const q/3} q^3$. 
 
  Proposition \ref{prp:dense} also  enforces that  $ \P_{\theta,\mu}[|\wh{\theta}_{q_{\max}+2}-\theta|> \delta_{q_{\max}+2}/2]\lesssim 1/n$. 
We conclude that 
\[
\P_{\theta,\mu}\left[|\wh{\theta}_{\ad} - \theta|\geq \frac{3}{2}\delta_{q_*}\right]\lesssim e^{-4\const q_*/3} q_*^3 + \frac{1}{n}\ , 
\]
which leads to the desired result. 
For $q_*= q_{\max}$, it follows again from \eqref{eq:deviation_importante} in the proof of Theorem \ref{thm:adaptation} that 
\[
 \P_{\theta,\mu}\Big[|\wh{\theta}_{q_{\max}}-\theta|\geq  v^*_k \Big] \lesssim \left(\frac{k}{\sqrt{n}}\right)^{4/3}\log^3(\frac{k}{\sqrt{n}})\ , \text{ where } v^*_k= \frac{4}{\sqrt{2}q_{\max}^{3/2}}\log^2\left(\frac{8n}{n-k}\right)\ .
\]
We consider two subcases: (i) $v^*_k\leq 2\sqrt{2\log(n)}= \delta_{q_{\max}+2}/2$ and (ii) $v^*_k> 2\sqrt{2\log(n)}$. Under (i), the event $\cA'= \{|\wh{\theta}_{q_{\max}}-\theta|\vee |\wh{\theta}_{q_{\max}+2}-\theta|\leq \delta_{q_{\max}+2}/2\}$ has large probability and ensures that $\wh{q}\leq q_{\max}$, which in turn implies that $|\wh{\theta}_{\ad}-\theta|$ is smaller than $v^*_k+ \delta_{q_{\max}}\lesssim v^*_k$. Under (ii), we simple use $|\wh{\theta}_{\ad}-\theta|\leq \delta_{q_{\max}+2}+ |\wh{\theta}_{q_{\max+2}}-\theta|$ which is less than  $3\delta_{q_{\max}+2}/2\lesssim v^*_k$ with probability higher than $1-c/n$ by Proposition \ref{prp:dense}. Finally, the case  $q_*= q_{\max}+2$ is handled similarly: we use $|\wh{\theta}_{\ad}-\theta|\leq \delta_{q_{\max}+2}+ |\wh{\theta}_{q_{\max+2}}-\theta|$, which is less than $3\delta_{q_{\max}+2}/2$ with probability higher than $1-c/n$ by Proposition \ref{prp:dense}.

\bigskip

Let us turn to the moment bound. We decompose the risk in a sum of two terms depending on the value of $\wh{q}$. 
\beq\label{eq:decomposition_risque}
\E_{\theta,\mu}\big[|\wh{\theta}_{\ad}-\theta|\big] = \E_{\theta,\mu}\big[|\wh{\theta}_q-\theta|\ind_{\wh{q}\leq q_*}\big]+\E_{\theta,\mu}\big[|\wh{\theta}_q-\theta|\ind_{\wh{q}>q_*}\big]\ .
\eeq
For $\wh{q}< q_*$, it follows from triangular inequality and the  definition of  $\wh{q}$ that  
\[
|\wh{\theta}_q-\theta|\ind_{\wh{q}=q}\leq |\wh{\theta}_{q_*}-\wh{\theta}_q|\ind_{\wh{q}=q} + |\wh{\theta}_{q_*}-\theta|\ind_{\wh{q}=q}\leq \delta_{q_*}\ind_{\wh{q}=q} +  |\wh{\theta}_{q_*}-\theta|\ind_{\wh{q}=q} \ . 
\]
Summing all these terms, we arrive at
\[
  \sum_{q=0}^{q_*}\E_{\theta,\mu}\Big[|\wh{\theta}_q-\theta|\ind_{\wh{q}=q}\Big]\leq \delta_{q_*}+ \E_{\theta,\mu}[|\wh{\theta}_{q_*}-\theta|]\ .
\]
This expectation has been studied in Proposition  \ref{prp:dense} and Theorem \ref{th-upperbound-onesided}, which leads us to 
\beq\label{eq:sous_estimation}
\E_{\theta,\mu}\Big[|\wh{\theta}_q-\theta|\ind_{\wh{q}\leq q_*}\Big]\lesssim \frac{\log^{2}\big(1+ \sqrt{\frac{k}{n-k}}\big)}{\log^{3/2}\big(1+ (\frac{k}{\sqrt{n}})^{2/3}\big)}\ . 
\eeq

Turning to the second sum in the decomposition \eqref{eq:decomposition_risque}, we first assume  that  $q_*<q_{\max}$
define $\widetilde{q}_+=\max\{q': |\wh{\theta}_{q'}-\theta|\geq \delta_{q'}/2\}$. By definition of $\delta_q$, one has $\widetilde{q}_+\geq \wh{q}-2$ under the event $\wh{q}>q_*$. Then, we deduce that
\beqn 
|\wh{\theta}_q-\theta|\ind_{\wh{q}=q}&\leq& \sum_{q'=q-2}^{q_{\max}+2}|\wh{\theta}_q-\theta|\ind_{\wh{q}=q}\ind_{\widetilde{q}_+=q'}\\
& \leq &  |\wh{\theta}_{q}-\theta| \ind_{\wh{q}=q}\ind_{\widetilde{q}_+=q-2} + \sum_{q'=q}^{q_{\max}}\Big[ |\wh{\theta}_q-\wh{\theta}_{q'+2}|+  |\wh{\theta}_{q'+2}-\theta| \Big]\ind_{\wh{q}=q}\ind_{\widetilde{q}_+=q'}  \\ && 
+ \Big[ |\wh{\theta}_q-\wh{\theta}_{q_{\max}+2}|+  |\wh{\theta}_{q_{\max}+2}-\theta| \Big]\ind_{\wh{q}=q}\ind_{\widetilde{q}_+=q_{\max}+2}  \\
&\leq & \frac{3}{2}\sum_{q'=q-2}^{q_{\max}}\delta_{q'+2}\ind_{\wh{q}=q}\ind_{\widetilde{q}_+=q'}  
+ \Big[ \delta_{q_{\max}+2} +  |\wh{\theta}_{q_{\max}+2}-\theta| \Big]\ind_{\wh{q}=q}\ind_{\widetilde{q}_+=q_{\max}+2}\ ,
\eeqn 
where we used again the definition of $\widetilde{q}_+$ and $\wh{q}$. Summing the above bound over all even $q>q_*$ leads to 
\beqn 
\sum_{q=q_*+2}^{q_{\max}}\E_{\theta,\mu}\Big[|\wh{\theta}_q-\theta|\ind_{\wh{q}=q}\Big]&\leq &   \frac{3}{2}\sum_{q=q_*+2}^{q_{\max}}\delta_{q+2}\P_{\theta,\mu}[\widetilde{q}_+=q]+     \E_{\theta,\mu}\Big[\big(\delta_{q_{\max+2}} +|\wh{\theta}_{q_{\max}+2}-\theta| \big)\ind_{\widetilde{q}_+=q_{\max}+2} \Big]\\
&\leq &\frac{3}{2}\sum_{q=q_*}^{q_{\max}}\delta_{q+2}\P_{\theta,\mu}[|\wh{\theta}_q-\theta|\geq \frac{\delta_{q}}{2}]+     3\E_{\theta,\mu}\Big[|\wh{\theta}_{q_{\max}+2}-\theta| \ind_{|\wh{\theta}_{q_{\max}+2}-\theta|\geq \delta_{q_{\max}+2}/2} \Big]\ .
\eeqn 
As explained earlier, we know that 
\[
\P_{\theta,\mu}[|\wh{\theta}_q-\theta|\geq \frac{\delta_q}{2}]\lesssim e^{-4\const q/3} q^3\quad \text{ for }q=q_*,\ldots, q_{\max}\ ,   
\]
and $\P_{\theta,\mu}[|\wh{\theta}_{q_{\max}+2}-\theta|\geq (\delta_{q_{\max}+2})/2]\lesssim 1/n$. Since $\delta_{q_{\max}+2}\lesssim \sqrt{\log(n)}$ 
\[
 \frac{3}{2}\sum_{q=q_*}^{q_{\max}}\delta_{q+2}\P_{\theta,\mu}\big[|\wh{\theta}_q-\theta|\geq \frac{\delta_q}{2}\big]\lesssim  \frac{1}{\sqrt{n}}\sum_{q=q_*}^{q_{\max}-2} q^{3/2}e^{-\const q/3} +  \frac{1}{\sqrt{n}}
 \lesssim \frac{1}{\sqrt{n}}\ . 
\]
Arguing as in the proof of Lemma \ref{lem:theta_max}, we obtain that the second term 
\[
\E_{\theta,\mu}\Big[|\wh{\theta}_{q_{\max}+2}-\theta| \ind_{|\wh{\theta}_{q_{\max}+2}-\theta|\geq \delta_{q_{\max}+2}/2} \Big]\lesssim \frac{1}{n} \ . 
\]
We have proved
\[
\E_{\theta,\mu}\Big[|\wh{\theta}_q-\theta|\ind_{\wh{q}> q_*}\big]\lesssim \frac{1}{\sqrt{n}}\ . 
\]
We now turn to the case $q_*=q_{\max}$. We only need to bound 
$\E_{\theta,\mu}[|\wh{\theta}_{q_{\max}+2}-\theta|\ind_{\wh{q}= q_{\max}+2}]$.
The event $\wh{q}=q_{\max}+2$  only occurs if either $|\wh{\theta}_{q_{\max}+2}-\theta|\geq \delta_{q_{\max}+2}/2$ or if 
$|\wh{\theta}_{q_{\max}}-\theta|\geq  \delta_{q_{\max}+2}/2$. This leads to 
\beqn 
\E_{\theta,\mu}\Big[|\wh{\theta}_{q_{\max}+2}-\theta|\ind_{\wh{q}= q_{\max}+2}\Big]&\leq& \E_{\theta,\mu}\Big[|\wh{\theta}_{q_{\max}+2}-\theta|\ind_{|\wh{\theta}_{q_{\max}+2}-\theta| \geq \delta_{q_{\max}+2}}\Big]\\ && + 2\sqrt{2\log(n)}\P_{\theta,\mu}\left[|\wh{\theta}_{q_{\max}}-\theta|\geq 2\sqrt{2\log(n)} \right]\\
&\lesssim & \frac{1}{\sqrt{n}}+ \sqrt{\log(n)}\P_{\theta,\mu}\left[|\wh{\theta}_{q_{*}}-\theta|\geq 2\sqrt{2\log(n)} \right]\ . 
\eeqn 
As previously, we consider two subcases: (i) $v^*_k< 2\sqrt{2\log(n)}$, in which case the deviation bound \eqref{eq:deviation_importante} implies that 
\beqn 
 \sqrt{\log(n)}\P_{\theta,\mu}\left[|\wh{\theta}_{q_{*}}-\theta|\geq 2\sqrt{2\log(n)} \right]&\lesssim& (\frac{k}{\sqrt{n}})^{-4/3} \log^{7/2}(n)\lesssim 
 n^{-2/3}\log^{7/2}(n)\\ &\lesssim &\frac{\log^{2}\big(1+ \sqrt{\frac{k}{n-k}}\big)}{\log^{3/2}\big(1+ (\frac{k}{\sqrt{n}})^{2/3}\big)}  \ , 
\eeqn 
since $q_*=q_{\max}$. If (ii) $v^*_k\geq 2\sqrt{2\log(n)}$, we straightforwardly derive the rough bound 
\[
\E_{\theta,\mu}\Big[|\wh{\theta}_{q_{\max}+2}-\theta|\ind_{\wh{q}= q_{\max}+2}\Big]\lesssim \sqrt{\log(n)}\ , 
\]
which is nevertheless optimal. 
Together with \eqref{eq:decomposition_risque} and \eqref{eq:sous_estimation}, we have proved the desired risk bound.

\subsection{Proofs for the quantile estimators (OSC)}
\label{p:thm:upper_robust}

\begin{proof}[Proof of Theorem \ref{thm:upper_robust}]

The proof is based on Lemmas~\ref{lem:biais_estimateur_quantile} and~\ref{lem:quantile_empirique_2}. 
First recall that the following holds: 
\begin{align*}
 \xi_{(q)}+ \ol{\Phi}^{-1}({q}/{n})  \preceq \wt{\theta}_q -\theta &\preceq \xi_{(q:n-k)}+ \ol{\Phi}^{-1}({q}/{n})\\
 &\preceq \left[\xi_{(q:n-k)}+ \ol{\Phi}^{-1}({q}/{(n-k)})\right]_++ \ol{\Phi}^{-1}({q}/{n})-\ol{\Phi}^{-1}({q}/{(n-k)})\ .
\end{align*}
Let us prove \eqref{eq:result_thetat_robustes}. 
 It follows from the above decomposition that, for any $x>0$, 
\[
 -x \leq  \wt{\theta}_q -\theta \leq \ol{\Phi}^{-1}({q}/{n})-\ol{\Phi}^{-1}({q}/{(n-k)}) + x \ , 
\]
with probability higher than $1- \P[\xi_{(q:n-k)}+ \ol{\Phi}^{-1}(\frac{q}{n-k})\geq x]- \P[\xi_{(q)}+ \ol{\Phi}^{-1}(\frac{q}{n})\leq -x ]$. Then, Lemmas~\ref{lem:biais_estimateur_quantile} and~\ref{lem:quantile_empirique_2} yield the desired bound \eqref{eq:result_thetat_robustes} for all $x\leq c_3 q$\footnote{Actually, $c_3$ corresponds to $c_1$ in the statement of Lemma \ref{lem:quantile_empirique_2}.}. 
 
 Let us now prove \eqref{eq:risk_theta_tilde_robust}.
Define the event
\[
 \cA= \left\{|\wt{\theta}_q -\theta|\leq c_1 \frac{\log\left(\frac{n}{n-k}\right)}{\sqrt{\log\big(\frac{n-k}{q}\big)_+}\vee 1} + c_2 \sqrt{\frac{c_3}{[\log(\frac{n-k}{q})\vee 1]}}\right\}\ . 
\]
From above, the random variable $|\wt{\theta}_q -\theta|\ind_{\cA}$ satisfies  for all $x>0$,
\begin{align}\label{defboundwithA}
 \P_{\theta,\pi}\left[ |\wt{\theta}_q -\theta|\ind_{\cA} \leq c_1 \frac{\log\left(\frac{n}{n-k}\right)}{\sqrt{\log\big(\frac{n-k}{q}\big)_+}\vee 1} + c_2 \sqrt{\frac{x}{q[\log(\frac{n-k}{q})\vee 1]}}\right]\geq 1- 2e^{-x}\ .
\end{align}
Integrating this deviation inequality yields 
\[
 \E_{\theta,\pi}\left[|\wt{\theta}_q -\theta|\ind_{\cA}\right] \leq c_1 \frac{\log\left(\frac{n}{n-k}\right)}{\sqrt{\log\big(\frac{n-k}{q}\big)_+}\vee 1} + c'_2 \sqrt{\frac{1}{q[\log(\frac{n-k}{q})\vee 1]}}\  . 
\]
Let us control the remaining term $\E_{\theta,\pi}[|\wt{\theta}_q -\theta|\ind_{\cA^c}]$. By Cauchy-Schwarz inequality, we have 
\[
 \E_{\theta,\pi}[|\wt{\theta}_q -\theta|\ind_{\cA^c}]\leq  \E_{\theta,\pi}^{1/2}[(\wt{\theta}_q -\theta)^2]\P_{\theta,\pi}^{1/2}[\cA^c]\leq
 \sqrt{2} e^{-c_3q/2}
 \left[\E_{\theta,\pi}^{1/2}[(\wt{\theta}_q -\theta)_-^2]+ \E_{\theta,\pi}^{1/2}[(\wt{\theta}_q -\theta)_+^2]\right]\ .
\]
 We can use a crude stochastic bound $\xi_{(1)}- \ol{\Phi}^{-1}(1/{n})  \preceq \wt{\theta}_q -\theta \preceq \xi_{(n)}+ \ol{\Phi}^{-1}(1/{n})$. By an union bound together with integration, we arrive at $\E_{\theta,\pi}[(\wt{\theta}_q -\theta)^2]\leq c\log(n)$.
Putting everything together, we obtain
\[
 \E_{\theta,\pi}\left[|\wt{\theta}_q -\theta|\right] \leq c_1 \frac{\log\left(\frac{n}{n-k}\right)}{\sqrt{\log\big(\frac{n-k}{q}\big)_+}\vee 1} + c'_2 \sqrt{\frac{1}{q[\log(\frac{n-k}{q})\vee 1]}}+ c'_4e^{-c_3q/2}\sqrt{\log(n)}\ .
\]
Taking $c_4= 4/c_3$ in the statement of the theorem, we have $e^{-c_3q/2}\leq n^{-2}$ and \eqref{eq:risk_theta_tilde_robust} follows.

\end{proof}
\begin{proof}[Proof of Corollary~\ref{prp:upper_bound_robuste_non_adaptive} ]
 For $k\leq n-n^{4/5}$, this bound is a straightforward consequence of \eqref{eq:risk_theta_tilde_robust}. 
 In the  proof of Proposition \ref{prp:dense} (see Section~\ref{p:prelest}), we have shown that, for any $\pi\in \ol{\cM}_{n-1}$, 
 $\E_{\theta,\pi}\left[|\wt{\theta}_{1} -\theta|\right]\lesssim \sqrt{\log(n)}$, which is (up to multiple constants) smaller than 
 $\frac{\log\left(\frac{n}{n-k}\right)}{\log^{1/2}(1+ \frac{k^2}{n})}$ for all $k\geq  n-n^{4/5}$. 
\end{proof}

\begin{proof}[Proof of Proposition~\ref{prp:adaptation_robust}]

Consider any $(\theta,\pi)\in \R\times \ol{\cM}$ and denote $k$ the number of contaminations in $\pi$. In the sequel, we write $(\ol{q}_1,\ldots, \ol{q}_{\max})$ for the ordered values in $\cQ$ so that $\ol{q}_1=1$, $\ol{q}_{\max}=\lceil n/2\rceil$ and in between, the $\ol{q}_i$ form a dyadic sequence. 
For short, we write $q_*=q_k$ so that $\wt{\theta}_{q_*}$ achieves minimax performances. Besides, we let $i^*$ be the indice such that $q_*=\overline{q}_{i^*}$.  We shall prove that $\wt{\theta}_{\ad}$ performs almost as well $\wt{\theta}_{q_*}$. The general strategy is the same as in Theorem \ref{thm:adaptation}.

As in the previous proof, we decompose the risk as a sum of two terms depending on the value of $\wh{q}$. 
\[
\E_{\theta,\pi}\big[|\widetilde{\theta}-\theta|\big] = \E_{\theta,\pi}\big[|\wt{\theta}_q-\theta|\ind_{\wh{q}\geq q_*}\big]+\E_{\theta,\pi}\big[|\wt{\theta}_q-\theta|\ind_{\wh{q}<q_*}\big]\ .
\]
For $q\geq  q_*$, it follows from triangular inequality and the  definition of  $\wh{q}$ that  
\[
|\wt{\theta}_q-\theta|\ind_{\wh{q}=q}\leq |\wt{\theta}_{q^{*}}-\wt{\theta}_{q}|\ind_{\wh{q}=q} + |\wt{\theta}_{q_*}-\theta|\ind_{\wh{q}=q}\leq \delta_{q_*}\ind_{\wh{q}=q} +  |\wt{\theta}_{q_*}-\theta|\ind_{\wh{q}=q} \ . 
\]
Summing all these terms over all $q\geq q_*$, we arrive at $\E_{\theta,\pi}\big[|\wt{\theta}_q-\theta|\ind_{\wh{q}\geq q_*}\big]\leq \delta_{q_*}+ \E_{\theta,\pi}[|\wt{\theta}_{q_*}-\theta|]$, which, by Corollary~\ref{prp:upper_bound_robuste_non_adaptive} together with the definition \eqref{eq:definition_delta_q_2} of $\delta_q$, leads to  
\beq\label{eq:upper_ad_osc_1}
\E_{\theta,\pi}\Big[|\wt{\theta}_q-\theta|\ind_{\wh{q}\geq q_*}\Big]\lesssim \frac{\log\big(\frac{n}{n-k}\big)}{\log^{1/2}\big(1+ \frac{k^2}{n}\big)}\ .  
\eeq
\medskip 

Turning to the second expression $\E_{\theta,\pi}\big[|\wt{\theta}_q-\theta|\ind_{\wh{q}<q_*}\big]$, we first assume either that $i^*= 1$ or $i^*> 3$, the case $i^*=2,3$ being deferred to the end of the proof. Define $\widetilde{q}_+=\min\{q': |\wt{\theta}_{q'}-\theta|\geq \delta_{q'}/2\}$. By definition of $\wh{q}$ and by monotonicity of $\delta_q$, one has $\widetilde{q}_+\leq \ol{q}_{\wh{i}+1}$ where $\wh{i}$ is such that $\ol{q}_{\wh{i}}= \wh{q}$. Then, we deduce that, for $q < q_*$, 
\beqn 
|\wt{\theta}_q-\theta|\ind_{\wh{q}=q}&\leq& \sum_{j=1}^{i^*}|\wt{\theta}_{q}-\theta|\ind_{\wh{q}=q}\ind_{\widetilde{q}_+=\ol{q}_j}\\
& \leq &   \sum_{j=2}^{i^*} \Big[ |\wt{\theta}_q-\wt{\theta}_{\ol{q}_{j-1}}|+  |\wt{\theta}_{\ol{q}_{j-1}}-\theta| \Big]\ind_{\wh{q}=q}\ind_{\widetilde{q}_+=\ol{q}_j}  + \Big[ |\wt{\theta}_q-\wt{\theta}_{1}|+  |\wt{\theta}_{1}-\theta| \Big]\ind_{\wh{q}=q}\ind_{\widetilde{q}_+=1}  \\
&\leq & \frac{3}{2}\sum_{j=2}^{i^*}\delta_{\ol{q}_{j-1}}\ind_{\wh{q}=q}\ind_{\widetilde{q}_+=\ol{q}_j}  
+ \Big[ \delta_{1} +  |\wt{\theta}_{1}-\theta| \Big]\ind_{\wh{q}=q}\ind_{\widetilde{q}_+=1}\ ,
\eeqn 
where we used again the definition of $\widetilde{q}_+$ and $\wh{q}$. Summing the above bound over all $q<q_*$ leads to 
\beqn 
\sum_{q<q_*}\E_{\theta,\pi}\Big[|\wt{\theta}_q-\theta|\ind_{\wh{q}=q}\Big]&\leq &   \frac{3}{2}\sum_{j=2}^{i^*}\delta_{\ol{q}_{j-1}}\P_{\theta,\pi}[\widetilde{q}_+=\ol{q}_j]+     \E_{\theta,\pi}\Big[\big(\delta_{1} +|\wt{\theta}_{1}-\theta| \big)\ind_{\widetilde{q}_+=1} \Big]\\
&\leq &\frac{3}{2}\sum_{j=2}^{i^*}\delta_{\ol{q}_{j-1}}\P_{\theta,\pi}\left[|\wt{\theta}_{\ol{q}_j}-\theta|\geq \frac{\delta_{\ol{q}_j}}{2}\right]+     3\E_{\theta,\pi}\Big[|\wt{\theta}_{1}-\theta| \ind_{|\wt{\theta}_{1}-\theta|\geq \delta_{1}/2} \Big]\ .
\eeqn 
For any $\ol{q}_4\leq q\leq q_*$, we apply Theorem \ref{thm:upper_robust} and it follows from  the choice \eqref{eq:definition_delta_q_2} of $\delta_q$ with $c_0$ large enough that $\P_{\theta,\pi}[|\wt{\theta}_q-\theta|\geq \frac{\delta_q}{2}]\leq \exp[- c(n/q)^{1/3}]$. For $q\leq \ol{q}_3\wedge q_*$, it follows from Theorem \ref{thm:upper_robust} and the proof of Proposition \ref{prp:dense} that
$\P_{\theta,\pi}[|\wt{\theta}_q-\theta|\geq \frac{\delta_q}{2}]\leq 1/n$. 
Since $\delta_{q_1}\lesssim \sqrt{\log(n)}$, we obtain 
\[
 \frac{3}{2}\sum_{j=2}^{i^*}\delta_{\ol{q}_j-1}\P_{\theta,\pi}\big[|\wt{\theta}_{\ol{q}_j}-\theta|\geq \frac{\delta_{\ol{q}_j}}{2}\big]\lesssim  \frac{\sqrt{\log(n)}}{n}+ \sum_{j=4}^{i^*} e^{-c (n/\ol{q}_j)^{1/3}} \frac{n^{1/6}}{\ol{q}_j^{2/3}\sqrt{\log(\frac{n}{\ol{q}_j})\vee 1}} \lesssim \frac{1}{\sqrt{n}}\ .
\]
Finally, arguing as in the proof of Theorem \ref{thm:adaptation}, we observe that $\E_{\theta,\pi}[|\wt{\theta}_{1}-\theta| \ind_{|\wt{\theta}_{1}-\theta|\geq \delta_{1}/2} ]\lesssim 1/\sqrt{n}$. Putting everything together we have proved
\beq\label{eq:upper_ad_osc_2}
\E_{\theta,\pi}\Big[|\wt{\theta}_q-\theta|\ind_{\wh{q}<q_*}\Big]\lesssim \frac{1}{\sqrt{n}}\ , 
\eeq
as long as $i^*=1$ or $i^*>3$. 

It remains to consider the case $i_*=2,3$. In that situation, observe that $\log(n/(n-k))\asymp \log(n)$. From Theorem \ref{thm:upper_robust},  we derive that, for $q=\ol{q}_1,\ol{q}_2$, 
\[
\E_{\theta,\pi}\big[|\wt{\theta}_q-\theta|\big]\lesssim \frac{\log\big(\frac{n}{n-k}\big)}{\log^{1/2}\big(1+ \frac{k^2}{n}\big)}\ .
\]
This leads us to 
\[
\E_{\theta,\pi}\Big[|\wt{\theta}_q-\theta|\ind_{\wh{q}< q_*}|\Big]\leq  \sum_{i=1}^2 \E_{\theta,\pi}\big[\wt{\theta}_{\ol{q}_i}-\theta|\big]\lesssim \frac{\log\big(\frac{n}{n-k}\big)}{\log^{1/2}\big(1+ \frac{k^2}{n}\big)}\ . 
\]
Together with \eqref{eq:upper_ad_osc_1} and \eqref{eq:upper_ad_osc_2}, this concludes the proof. 
\end{proof}

\begin{proof}[Proof of Proposition~\ref{prp:upper_unknown}]
First consider the variance estimator $\widetilde{\sigma}_{q_k,q'_k}$. We  only deal with the case where $k\leq n-n^{4/5}$, the other case being trivial.
We start from the decomposition
\beqn 
\frac{\big|\widetilde{\sigma}_{q,q'}-\sigma \big|}{\sigma}\leq \left|\frac{Y_{(q)}/\sigma -\theta/\sigma  + \ol{\Phi}^{-1}(\frac{q}{n}) }{\overline{\Phi}^{-1}(q'/n)- \overline{\Phi}^{-1}(q/n)}\right|+ \left|\frac{Y_{(q')}/\sigma -\theta/\sigma  +  \ol{\Phi}^{-1}(\frac{q'}{n}) }{\overline{\Phi}^{-1}(q'/n)- \overline{\Phi}^{-1}(q/n)}\right|\ . 
\eeqn 
Since the rescaled non contaminated observations $Y_i/\sigma$ have variance $1$, we can  apply Theorem \ref{thm:upper_robust} to control the expectations of the above rhs term. 
This leads us to 
\beq\label{eq:first_upper_risk_sigma}
\E_{\theta,\pi,\sigma}\left[\frac{\big|\widetilde{\sigma}_{q_k,q_k'}-\sigma \big|}{\sigma}\right] \lesssim \frac{\log\left(\frac{n}{n-k}\right)}{\log^{1/2}\left(1+ \frac{k^2}{n}\right)}\cdot \frac{1}{\overline{\Phi}^{-1}(q'_k/n)- \overline{\Phi}^{-1}(q_k/n)}\ .
\eeq
It remains to derive a lower bound of the difference in the denominator. 
We claim that 
\beq \label{eq:claim_diff}
\overline{\Phi}^{-1}(q'_k/n)- \overline{\Phi}^{-1}(q_k/n)\gtrsim \sqrt{\log\left(\frac{k}{\sqrt{n}}\right)\vee 1}\ , 
\eeq
which together with previous bound leads to \eqref{eq:risk_upper_sigma_hat}. Let us show this claim. When  $k/\sqrt{n}$ is smaller than some constant $\ol{c}$  that will be fixed later, the difference is lower bounded by an absolute constant (depending on $\ol{c}$) by the first inequality in \eqref{eq:lower_difference_quantile} in Lemma \ref{lem:difference_quantile}. If $\ol{c}$ is chosen large enough, we have, $q'_k\leq q_k\leq 0.004 \:n$ for $k\geq \ol{c} \sqrt{n}$.
 The ratio $q_k/q'_k$ is larger than $k/\sqrt{n}$. The third inequality in \eqref{eq:lower_difference_quantile} together with \eqref{eq:encadrement_quantile_1plus} then implies that 
\beqn 
\overline{\Phi}^{-1}(q'_k/n)- \overline{\Phi}^{-1}(q_k/n)&\geq &\frac{1}{\overline{\Phi}^{-1}(q_k'/n) }\left[\log\left(\frac{q_k}{eq'_k}\right)+  \frac{1}{2}\log\log\left(\frac{n}{q_k}\right)- \frac{1}{2}\log\log\left(\frac{n}{q'_k}\right)\right]\\
&\gtrsim  & \frac{1}{\sqrt{\log(\frac{k}{\sqrt{n}})_+} }\left[\log\left(\frac{k}{\sqrt{n}}\right)  - c' - \log\log\left(\frac{k}{\sqrt{n}}\right) \right]\ ,
\eeqn 
where $c'$ is some constant. Since $k/\sqrt{n}\geq \ol{c}$, the first logarithmic term is larger than the remaining expressions in the rhs, and we  obtain \eqref{eq:claim_diff}.

We now consider the estimator $\wt{\theta}_{q_k,q'_k}$. We start from the decomposition 
\[
 \E_{\theta,\pi,\sigma}\left[\frac{|\wt{\theta}_{q_k,q'_k} - \theta|}{\sigma}\right]\leq \E_{\theta,\pi,\sigma}\left[ \frac{|\wt{\theta}_{q_k}- \theta|}{\sigma}\right]+ \E_{\theta,\pi,\sigma}\left[\frac{|\wt{\sigma}_{q_k,q'_k}-\sigma|}{\sigma} \right]\left|\ol{\Phi}^{-1}\left(\frac{q_k}{n}\right)\right|\ .
\]
The first expectation in the rhs has been controlled in Corollary~\ref{prp:upper_bound_robuste_non_adaptive} whereas the second expectation has been handled in the first part of this proof. We deduce from Lemma \ref{lem:quantile} that $\ol{\Phi}^{-1}\left(\frac{q_k}{n}\right)\lesssim \sqrt{\log(n/q_k)_+\vee 1}\lesssim \sqrt{\log(\frac{k^2}{n})_+\vee 1}$. Putting everything together leads to the desired result.

\end{proof}

\subsection{Proof of Proposition~\ref{rem:estimator} (OSC)}

\begin{proof}
First note that $q_n/n \asymp n^{-1/4}$ while $q'_n/(n-k_0) \asymp n^{-3/4}$, because we assume $k_0\leq 0.9 n$. This implies  that $\overline{\Phi}^{-1}(q'_n/(n-k_0))\geq  \overline{\Phi}^{-1}(q_n/n)$ for $n$ large enough.
Also by Lemma~\ref{lem:difference_quantile}, we have 
\begin{equation*}
\overline{\Phi}^{-1}(q'_n/(n-k_0)) -  \overline{\Phi}^{-1}(q_n/n) \asymp \log^{1/2} (n)\ .
\end{equation*}
Now use the following decomposition:
\begin{align*}
\frac{\wt{\sigma}_{+}- \sigma }{\sigma} 
& = \frac{Y_{(q_n)}/\sigma-\theta/\sigma + \overline{\Phi}^{-1}(q_n/n)}{\overline{\Phi}^{-1}(q'_n/(n-k_0))- \overline{\Phi}^{-1}(q_n/n)} -
\frac{Y_{(q'_n)}/\sigma-\theta/\sigma + \overline{\Phi}^{-1}(q'_n/(n-k_0))}{\overline{\Phi}^{-1}(q'_n/(n-k_0))- \overline{\Phi}^{-1}(q_n/n)}\\
&= \quad\quad\quad\quad T_1 \quad\quad\quad\quad\quad\quad\quad\quad\quad- \quad\quad\quad\quad T_2\ ,
\end{align*}
We consider separately the deviations of $T_1$ and $T_2$. We apply Theorem~\ref{thm:upper_robust} to $T_1$. Hence, for some constant $c>0$ and for all $x \in (0,c_3 q_n)$, we have 
\begin{align*}
\P_{\theta,\pi,\sigma}\left( T_1 < - c\frac{\sqrt{x}}{n^{3/8}\log(n)} \right)
& \leq \P_{\theta,\pi,\sigma}\left( [\overline{\Phi}^{-1}(q'_n/(n-k_0))- \overline{\Phi}^{-1}(q_n/n)]T_1 < - c'\sqrt{\frac{x}{q_n\log(n)}} \right)\\ &\leq 2 e^{-x}\ .
\end{align*}
For $T_2$, we start from
\begin{align*}
Y_{(q'_n)}/\sigma-\theta/\sigma + \overline{\Phi}^{-1}(q'_n/(n-k_0))
&\leq Y_{(q'_n)}/\sigma-\theta/\sigma + \overline{\Phi}^{-1}(q'_n/n_0)\\
&\preceq \xi_{(q'_n:n_0)} + \overline{\Phi}^{-1}(q'_n/n_0)\ .
\end{align*}
Then, we use \eqref{equ3Nico} to derive that there exists  constants $c'$ and $\ol{c}'$ such that,  for all $x\in (0, n^{1/4})$,
\begin{align*}
&\P_{\theta,\pi,\sigma}\left( T_2 >  c'\frac{\sqrt{x}}{ n^{1/8} \log(n)} \right)
\leq \P\left(\xi_{(q'_n:n_0)} + \overline{\Phi}^{-1}(q'_n/n_0) > \ol{c}'\sqrt{\frac{x}{ q'_n\log(n)}} \right)\leq e^{-x}\ ,
\end{align*}
Combining the two bounds leads to the following deviation inequality,
\begin{equation}\label{equsigmatildeproof}
\P_{\theta,\pi,\sigma}\left(\frac{\wt{\sigma}_{+}- \sigma }{\sigma}  < - c'' \frac{\sqrt{x}}{n^{1/8}\log(n)} \right) \leq 3 e^{-x}\ , 
\end{equation}
holding for all $x\in (0, n^{1/4})$. 
We obtain  \eqref{inequthetachapsigmachapbis} by taking $x=(c'')^{-2} n^{1/8} \log^{2} (n)$ in \eqref{equsigmatildeproof}. 

Relation \eqref{inequthetachapsigmachap} is obtained similarly by using the decomposition:
\begin{align}
\frac{\wt{\theta}_{+}- \theta }{\sigma} 
& = Y_{(q_n)}/\sigma-\theta/\sigma + \overline{\Phi}^{-1}(q_n/n)+
\frac{ \wt{\sigma}_{+} -\sigma}{\sigma}\:\ol{\Phi}^{-1}\left(\frac{q_n}{n}\right)\ . \label{equ:fromsigmatotheta}
\end{align}
This gives that for some constant $c'''>0$, for all $x\in (0, n^{1/4})$, 
\begin{equation}\label{equsigmatildeproof2}
\P_{\theta,\pi,\sigma}\left(\frac{\wt{\theta}_{+}- \theta }{\sigma}  < - c''' \sqrt{x} n^{-1/8} \log^{-1/2} (n)\right) \leq e^{-x}\ ,
\end{equation}
which, for $x=(c''')^{-2} n^{1/8} \log(n)$ leads to \eqref{inequthetachapsigmachap}.

Let us now establish \eqref{inequthetachapsigmachap2bis}. By \eqref{equsigmatildeproof}, we only have to study the probabilities of overestimation. As in the first part of the proof, we consider separately $T_1$ and $T_2$. 
First, Theorem~\ref{thm:upper_robust} (used with $k=k_0$), gives that for some constant $c>0$, for all $x \in (0,c_3 q_n)$,
\begin{align*}
&\P_{\theta,\pi,\sigma}\left( T_1 > c  \frac{k_0}{n\log(n)} + c\frac{\sqrt{x}}{ n^{3/8} \log(n)} \right)\\
&\leq \P_{\theta,\pi,\sigma}\left(Y_{(q_n)}/\sigma-\theta/\sigma + \overline{\Phi}^{-1}(q_n/n) >  \frac{\ol{c}\:  k_0}{n\sqrt{\log(n)}}+\ol{c}\sqrt{\frac{x}{n^{3/4}\log(n)}}  \right)\leq 2 e^{-x}\ . 
\end{align*}
 Turning to  $T_2$, we start by controlling the difference of quantiles with  \eqref{eq:upper_difference_quantile}:
$$
\overline{\Phi}^{-1}(q'_n/n) -\overline{\Phi}^{-1}(q'_n/(n-k_0))  \lesssim \frac{q'_n/(n-k_0)-q'_n/n}{\tfrac{q'_n}{n}\log^{1/2}(n)}  \lesssim \frac{k_0}{n \log^{1/2} (n)}\ .
$$
Then, by stochastic domination
, we have for some constant $c_0>0$,
\begin{align*}
Y_{(q'_n)}/\sigma-\theta/\sigma + \overline{\Phi}^{-1}(q'_n/(n-k_0))
&\geq Y_{(q'_n)}/\sigma-\theta/\sigma + \overline{\Phi}^{-1}(q'_n/n)  - c_0 \frac{k_0}{n \log^{1/2} (n)} \\
&\succeq \xi_{(q'_n)} + \overline{\Phi}^{-1}(q'_n/n) - c_0\frac{k_0}{n \log^{1/2} (n)}\ .
\end{align*}
Putting the above inequalities together and relying on  the deviation bound~\eqref{equ4Nico} leads to 
\begin{align*}
&\P_{\theta,\pi,\sigma}\left( T_2 <-c' \frac{k_0}{n\log(n)} -  c_0\frac{\sqrt{x}}{ n^{1/8} \log(n)} \right)\\
&\leq \P\left( \frac{\xi_{(q'_n)} + \overline{\Phi}^{-1}(q'_n/n)}{\overline{\Phi}^{-1}(q'_n/(n-k_0))- \overline{\Phi}^{-1}(q_n/n)} <- c'\frac{\sqrt{x}}{n^{1/8} \log(n)}\right)\leq e^{-x}\ ,
\end{align*}
 for all $x\in (0, q_n'/8)$. 
 Combining the two deviation inequalities  for $T_1$ and 
 $T_2$ gives that for some constants $c'',c_4>0$, for all $x\in (0, c_4 n^{1/4})$, 
 $$
\P_{\theta,\pi,\sigma}\left(\frac{\wt{\sigma}_{+}- \sigma }{\sigma}  > c'' \frac{k_0}{n\log(n)} + c'' \frac{\sqrt{x}}{n^{1/8}\log(n)} \right) \leq 3 e^{-x}\ .
$$
Choosing $x=(c'')^{-2} n^{1/8} \log^2(n)$ leads to \eqref{inequthetachapsigmachap2bis}.

Finally,  \eqref{inequthetachapsigmachap2} follows from the decomposition \eqref{equ:fromsigmatotheta}, the relation on $T_1$ and \eqref{inequthetachapsigmachap2bis}.
\end{proof}

\section{Proofs for multiple testing and post hoc bounds}

In these proofs, to lighten the notation, the subscript $\alpha$ will be sometimes dropped in $\hat{\ell}_\alpha(u,s),\wh{t}_{\alpha}(u,s)$ ; the parameters $\theta,\pi,\sigma$ are removed in $\P_{\theta,\pi,\sigma}$ and $\E_{\theta,\pi,\sigma}$ and $\wt{\theta}_{+}$ (resp. $\wt{\sigma}_{+}$) are denoted by $\hat{\theta}$ (resp. $\hat{\sigma}$). We also let $\delta_n=n^{-1/16}$, so that, by Proposition~\ref{rem:estimator}, we have
$\P(\hat{\theta}- \theta   \leq  -  \sigma \delta_n)  \leq c  / n$ and $\P(\hat{\sigma}- \sigma   \leq  -  \sigma \delta_n)  \leq c /n$, for some constant $c>0$.

\subsection{Proof of Theorem~\ref{th:FDPquantile}}\label{sec:proofth:FDPquantile}

We start with a key observation. For $i\in \{1,\ldots,n\}$, the quantity $Y^{(i)}_{(q)}$ denotes the $q$-smallest element of $\{Y_j,1\leq j \leq n, j\neq i\}$ and
\begin{equation*}
\left\{\begin{array}{l}
\mbox{$\hat{\theta}^{(i)}=Y^{(i)}_{(q_n)}+\hat{\sigma}^{(i)}\: \ol{\Phi}^{-1}(q_n/n)$\ ;}\\
\mbox{ $\hat{\sigma}^{(i)}=\frac{Y^{(i)}_{(q_n)}-Y^{(i)}_{(q_n')}}{\overline{\Phi}^{-1}(q_n'/(0.1 n))- \overline{\Phi}^{-1}(q_n/n)}$ ,}
\end{array}
\right.
\end{equation*}
so that the estimators $\hat{\theta}^{(i)}$ and $\hat{\sigma}^{(i)}$ are independent of $Y_i$. 

\medskip 

\noindent
{\bf Claim}: For any $n$ large enough and for any $t\in (0,\alpha]$, we have 
\begin{align}
\left\{ p_i(\hat{\theta},\hat{\sigma}) \leq t\right\} &= \left\{ p_i(\hat{\theta},\hat{\sigma} ) \leq t , \hat{\theta}=\hat{\theta}^{(i)},\hat{\sigma}=\hat{\sigma}^{(i)}\right\}= \left\{ p_i(\hat{\theta}^{(i)},\hat{\sigma}^{(i)}) \leq t\right\}
\label{keyrelation}
\ .
\end{align}

\medskip 

\begin{proof}[Proof of the claim]
Consider any $i$ such that $p_i(\hat{\theta},\hat{\sigma}) \leq t$. 
 By definition \eqref{equ-pvalues} of the $p$-values, we have $
Y_i - Y_{(q_n)} \geq \hat{\sigma}\big[ \ol{\Phi}^{-1}(q_n/n) + \ol{\Phi}^{-1}(\alpha)\big]
$ which is positive for $n$ large enough. This entails $Y_{(q_n)}=Y^{(i)}_{(q_n)}$, $Y_{(q'_n)}=Y^{(i)}_{(q'_n)}$ and therefore $\hat{\theta}=\hat{\theta}^{(i)}$, $\hat{\sigma}=\hat{\sigma}^{(i)}$.

 Conversely, if 
$p_i(\hat{\theta}^{(i)},\hat{\sigma}^{(i)}) \leq t$, we have 
$
Y_i - Y_{(q_n)}^{(i)} \geq \hat{\sigma}^{(i)}\big[\ol{\Phi}^{-1}(q_n/n) +\ol{\Phi}^{-1}(\alpha)\big]>0
$
for $n$ large enough. This also leads to  $Y_{(q_n)}=Y^{(i)}_{(q_n)}$, $Y_{(q'_n)}=Y^{(i)}_{(q'_n)}$. We have proved \eqref{keyrelation}.

\end{proof}

The latter property can be suitably combined with Lemma~\ref{lem:MT} (see the notation therein) to give the following equalities:
\begin{eqnarray}
\left\{ p_i(\hat{\theta},\hat{\sigma}) \leq \alpha  \wh{\ell}(\hat{\theta},\hat{\sigma})/n \right\}&=&\left\{ p_i(\hat{\theta},\hat{\sigma})\leq \alpha  \wh{\ell}^{(i)}(\hat{\theta},\hat{\sigma})/n \right\}  \nonumber \\ 
&=&\left\{ p_i(\hat{\theta}^{(i)},\hat{\sigma}^{(i)}) \leq \alpha  \wh{\ell}^{(i)}(\hat{\theta},\hat{\sigma})/n,\hat{\theta}=\hat{\theta}^{(i)},\hat{\sigma}=\hat{\sigma}^{(i)} \right\}\nonumber\\ 
&=&\left\{ p_i(\hat{\theta}^{(i)},\hat{\sigma}^{(i)}) \leq \alpha  \wh{\ell}^{(i)}(\hat{\theta}^{(i)},\hat{\sigma}^{(i)})/n,\hat{\theta}=\hat{\theta}^{(i)},\hat{\sigma}=\hat{\sigma}^{(i)}  \right\}\nonumber\\ 
&=&\left\{ p_i(\hat{\theta}^{(i)},\hat{\sigma}^{(i)}) \leq \alpha  \wh{\ell}^{(i)}(\hat{\theta}^{(i)},\hat{\sigma}^{(i)})/n \right\}\nonumber\\ 
& \subset &\left\{\wh{\ell}^{(i)}(\hat{\theta}^{(i)},\hat{\sigma}^{(i)})=\wh{\ell}(\hat{\theta},\hat{\sigma})\right\}\ .\label{keyrelation2}
\end{eqnarray}
The first equality is a direct application of Lemma~\ref{lem:MT}, used with $u=\hat{\theta}$ and $s=\hat{\sigma}$. The second equality comes from \eqref{keyrelation} used with $t=\alpha  \wh{\ell}^{(i)}(\hat{\theta},\hat{\sigma})/n$. The third equality is trivial. The fourth equality comes from \eqref{keyrelation} used with $t= \alpha  \wh{\ell}^{(i)}(\hat{\theta}^{(i)},\hat{\sigma}^{(i)})/n$. 
Now, since $\BH_\alpha(\hat{\theta},\hat{\sigma})=\{ 1\leq i \leq n\::\: p_i(\hat{\theta},\hat{\sigma})\leq \alpha \hat{\ell}(\hat{\theta},\hat{\sigma})/n\} $, we have
\begin{align*}
&\FDP(\pi, \BH_\alpha(\hat{\theta},\hat{\sigma}))\:\mathds{1}{\{\hat{\theta}- \theta   >  -  \sigma\delta_n, \hat{\sigma}- \sigma   >  -  \sigma\delta_n\}} \\
&= \sum_{i\in \cH_0}\frac{\mathds{1}{\{p_i(\hat{\theta},\hat{\sigma})\leq \alpha \hat{\ell}(\hat{\theta},\hat{\sigma})/n\}}}{\hat{\ell}(\hat{\theta},\hat{\sigma})\vee 1} \:\mathds{1}{\{\hat{\theta}^{(i)}- \theta   >  -  \sigma\delta_n, \hat{\sigma}^{(i)}- \sigma   >  -  \sigma\delta_n\}}\\
&= \sum_{i\in \cH_0}\frac{\mathds{1}{\{p_i(\hat{\theta}^{(i)},\hat{\sigma}^{(i)})\leq \alpha \hat{\ell}^{(i)}(\hat{\theta}^{(i)},\hat{\sigma}^{(i)})/n\}}}{\hat{\ell}^{(i)}(\hat{\theta}^{(i)},\hat{\sigma}^{(i)})} \:\mathds{1}{\{\hat{\theta}^{(i)}- \theta   >  -  \sigma\delta_n, \hat{\sigma}^{(i)}- \sigma   >  -  \sigma\delta_n\}}\ ,
\end{align*}
by applying \eqref{keyrelation} and \eqref{keyrelation2}. By integration,
\begin{align*}
&\E\left[ \FDP(\pi, \BH_\alpha(\hat{\theta},\hat{\sigma}))\:\mathds{1}{\{\hat{\theta}- \theta   >  -  \sigma\delta_n, \hat{\sigma}- \sigma   >  -  \sigma\delta_n\}}\right]\\ 
&= \sum_{i\in \cH_0}\E\left[\frac{\mathds{1}{\{p_i(\hat{\theta}^{(i)},\hat{\sigma}^{(i)})\leq \alpha \hat{\ell}^{(i)}(\hat{\theta}^{(i)},\hat{\sigma}^{(i)})/n\}}}{\hat{\ell}^{(i)}(\hat{\theta}^{(i)},\hat{\sigma}^{(i)})} \:\mathds{1}{\{\hat{\theta}^{(i)}- \theta   >  -  \sigma\delta_n, \hat{\sigma}^{(i)}- \sigma   >  -  \sigma\delta_n\}}\right]\\
&= \sum_{i\in \cH_0}\E\left[\frac{\mathds{1}{\{\hat{\theta}^{(i)}- \theta   >  -  \sigma\delta_n, \hat{\sigma}^{(i)}- \sigma   >  -  \sigma\delta_n\}}}{\hat{\ell}^{(i)}(\hat{\theta}^{(i)},\hat{\sigma}^{(i)})} \P\left[p_i(\hat{\theta}^{(i)},\hat{\sigma}^{(i)})\leq \alpha \hat{\ell}^{(i)}(\hat{\theta}^{(i)},\hat{\sigma}^{(i)})/n \:\mid \: Y_j,j\neq i\right]\right]\\
&= \sum_{i\in \cH_0}\E\left[\frac{\mathds{1}{\{\hat{\theta}^{(i)}- \theta   >  -  \sigma\delta_n, \hat{\sigma}^{(i)}- \sigma   >  -  \sigma\delta_n\}}}{\hat{\ell}^{(i)}(\hat{\theta}^{(i)},\hat{\sigma}^{(i)})} U_{\hat{\theta}^{(i)},\hat{\sigma}^{(i)}}(\alpha \hat{\ell}^{(i)}(\hat{\theta}^{(i)},\hat{\sigma}^{(i)})/n) \right]\ ,
\end{align*}
by independence between the $Y_i$'s and by \eqref{fromperfecttoapprox}, because the perfectly corrected $p$-values \eqref{equ-pvaluesperfect} are uniformly distributed on $(0,1)$.
Now,  using that $U_{u,s}(t)$ is nonincreasing both in $u$ and $s$ for $t<1/2$, the last display is smaller than
\begin{align*}
 \sum_{i\in \cH_0}\E\left[\frac{U_{\theta-\sigma\delta_n,\sigma-\sigma\delta_n}(\alpha \hat{\ell}^{(i)}(\hat{\theta}^{(i)},\hat{\sigma}^{(i)})/n)}{\hat{\ell}^{(i)}(\hat{\theta}^{(i)},\hat{\sigma}^{(i)})}  \right]
 &\leq \frac{\alpha}{n}\sum_{i\in \cH_0}\left[\max_{\alpha/n \leq t \leq \alpha}  \frac{U_{\theta-\sigma\delta_n,\sigma-\sigma\delta_n}(t)}{t}  \right] \\
 &= \alpha \frac{n_0}{n}\left(1+\max_{\alpha/n \leq t \leq \alpha} \left\{ \frac{U_{\theta-\sigma\delta_n,\sigma - \sigma\delta_n}(t)-t}{t}  \right\} \right)\ .
\end{align*}
The first inequality comes from $ \hat{\ell}^{(i)}(\hat{\theta}^{(i)},\hat{\sigma}^{(i)})\geq 1$, which holds because the BH procedure always rejects a null hypothesis corresponding to a zero $p$-value (see the notation of Lemma~\ref{lem:MT}). 
The result  \eqref{FDRcontrol} by is then a consequence of Lemma~\ref{lem:majorationnull} and Proposition~\ref{rem:estimator}.

Let us now turn to the second statement \eqref{Powercontrol} and assume $n_1\geq 1$ (otherwise the result is trivial). 
Remember that we have $n_1/n\asymp k_0/n$. Thus,  Proposition~\ref{rem:estimator} implies, for some constant $c>0$, $x_n=c\big( (n_1/n)\log^{-1/2} (n) + n^{-1/16} \big)$ and $y_n=c\big( (n_1/n)\log^{-1}(n) + n^{-1/16} \big)$, the deviation inequalities 
$\P\left(|\hat{\theta}- \theta|   \geq     \sigma x_n \right)  \leq  c /n $ and $\P\left(|\hat{\sigma}- \sigma|   \geq    \sigma y_n\right)  \leq c/n$. 
Denote the event 
$$
\mathcal{A}=\left\{\wh{\theta}- \theta   \leq   \sigma x_n ;\quad\wh{\sigma}- \sigma   \leq  \sigma y_n ;\quad  \wh{\sigma} \geq \sigma/2\right\}\ ,
$$
so that $\P(\mathcal{A}^c)\lesssim 1/n$.
For any  $\eta>0$, we have 
\begin{align}
\E\left( \TDP(\pi, \BH_\alpha^\star)  \right)&\leq \eta + \int_\eta^1 \P\left( \TDP(\pi, \BH_\alpha(\theta,\sigma))  \geq u\right) du\nonumber\\
&\leq \eta + \int_\eta^1 \P\left( \TDP(\pi, \BH_\alpha(\theta,\sigma))  \geq u, \mathcal{A}\right) du +\P(\mathcal{A}^c)\ .\label{equcomingback}
\end{align}
Consider the event  $\mathcal{A}\cap\{\TDP(\pi, \BH_\alpha(\theta,\sigma))  \geq \eta\}$.  By definition \eqref{equ-TDP} of the TDP, when this event holds, we have $\wh{\l}_{\alpha}(\theta,\sigma)\geq \eta n_1$. Write $t_0 = \alpha \eta n_1/n$. Invoking Lemma~\ref{lem:MT3}, we obtain 
$
\wh{\l}_{\alpha_0}(\hat{\theta},\hat{\sigma}) \geq\wh{\l}_{\alpha}(\theta,\sigma)\geq 1,
$
for $\alpha_0>0$ such that
\begin{align*}
\frac{\alpha_0}{\alpha}&=  \frac{U_{\hat{\theta},\hat{\sigma}}^{-1}(\wh{t}_{\alpha}(\theta,\sigma))}{\wh{t}_{\alpha}(\theta,\sigma)} 
=
\frac{ \ol{\Phi}\left( \ol{\Phi}^{-1}\left(\wh{t}_{\alpha}(\theta,\sigma)\right)  - \frac{\hat{\sigma}-\sigma}{\hat{\sigma}}\ol{\Phi}^{-1}\left(\wh{t}_{\alpha}(\theta,\sigma)\right) +\frac{\theta-\hat{\theta}}{\hat{\sigma}} \right)}{\wh{t}_{\alpha}(\theta,\sigma)}\\
&\leq \sup_{t\in[t_0,\alpha]} \left\{\frac{ \ol{\Phi}\left( \ol{\Phi}^{-1}\left(t\right) -\frac{2y_n\sigma}{\sigma}\ol{\Phi}^{-1}\left(t\right)  -\frac{2x_n\sigma}{\sigma} \right)}{t}\right\} = 1+\sup_{t\in[t_0,\alpha]} \left\{\frac{ U_{\theta-2x_n\sigma,\sigma-2y_n\sigma}(t)-t}{t}\right\}\ .
\end{align*}

Now using Lemma~\ref{lem:majorationnull}, we get
\begin{align*}
\frac{\alpha_0-\alpha}{\alpha}&\lesssim x_n \log^{1/2} (\tfrac{1}{t_0}) + y_n \log(\tfrac{1}{t_0}) 
\ ,
\end{align*}
as soon as this upper-bound is smaller than some constant small enough.
But now
\begin{align*}
x_n \log^{1/2} (\tfrac{1}{t_0}) + y_n \log(\tfrac{1}{t_0}) \lesssim&  \:n^{-1/16}\log\big(\frac{n}{\alpha \eta n_1}\big)  + \frac{n_1}{n} \log^{-1/2} (n)  \log \big(\frac{n}{n_1\alpha \eta}\big)\ .
\end{align*}
Since $\sup_{x\in (0,1)}(x\log(1/x))=e^{-1}$ 
and since $\epsilon_n\gg \log^{-1/2} (n)$, 
we have for $n$ large enough (as a function of $\alpha$ and $\eta$),
$$
\frac{\alpha_0-\alpha}{\alpha} \leq \epsilon_n\ .
$$
As a consequence,    on the event $\mathcal{A}\cap\{\TDP(\pi, \BH_\alpha(\theta,\sigma))  \geq \eta\}$, we have $\wh{\ell}_{\alpha(1+\epsilon_n)}(\wh{\theta},\wh{\sigma}) \geq  \wh{\ell}_{\alpha_0}(\wh{\theta},\wh{\sigma}) \geq  \wh{\ell}_{\alpha}(\theta,\sigma),$ and thus
 $\TDP(\pi, \BH_{\alpha(1+\epsilon_n)}(\wh{\theta},\wh{\sigma}))\geq \TDP(\pi, \BH_{\alpha}(\theta,\sigma))$ for $n$ large enough.
 Coming back to \eqref{equcomingback}, we obtain for $n$ large enough,
\begin{align*}
\E\left( \TDP(\pi, \BH_\alpha(\theta))  \right) 
&\leq \eta + \int_\eta^1 \P\left(\TDP(\pi, \BH_{\alpha(1+\epsilon_n)}(\wh{\theta},\wh{\sigma}))\geq u\right) du +\P(\mathcal{A}^c)\\
&\leq \eta + \E\left(  \TDP(\pi, \BH_{\alpha(1+\epsilon_n)}(\wh{\theta},\wh{\sigma})) \right)  +\P(\mathcal{A}^c)\ .
\end{align*}
As a result, 
$$
\limsup_n \{\E\left( \TDP(\pi, \BH_\alpha(\theta))  \right)  - \E\left(  \TDP(\pi, \BH_{\alpha(1+\epsilon_n)}(\wh{\theta},\wh{\sigma})) \right)\}\leq \eta.
$$
This gives the result by making 
$\eta$ tends to $0$.

\subsection{Proof of Theorem~\ref{th:simescorrected}}\label{p:th:simescorrected}
By using \eqref{fromperfecttoapprox}, we obtain
\begin{align*}
&\P\left[ \exists  \ell \in \{1,\dots,n_0\}\::\: p_{(\ell:\cH_0)}(\hat{\theta},\hat{\sigma})\leq \alpha \ell/n,\:\hat{\theta}- \theta   >  -  \sigma\delta_n,\:\hat{\sigma}- \sigma   >  -  \sigma\delta_n\right]\\ 
&=\P\left[ \exists  \ell \in \{1,\dots,n_0\}\::\: p^\star_{(\ell:\cH_0)}\leq U_{\hat{\theta},\hat{\sigma}}(\alpha \ell/n),\:\hat{\theta}- \theta   >  -  \sigma\delta_n,\:\hat{\sigma}- \sigma   >  -  \sigma\delta_n\right]\\ 
&\leq \P\left[ \exists  \ell \in \{1,\dots,n_0\}\::\: p^\star_{(\ell:\cH_0)}\leq U_{\theta-\sigma\delta_n,\sigma - \sigma\delta_n}(\alpha \ell/n) \right]\ .
\end{align*}
because the quantity $U_{u,s}(\alpha \ell/n)$ is nonincreasing both in $u$ and $s$ (since $\ol{\Phi}^{-1}\left(\alpha\right)\geq 0$).
Now using the classical Simes inequality \eqref{Simesperfect}, the last display is upper-bounded by
\begin{align*}
&  \P\left[ \exists  \ell \in \{1,\dots,n_0\}\::\: p^\star_{(\ell:\cH_0)}\leq 
\max_{\alpha/n \leq t \leq \alpha} \left\{ \frac{U_{\theta-\sigma\delta_n,\sigma - \sigma\delta_n}(t)}{t}  \right\} 
\alpha \ell/n  \right]\\ 
&\leq \max_{\alpha/n \leq t \leq \alpha} \left\{ \frac{U_{\theta-\sigma\delta_n,\sigma - \sigma\delta_n}(t)}{t}  \right\} 
\alpha = \alpha + \alpha\max_{\alpha/n \leq t \leq \alpha} \left\{ \frac{U_{\theta-\sigma\delta_n,\sigma - \sigma\delta_n}(t)-t}{t}  \right\} \\
&\leq  \alpha+c  \delta_n \log (n) \  ,
\end{align*}
for some constant $c>0$, by applying Lemma~\ref{lem:majorationnull} and Proposition~\ref{rem:estimator}.

\subsection{Proof of Corollary~\ref{cor:Simescorrected}}\label{p:cor:Simescorrected}

Denote
$$
R_\ell = \{ i\in \{1,\dots,n\}\::\: p_i(\hat{\theta},\hat{\sigma})\leq \alpha \ell/n\}, \:\: 1\leq \ell \leq n\ .
$$
From \eqref{Simescorrected}, with probability at least $1-\alpha - c \log(n)/n^{1/16}$, the following event holds true:
\begin{align*}
\mathcal{E}&=\{\forall  \ell \in \{1,\dots,n_0\}\:,\: p_{(\ell:\cH_0)}(\hat{\theta},\hat{\sigma})> \alpha \ell/n\}\\
&=\{\forall  \ell \in \{1,\dots,n_0\}\:,\: |R_\ell\cap \cH_0| \leq \ell-1 \} \ ,
\end{align*}
Now, on the event $\mathcal{E}$, we have for any $S\subset\{1,\dots,n\}$, and for any $\ell\in \{1,\dots,n\}$,
\begin{align*}
|S\cap \cH_0| &= |S\cap \cH_0 \cap R_\ell| +|S\cap \cH_0 \cap R^c_\ell|\\
&\leq  | \cH_0 \cap R_\ell|+|S \cap R^c_\ell|\\
&\leq \ell-1 + |S \cap R^c_\ell|\ .
\end{align*}
By taking the minimum in $\ell$ in the latter relation, we have 
$$
|S\cap \cH_0| \leq \min_{\ell \in\{1,\dots,n\}} \{\ell-1 + |S \cap R^c_\ell|\},
$$
so that 
\begin{align*}
\FDP(\pi,S)  &= \frac{|S \cap \cH_0 |}{|S|\vee 1}
\leq \frac{ \min_{\ell \in\{1,\dots,n\}} \{\ell-1 + |S \cap R^c_\ell|\} }{|S|\vee 1},
\end{align*}
which yields \eqref{posthocbound}.

\subsection{Auxiliary lemmas}

The next lemma is a well known property of step-up procedure in the multiple testing theory, see, e.g., \cite{FZ2006} and Lemma~7.1 in \cite{RV2011}.

\begin{lem}\label{lem:MT}
Consider arbitrary $u\in\R$, $s>0$ and $\alpha\in (0,1)$.
The rejection number  $\wh{\ell}(u,s)$ of the Benjamini-Hochberg procedure $\BH_\alpha(u,s)$ (defined in Section~\ref{settingMT}) satisfies the following: for all $i\in\{1,\dots,n\}$,
\begin{align*}
\left\{ p_i(u,s) \leq \alpha  \wh{\ell}(u,s)/n \right\}=\left\{ p_i(u,s) \leq \alpha  \wh{\ell}^{(i)}(u,s)/n \right\} = \left\{\wh{\ell}^{(i)}(u,s)=\wh{\ell}(u,s)\right\}\ ,
\end{align*}
where $\wh{\ell}^{(i)}(u,s)$ is the rejection number of the Benjamini-Hochberg procedure applied to the $p$-value set $\{0,p_j(u,s),j\neq i\}$, that is, to the $p$-value set where $p_i(u,s)$ has been replaced by $0$.
\end{lem}

\begin{lem}\label{lem:MT2}
Fix $\theta\in \R$ and $\sigma>0$ and $U_{\cdot}(\cdot)$ as in \eqref{equUu}. Consider arbitrary $u\in\R$, $s>0$ and $\alpha\in (0,1)$. Then
$$
\wh{t}_{\alpha}(u,s) = \max\left\{t\in[0,1]\::\: \wh{G}_{u,s}(t) \geq t/\alpha\right\}\ ,
$$
where $\wh{G}_{u,s}(t) = n^{-1}\sum_{i=1}^n \mathds{1}_{\{p_i(u,s)\leq t\}} =n^{-1}\sum_{i=1}^n  \mathds{1}_{\{p_i(\theta,\sigma)\leq U_{u,s}(t)\}}=\wh{G}_{\theta,\sigma}(U_{u,s}(t))$
\end{lem}

\begin{lem}\label{lem:MT3}
Fix $\theta\in \R$ and $\sigma>0$ and $U_{\cdot}(\cdot)$ as in \eqref{equUu}.
Consider arbitrary $u\in\R$, $s>0$ and $\alpha\in (0,1)$. Then
$$
\wh{\l}_{\alpha_0}(u,s) \geq  \wh{\l}_{\alpha}(\theta,\sigma), \mbox{ \: where \: }\alpha_0= \alpha \frac{U_{u,s}^{-1}(\wh{t}_{\alpha}(\theta,\sigma))}{\wh{t}_{\alpha}(\theta,\sigma)}\ .
$$
\end{lem}

\begin{proof}
Denoting $t_0 = \wh{t}_{\alpha}(\theta,\sigma)$ and invoking Lemma \ref{lem:MT2}, we have
$$
\wh{G}_{u,s}\left( U_{u,s}^{-1}(t_0)  \right) = \wh{G}_{\theta,\sigma}\left( t_0  \right) \geq  t_0/\alpha =  \frac{t_0}{U_{u,s}^{-1}(t_0)} \frac{U_{u,s}^{-1}(t_0)}{\alpha} = \frac{U_{u,s}^{-1}(t_0)}{\alpha_0}\ .
$$
By using again Lemma \ref{lem:MT2}, this gives $\wh{t}_{\alpha_0}(u,s) \geq  U_{u,s}^{-1}(t_0) $. Hence, $\wh{t}_{\alpha_0}(u,s) \geq  \frac{\alpha_0}{\alpha} \:\wh{t}_{\alpha}(\theta,\sigma)$, which gives the result.
\end{proof}

\begin{lem}\label{lem:majorationnull}
There exists a universal constant $c>0$ 
 such that the following holds. For all $\alpha\in (0,0.4)$, for all $x,y\geq 0$ and $t_0\in(0,\alpha)$, we have
\begin{align}\label{equ:majorationnull}
\max_{t_0 \leq t \leq \alpha}\left\{ \frac{U_{\theta-x,\sigma-y}(t)-t}{t}  \right\}&\leq c \left(\frac{x}{\sigma}(2\log (1/t_0))^{1/2} +\frac{y}{\sigma} 2\log(1/t_0) \right)\ .
\end{align}
provided that $(x/\sigma) (2\log(1/t_0))^{1/2}+(y/\sigma) 2\log (1/t_0)  \leq 0.05$, and
where $U_{\theta-x,\sigma-y}(\cdot)$ is defined by \eqref{equUu}.
\end{lem}

\begin{proof}
First note that by \eqref{equUu}, we have
\begin{align*}
U_{\theta-x,\sigma-y}(t)= \ol{\Phi}\left(\ol{\Phi}^{-1}\left(t\right)- z(t)\right)\ , \:\:z(t)=\frac{y}{\sigma}\ol{\Phi}^{-1}\left(t\right) +\frac{x}{\sigma}.
\end{align*}
By Lemma \ref{lem:quantile}, we have for all $t\in[t_0,\alpha]$, 
$$
z(t) \leq \frac{y}{\sigma} (2\log(1/t))^{1/2} +\frac{x}{\sigma} \leq 0.05 (2\log (1/t))^{-1/2} \leq 0.05/ \ol{\Phi}^{-1}\left(t\right)\ ,
$$
where we used the assumption of the lemma. Now using that $\ol{\Phi}(\sqrt{0.05})\geq 0.4 \geq t$, we deduce
$
z(t)\leq \ol{\Phi}^{-1}\left(t\right)
$
for all $t\in[t_0,\alpha]$. Also deduce that for such a value of $t$,
$$
\frac{\phi\left(\ol{\Phi}^{-1}(t)-z(t)\right) }{\phi\left(\ol{\Phi}^{-1}(t)\right) } = e^{-z^2(t)/2} e^{z(t) \ol{\Phi}^{-1}(t)} \leq e^{z(t) \ol{\Phi}^{-1}(t)}  \leq e^{0.05} \leq 2.
$$
Now, since $\ol{\Phi}$ is decreasing and its derivative is  $-\phi$, we have for all $t\in[t_0,\alpha]$
\begin{align*}
\ol{\Phi}(\ol{\Phi}^{-1}(t)-z(t)) - \ol{\Phi}(\ol{\Phi}^{-1}(t))
&\leq z(t)\: \phi\left(\ol{\Phi}^{-1}(t)-z(t)\right) \\
&\leq z(t)\: \frac{\phi\left(\ol{\Phi}^{-1}(t)-z(t)\right) }{\phi\left(\ol{\Phi}^{-1}(t)\right) }  \phi\left(\ol{\Phi}^{-1}(t)\right)\\
&\leq 2z(t)  \phi\left(\ol{\Phi}^{-1}(t)\right)\\
& \leq 2z(t)\left(1 + \left( \ol{\Phi}^{-1}(t)\right)^{-2}\right)\: t \:\ol{\Phi}^{-1}(t),
\end{align*}
by using Lemma \ref{lem:quantile}. Finally, the last display is smaller than
$$
z(t_0)\: t \:\ol{\Phi}^{-1}(t_0) 2\left(1 + \left( \ol{\Phi}^{-1}(0.4)\right)^{-2}\right),
$$
which gives \eqref{equ:majorationnull}. 
\end{proof}

\section{Auxiliary results}\label{sec:appendix}

\subsection{Chebychev polynomials}\label{tchebysection}

In this subsection, we first remind the reader of the definition and important properties of Chebychev polynomials.
For any $k\geq 1$, the $k$-th Chebychev polynomial is defined by 
$$
T_k(x) = (k/2) \sum_{j=0}^{\lfloor k/2\rfloor} (-1)^j  \frac{(k-j-1)!}{j!(k-2j)!} (2x)^{k-2j}\ . 
$$
It satisfies the following inequalities. 
\begin{prp}\label{prop:Chebychev}
For $k\geq 1$, 
the polynomial $T_k$
 satisfies the following properties:
\begin{itemize}
\item[(i)]  for all $y\in[0,1]$, $T_k(1-2y)=T_{2k}(\sqrt{1-y})$, and thus for all $y\in\R$,
 \begin{equation}\label{equ:chebexplicit}
 T_k(1-2y)= \sum_{j=0}^{k} (-4)^j\frac{k(k+j-1)!}{(k-j)!(2j)!}{y^j}\ .
 \end{equation} 
\item[(ii)] for $x\in \R$,  
$$
T_k(x)=\left\{\begin{array}{ll}\cos(k\:\arccos x) & \mbox{ if $x\in [-1,1]$ ;}\\
\cosh(k\:\arccosh x) & \mbox{ if $x\geq 1$ ;}\\
(-1)^k\cosh(k\:\arccosh (-x)) & \mbox{ if $x\leq -1$ .}
\end{array}\right.
$$

\end{itemize}
\end{prp}

Next, we remind the reader of extremal properties satisfied by the Chebychev polynomial. The first inequality is a consequence of Chebychev's Theorem whereas the second inequality is a consequence of Markov's theorem. Both may be found in \cite{Gam1990}, Page 119.

\begin{lem}\label{lem:optimalTN}
Denote by $\mathcal{P}_k$ the set of polynomials of degree smaller than or equal to $k$. Let $a,b,c \in\R$ with $a<b<c$.
 Then we have, 
 \beqn 
\sup_{\substack {P\in  \mathcal{P}_k \\ \|P\|_{\infty,[b,c]}\leq 1}} |P(a)| = \left|T_{k}\left( 1+2 \frac{b-a}{c-b}\right)\right|\ ,\\
\sup_{x \in[a,c]} |P'(x)| \leq \frac{2 k^2}{c-a}  \sup_{x \in[a,c]} |P(x)|\ \quad \quad \forall P\in  \mathcal{P}_k\ .
 \eeqn 
\end{lem}

The coefficient of the polynomial defined in \eqref{eq:defintion_g} are given explicitly using the explicit expression of Chebychev polynomials: 
\begin{equation}\label{eq:ajq}
a_{j,q}=(-4)^j\frac{q(q+j-1)!}{(q-j)!(2j)!}, \:\:\:\:0\leq  j \leq q\ ,
\end{equation}

\begin{lem}\label{lem:a_j_q}
 For any even integer $q$ and any integer $j\in [0,q]$, we have
 \beq\label{eq:upper_a_j_q}
  |a_{j,q}|\leq(3+2\sqrt{2})^{q}
 \eeq
\end{lem}

\begin{proof}
This upper bound is obviously true for $j=0$ and $j=q$. Henceforth, we restrict ourselves to the case $j\in [1,q-1]$ (and therefore $q\geq 2$). For any positive integer $n$, Stirling's inequalities ensure that $n!e^nn^{-n-1/2}\in [\sqrt{2\pi},e]$. This leads us to 
\[
 |a_{j,q}|\leq 4^j \frac{e}{2\pi}   \frac{(q+j)^{q+j+1/2}}{(q-j)^{q-j+1/2}(2j)^{2j+1/2}}\leq \frac{e}{2\pi} \sqrt{\frac{q+j}{2(q-j)j}} \frac{(q+j)^{(q+j)}}{j^{2j}(q-j)^{q-j} }\ . 
 \]
 Since we assume that $j\in [1,q-1]$ and since the function $x\mapsto x(1-x)$ is increasing on $(0,1/2)$ and decreasing on $(1/2,1)$, we obtain
 \beqn 
  \sqrt{\frac{q+j}{2(q-j)j}}\leq \sqrt{\frac{2q}{2q(1-1/q)}}\leq \sqrt{\frac{1}{1-1/q}}\leq \sqrt{2}\ , 
 \eeqn 
which leads us to 
 \beq\label{eq:first_upper_a_j_q}
|a_{j,q}|\leq \frac{e}{\sqrt{2}\pi}\exp\left[q\left((1+\frac{j}{q})\log(1+\frac{j}{q})- 2\frac{j}{q}\log(\frac{j}{q}) -(1-\frac{j}{q})\log(1-\frac{j}{q})\right)\right]\ .
\eeq
Consider the function $h:x\mapsto (1+x)\log(1+x) -(1-x)\log(1-x)-2x\log(x)$ defined on $(0,1)$. Relying on standard derivation arguments, we observe that $h$ achieves its maximum $x_0= 1/\sqrt{2}$ and that $h(x_0)=\log(3+2\sqrt{2})$. Coming back to \eqref{eq:first_upper_a_j_q}, we have proved that 
\[
 |a_{j,q}|\leq (3+2\sqrt{2})^{q}
\]

\end{proof}

\subsection{Inequalities for Gaussian quantile and upper-tail functions}\label{sec:ineq-quantile}

\begin{lem}[Quantile function of the normal distribution]\label{lem:quantile}
We have 
\beq 
\max\left(\frac{t\phi(t)}{1+t^2}, \frac{1}{2}- \frac{t}{\sqrt{2\pi}}\right) \leq \ol{\Phi}(t)\leq \phi(t) \min\left(\frac{1}{t}, \sqrt{\frac{\pi}{2}}\right) , \:\:\:\:\mbox{ for all $t>0$}\label{eq:maj-fonctionrepgauss}\ . 
\eeq
As a consequence, for any $x< 0.5$, we have
\begin{align}
\sqrt{2\pi}(1/2- x)&\leq \overline{\Phi}^{-1}(x)\leq \sqrt{2\log\left(\frac{1}{2x}\right)} \ , \label{eq:encadrement_quantile_1}\\
   \log\left( \frac{[\overline{\Phi}^{-1}(x)]^2}{[\overline{\Phi}^{-1}(x)]^2+1}\right) &\leq \frac{[\overline{\Phi}^{-1}(x)]^2}{2}  -  \log\left(\frac{1}{x}\right) + \log\left(\sqrt{2\pi} \overline{\Phi}^{-1}(x)\right)\leq 0\ ,\label{eq:encadrement_quantile}
\end{align}
and if additionally $x\leq 0.004$, we have
\begin{align}
\overline{\Phi}^{-1}(x)\geq \sqrt{\log\left(\frac{1}{x}\right)} \label{eq:encadrement_quantile_1plus}\ .
\end{align}

\end{lem}

\begin{proof}[Proof of Lemma \ref{lem:quantile}]
Inequality \eqref{eq:maj-fonctionrepgauss} is standard.  
Relation \eqref{eq:encadrement_quantile_1} (resp. \eqref{eq:encadrement_quantile}) is a consequence of $ {1}/{2}- {t}/\sqrt{2\pi} \leq \ol{\Phi}(t)\leq \phi(t)\sqrt{\tfrac{\pi}{2}}$ (resp. $(t^2/(1+t^2)) \phi(t)/t\leq \ol{\Phi}(t)\leq \phi(t)/t$ ).
The last relation \eqref{eq:encadrement_quantile_1plus} comes from \eqref{eq:encadrement_quantile}, because for $x\leq 0.004$ (thus $\overline{\Phi}^{-1}(x)\geq 1$), we have
\begin{align*}
[\overline{\Phi}^{-1}(x)]^2 &\geq 2 \log\left(\frac{1}{x}\right) -2 \log\left(\sqrt{2\pi} \overline{\Phi}^{-1}(x)\frac{[\overline{\Phi}^{-1}(x)]^2+1}{[\overline{\Phi}^{-1}(x)]^2}\right)\\
&\geq 2 \log\left(\frac{1}{x}\right) -2 \log\left(2 \sqrt{2\pi} \overline{\Phi}^{-1}(x) \right)\ ,
\end{align*}
which is larger than $\log\left(\frac{1}{x}\right) $ provided that $16\pi x \log\left(\frac{1}{2x}\right)\leq 1$ by \eqref{eq:encadrement_quantile_1}. This last bound holds for $x\leq 0.004$ by monotonicity.
\end{proof}

\begin{lem}\label{lem:difference_quantile}
We have 
\beq\label{eq:upper_difference_quantile}
 \overline{\Phi}^{-1}(x)- \overline{\Phi}^{-1}(y)\leq\left\{
\begin{array}{cc}
3|y-x|   &\text{ if } 0.3\leq x \leq y \leq 0.7\ ; \\
\frac{|y-x|}{x \overline{\Phi}^{-1}(x)}& \text{ if } x< \min (y, 1-y)\ ;\\
\frac{1}{\overline{\Phi}^{-1}(y) }\left[\log\left(\frac{y}{x}\right)+  \frac{1}{[\overline{\Phi}^{-1}(y)]^2} \right] & \text{ if } x\leq  y < 0.5\ .
\end{array}
 \right.   
\eeq
Besides, we also have 
\beq\label{eq:lower_difference_quantile}
 \overline{\Phi}^{-1}(x)- \overline{\Phi}^{-1}(y)\geq\left\{
\begin{array}{cc}
2.5|y-x|   &\text{ if } x \leq y \  ;\\
\left(\frac{|\overline{\Phi}^{-1}(y)|^2}{1+|\overline{\Phi}^{-1}(y)|^2}\right)\frac{|y-x|}{y \overline{\Phi}^{-1}(y)}& \text{ if } x\leq y < 0.5\ ;\\
\frac{1}{\overline{\Phi}^{-1}(x) }\left[\log\left(\frac{y}{ex}\right)+  \frac{1}{2}\log\log\left(\frac{1}{y}\right)- \frac{1}{2}\log\log\left(\frac{1}{x}\right)\right]
 & \text{ if } x\leq  y \leq 0.004\ .
\end{array}
 \right. \   
\eeq

\end{lem}

\begin{proof}[Proof of Lemma \ref{lem:difference_quantile}]
We start by proving the two first inequalities of each bound   \eqref{eq:upper_difference_quantile} and \eqref{eq:lower_difference_quantile}. By the mean-value theorem, we have 
\beq\label{eq:lower_mean_value_theorem}
\frac{y-x}{\sup_{z\in [x,y]}\phi(\overline{\Phi}^{-1}(z))} \leq \overline{\Phi}^{-1}(x) - \overline{\Phi}^{-1}(y) \leq \frac{y-x}{\inf_{z\in [x,y]}\phi(\overline{\Phi}^{-1}(z))}\ . 
\eeq
The function $t\mapsto \phi(\overline{\Phi}^{-1}(t+1/2))$ defined on $[-1/2,1/2]$  is symmetric and increasing on $[-1/2,0]$. Thus if $0.3\leq x\leq y\leq 0.7$, the above  infimum equals 
$\phi(\overline{\Phi}^{-1}(0.3))$ which is larger than $1/3$. This proves the first inequality. Turning to the second inequality ($x\leq \min(y,1-y)$), the above infimum is achieved at $z=x$ which yields. 
\beqn
0\leq \overline{\Phi}^{-1}(x) - \overline{\Phi}^{-1}(y) \leq \frac{|y-x|}{\phi(\overline{\Phi}^{-1}(x))}\ . 
\eeqn
It follows  from \eqref{eq:encadrement_quantile}, that $\phi(\overline{\Phi}^{-1}(x))\geq x |\overline{\Phi}^{-1}(x)|$, which yields the second result. 
Turning to the lower bounds, we still apply the mean theorem \eqref{eq:lower_mean_value_theorem} and observe that $\phi(\overline{\Phi}^{-1}(z))\leq 1/\sqrt{2\pi}\leq (2.5)^{-1}$, which proves the first bound in \eqref{eq:lower_difference_quantile}. As for the second bound in \eqref{eq:lower_difference_quantile}, the maximum of $\phi(\overline{\Phi}^{-1}(z))$ is achieved at $z=y$ and the result follows from  \eqref{eq:encadrement_quantile} in Lemma \ref{lem:quantile}.

Finally, we consider the last bounds in \eqref{eq:upper_difference_quantile} and \eqref{eq:lower_difference_quantile}. We first apply the  mean value theorem to the  square root function. 
\[ 
\frac{[\overline{\Phi}^{-1}(x)]^2 - [\overline{\Phi}^{-1}(y)]^2}{2\overline{\Phi}^{-1}(x) } \leq \overline{\Phi}^{-1}(x)- \overline{\Phi}^{-1}(y) \leq \frac{[\overline{\Phi}^{-1}(x)]^2 - [\overline{\Phi}^{-1}(y)]^2}{2\overline{\Phi}^{-1}(y) }\ .
\]
For the upper bound, we use Lemma \ref{lem:quantile} for $x$ and $y$ to get 
\beqn 
 \overline{\Phi}^{-1}(x)- \overline{\Phi}^{-1}(y) &\leq &  \frac{1}{\overline{\Phi}^{-1}(y) }\left[\log\left(\frac{y}{x}\right)+ \log\left(\frac{\overline{\Phi}^{-1}(y)}{\overline{\Phi}^{-1}(x)}\right)- \log\left( \frac{[\overline{\Phi}^{-1}(y)]^2}{1+[\overline{\Phi}^{-1}(y)]^2} \right) \right]\\
&\leq &  \frac{1}{\overline{\Phi}^{-1}(y) }\left[\log\left(\frac{y}{x}\right)+  \frac{1}{[\overline{\Phi}^{-1}(y)]^2 } \right]\ .
\eeqn 
For the lower bound, we use again Lemma \ref{lem:quantile} together with $\overline{\Phi}^{-1}(x)\geq \overline{\Phi}^{-1}(0.004)$, to get
\beqn 
 \overline{\Phi}^{-1}(x)- \overline{\Phi}^{-1}(y) &\geq &  \frac{1}{\overline{\Phi}^{-1}(x) }\left[\log\left(\frac{y}{x}\right)+ \log\left(\frac{\overline{\Phi}^{-1}(y)}{\overline{\Phi}^{-1}(x)}\right)+ \log\left( \frac{[\overline{\Phi}^{-1}(x)]^2}{1+[\overline{\Phi}^{-1}(x)]^2} \right) \right]\\
 &\geq & \frac{1}{\overline{\Phi}^{-1}(x) }\left[\log\left(\frac{y}{x}\right)+ \frac{1}{2}\log\left(\frac{ \log(1/y)}{2\log(1/x)}\right)+ \log\left( \frac{[\overline{\Phi}^{-1}(0.004)]^2}{1+[\overline{\Phi}^{-1}(0.004)]^2} \right) \right]
 \\
&\geq &  \frac{1}{\overline{\Phi}^{-1}(x) }\left[\log\left(\frac{y}{x}\right)+  \frac{1}{2}\log\log\left(\frac{1}{y}\right)- \frac{1}{2}\log\log\left(\frac{1}{x}\right)-1\right]\ ,
\eeqn 
which concludes the proof.
\end{proof}

\begin{lem}\label{lem:biais_estimateur_quantile}
There exists some universal constant $c$ such that the following holds for all positive integers $n$, $ k\leq n-1$ and all positive  integers $q\leq 0.7(n-k)$:
\[
\overline{\Phi}^{-1}({q}/{n})- \overline{\Phi}^{-1}({q}/{(n-k)})\leq c \frac{\log\left(\frac{n}{n-k}\right)}{\sqrt{\log\big(\frac{n-k}{q}\big)_+}\vee 1}\ .
\]

\end{lem}
 
\begin{proof}
We first consider the case $q\geq n/3$. Since $q\leq 0.7(n-k)$, it follows that $n\leq 2.1(n-k)$. We then deduce from 
\eqref{eq:upper_difference_quantile} that  
\[
\overline{\Phi}^{-1}\Big(\frac{q}{n}\Big)- \overline{\Phi}^{-1}\Big(\frac{q}{n-k}\Big)\leq 3\frac{qk}{n(n-k)}\leq 3 \frac{k}{n-k}\lesssim \log\left(\frac{n}{n-k}\right) \lesssim  \frac{\log\left(\frac{n}{n-k}\right)}{\sqrt{\log\big(\frac{n-k}{q}\big)_+}\vee 1}\ ,
\]
because $n/(n-k)\leq 2.1$ and $(n-k)/q\leq 3$.

Now assume  $q\leq n/3$ and $k<n/2$. It follows from the second inequality in \eqref{eq:upper_difference_quantile} (which can be used because $q/n<q/(n-k)$ and $q/n \leq 1-2q/n < 1-q/(n-k)$) that 
\[
 \overline{\Phi}^{-1}\Big(\frac{q}{n}\Big)- \overline{\Phi}^{-1}\Big(\frac{q}{n-k}\Big)\lesssim \frac{k}{(n-k)\ol{\Phi}^{-1}(q/n)}\lesssim  \frac{\log\left(\frac{n}{n-k}\right)}{\ol{\Phi}^{-1}(q/n)}\ ,
\]
because $k/(n-k)\leq 1$. Also, for $x\leq 0.004$, we have $\ol{\Phi}^{-1}(x) \geq \sqrt{\log(1/x)}$ by  \eqref{eq:encadrement_quantile_1plus}. This implies $\ol{\Phi}^{-1}(x) \gtrsim  \sqrt{\log(1/x)}$ for all $x\leq 1/3$, and we obtain 
\[
 \overline{\Phi}^{-1}\Big(\frac{q}{n}\Big)- \overline{\Phi}^{-1}\Big(\frac{q}{n-k}\Big)\lesssim \frac{\log\left(\frac{n}{n-k}\right)}{\sqrt{\log\big(\frac{n}{q}\big)_+}\vee 1}\lesssim \frac{\log\left(\frac{n}{n-k}\right)}{\sqrt{\log\big(\frac{n-k}{q}\big)_+}\vee 1}\ .
\]

 Then, we consider the case where $q\leq n/3$, $k\geq n/2$, and $q\leq 0.4(n-k)$. It follows from the last inequality  in \eqref{eq:upper_difference_quantile} that 
\[
 \overline{\Phi}^{-1}\Big(\frac{q}{n}\Big)- \overline{\Phi}^{-1}\Big(\frac{q}{n-k}\Big)\leq \frac{\log\left(\frac{n}{n-k}\right)+ \frac{1}{[\overline{\Phi}^{-1}(\frac{q}{n-k})]^2}}{\overline{\Phi}^{-1}(\frac{q}{n-k})}\lesssim\frac{\log\left(\frac{n}{n-k}\right)}{\sqrt{\log\big(\frac{n-k}{q}\big)_+}\vee 1}\ ,
\]
where we used in the last inequality that $\overline{\Phi}^{-1}(\frac{q}{n-k})\gtrsim \sqrt{\log\big(\frac{n-k}{q}\big)_+}$ for $q/(n-k)\leq 0.4$. 

Finally, we assume that $q\leq n/3$, $k\geq n/2$ and $q/(n-k)\in (0.4,0.7]$. Then, it follows from \eqref{eq:encadrement_quantile_1} that 
\beqn 
 \overline{\Phi}^{-1}\Big(\frac{q}{n}\Big)- \overline{\Phi}^{-1}\Big(\frac{q}{n-k}\Big)&\leq& \overline{\Phi}^{-1}\Big(\frac{q}{n}\Big)+\overline{\Phi}^{-1}(0.3)\leq \sqrt{2\log\big(n/q\big)}+ \overline{\Phi}^{-1}(0.3)\lesssim \sqrt{\log\big(n/q\big)}\\
 &\lesssim & \sqrt{\log\Big(\frac{n}{n-k}\Big)}\lesssim \frac{\log\left(\frac{n}{n-k}\right)}{\sqrt{\log\big(\frac{n-k}{q}\big)_+}\vee 1}\ .
\eeqn

\end{proof}

\subsection{Deviation inequalities for Gaussian empirical quantiles}

\begin{lem} \label{lem:quantile_empirique}
 Let $\xi=(\xi_{1},\ldots, \xi_{n})$ be a standard Gaussian vector of size $n$. 
 For any integer $q\in (0.3n, 0.7n)$, we have for all $0< x\leq \frac{8}{225}q \wedge\bigg( \frac{n^2}{18q}[\overline{\Phi}^{-1}(q/n)-\overline{\Phi}^{-1}(0.7)]^{2}\bigg)$,
\begin{align}
& \P\Big[\xi_{(q)}+ \overline{\Phi}^{-1}(q/n) \geq  3\frac{\sqrt{2qx}}{n}\Big]\leq  e^{-x} \ ,\label{equ1Nico}
 \end{align}
 and for all $0<x \leq \frac{n^2}{18q}[ \overline{\Phi}^{-1}(0.3)- \overline{\Phi}^{-1}(q/n)]^{2}$,
\begin{align} 
& \P\Big[\xi_{(q)}+ \overline{\Phi}^{-1}(q/n) \leq   -3\frac{\sqrt{2qx}}{n}\Big]\leq  e^{-x} \ . \label{equ2Nico}
 \end{align}
Now consider any integer $q\leq 0.4n$. For all $x\leq \tfrac{1}{8}q[(\overline{\Phi}^{-1}(q/n))^2\wedge (\overline{\Phi}^{-1}(q/n))^4]$, we have 
\begin{align}
 \P\left[\xi_{(q)}+ \overline{\Phi}^{-1}(q/n) \geq \frac{1}{\overline{\Phi}^{-1}(q/n)}\sqrt{\frac{2x}{q}} +  \frac{8}{3}\frac{x}{q\overline{\Phi}^{-1}(q/n)}\right]&\leq  e^{-x}  \ . \label{equ3Nico}
 \end{align}
 For  all $x\leq q/8$, we have  
\begin{align}\label{equ4Nico}
 \P\left[\xi_{(q)}+ \overline{\Phi}^{-1}(q/n) \leq  -\frac{2}{\overline{\Phi}^{-1}(q/n)}\sqrt{\frac{2x}{q}} \right]&\leq  e^{-x}  \ . 
 \end{align}
 \end{lem}

\begin{proof}[Proof of Lemma \ref{lem:quantile_empirique}]
Consider any  $t\geq 0$ and denote $p= \overline{\Phi}[\overline{\Phi}^{-1}(q/n)- t]$ which belongs to $(q/n,1)$. We have 
 \beq\label{eq:first_upper_proba}
 \P\Big[\xi_{(q)}\geq -\overline{\Phi}^{-1}(q/n)+t\Big]= \P\left[\mathcal{B}(n,p)\leq q-1\right]\leq\P\left[\mathcal{B}(n,p)\leq q\right]\ . 
 \eeq
 By the mean value theorem, we have $ p-q/n \geq t \inf_{x\in [0,t]}\phi[\overline{\Phi}^{-1}(q/n)-x]$. 
 Assume first that $q/n$ belongs to $(0.3,0.7)$ and that  $ \overline{\Phi}^{-1}(q/n)-t\geq \overline{\Phi}^{-1}(0.7)$. Then, it follows from the previous inequality that 
 $p-q/n \geq t \phi[\overline{\Phi}^{-1}(0.3)]\geq t/3$. Together with Bernstein's inequality, we obtain 
\beqn 
 \P\Big[\xi_{(q)}\geq -\overline{\Phi}^{-1}(q/n)+t\Big]&\leq & \P\left[\mathcal{B}(n,q/n+ t/3)\leq q\right]\\
 &\leq & \exp\left[- \frac{n^2t^2/9}{2(q+nt/3)(1-(q/n+t/3))+ 2nt/9}\right]\\
 &\leq & \exp\left[- \frac{n^2t^2/9}{1.4q + nt(1.4/3+ 2/9)}\right]\ ,
 \eeqn 
 where we used that $q/n\geq 0.3$ in the last line. If we further assume that $t\leq 0.8q/n$, we obtain 
 \[
  \P\Big[\xi_{(q)}\geq -\overline{\Phi}^{-1}(q/n)+t\Big]\leq \exp\left[-n^2t^2/(18q)\right]\ .
 \]
In view of the conditions $t\leq 0.8q/n$ and $t\leq \overline{\Phi}^{-1}(q/n)- \overline{\Phi}^{-1}(0.7)$, we have proved  \eqref{equ1Nico}. 

Let us now prove \eqref{equ3Nico}. Assume that $q/n\leq 0.4$ and $t\leq \overline{\Phi}^{-1}(q/n)$. This implies $p\leq 1/2$ and we have 
$ p-q/n \geq t \phi[\overline{\Phi}^{-1}(q/n)]\geq t(q/n)\overline{\Phi}^{-1}(q/n)$ by Inequality \eqref{eq:encadrement_quantile} in Lemma \ref{lem:quantile}. Then, \eqref{eq:first_upper_proba} together with Bernstein's inequality yields 
\beqn 
\P\Big[\xi_{(q)}\geq -\overline{\Phi}^{-1}(q/n)+t\Big]\leq \exp\left[- \frac{t^2 q^2 (\overline{\Phi}^{-1}(q/n))^2 }{2q + \frac{8}{3}tq\overline{\Phi}^{-1}(q/n)} \right]\ ,
\eeqn 
which implies \eqref{equ3Nico} by simple algebraic manipulations. 

\medskip 
 
Next, we consider the left deviations. For any $t>0$, we write $p= \overline{\Phi}[\overline{\Phi}^{-1}(q/n) + t]$. We have 
$$\P[\xi_{(q)}\leq -\overline{\Phi}^{-1}(q/n)-t]= \P[\mathcal{B}(n,p)\geq   q].$$ Then, Bernstein's inequality yields
\beq\label{eq:second_upper_proba} 
 \P\Big[\xi_{(q)}\leq -\overline{\Phi}^{-1}(q/n)-t\Big]\leq 
  \exp\left[- \frac{(q-np)^2}{2np(1-p)+ 2(q-np)/3}\right]\leq  \exp\left[- \frac{(q-np)^2}{2q}\right]\ ,
 \eeq
because $2np(1-p)+ 2(q-np)/3\leq (2-2/3)np+ 2q/3\leq 2q$ since $p\leq q/n$.
First, assume that $q/n\in (0.3,0.7)$ and that $\overline{\Phi}^{-1}(q/n)+  t \leq  \overline{\Phi}^{-1}(0.3)$.
Then, it follows from the mean value theorem that 
$
  q/n-p \geq t \inf_{x\in [0,t]}\phi\Big[\overline{\Phi}^{-1}(q/n)+x\Big]\geq t/3\ ,
$
which implies 
\[
 \P\Big[\xi_{(q)}\leq -\overline{\Phi}^{-1}(q/n)-t\Big]\leq \exp\left[- \frac{n^2 t^2 }{18q}\right]\ . 
\]
We have shown \eqref{equ2Nico}. 

Now assume that $q/n\leq 0.4$ and consider any $0<t<(\overline{\Phi}^{-1}(q/n))^{-1}$. It follows again from the mean value theorem and Lemma \ref{lem:quantile} that
\beqn 
q/n-p \geq t \phi\Big[\overline{\Phi}^{-1}(p)\Big]\geq t p\ \overline{\Phi}^{-1}(p)\geq t p\ \overline{\Phi}^{-1}(q/n)\  ,
\eeqn 
which implies $n p\leq q/(1+t\ \overline{\Phi}^{-1}(q/n))$ and thus
\[
 q - np \geq   q\left(1 - \frac{1}{1+t\ \overline{\Phi}^{-1}(q/n) }\right) =\frac{tq\ \overline{\Phi}^{-1}(q/n)}{1+t\ \overline{\Phi}^{-1}(q/n) }\geq  \tfrac{1}{2}tq\ \overline{\Phi}^{-1}(q/n)\ .
\] 
Coming back to \eqref{eq:second_upper_proba}, we get 
\[
 \P\Big[\xi_{(q)}\leq -\overline{\Phi}^{-1}(q/n)-t\Big]\leq 
    \exp\left[- \frac{t^2q (\overline{\Phi}^{-1}(q/n))^2}{8}\right]\ ,
 \] 
 which shows \eqref{equ4Nico}.
\end{proof}

\begin{lem} \label{lem:quantile_empirique_2}
 Let $\xi=(\xi_{1},\ldots, \xi_{n})$ be a standard Gaussian vector of size $n$. There exist two positive constants $c_1$ and $c_2$ such that the following holds: for any integer $1\leq q\leq n-1$ and any $x\leq c_1 q$, we have 
\begin{align}
 \P\left[\xi_{(q)}+ \overline{\Phi}^{-1}(q/n) \geq  c_2 \sqrt{\frac{x}{q[\log(\frac{n}{q})\vee 1]}}\right]&\leq  e^{-x}\label{cordevquantile1}\ ,\\
 \P\left[\xi_{(q)}+ \overline{\Phi}^{-1}(q/n) \leq  -c_2 \sqrt{\frac{x}{q[\log(\frac{n}{q})\vee 1]}}\right]&\leq  e^{-x} \label{cordevquantile2}\ .
 \end{align}
 \end{lem}

 \begin{proof}
First consider the case $q\leq 0.6n$.  It follows from Lemma \ref{lem:quantile} that $|\ol{\Phi}^{-1}(q/n)|\vee 1\lesssim \sqrt{\log(\frac{n}{q})}\vee 1$. Then, the result  is a straightforward consequence of Lemma \ref{lem:quantile_empirique} by gathering all the deviation bounds and taking  $c_1$ small enough and  $c_2$ large enough. For $q\geq 0.6 n$, we use the symmetry of the normal distribution and observe that $\xi_{(q)}$ is distributed as $-\xi_{(n-q)}$ while $\ol{\Phi}^{-1}(q/n)=-  \ol{\Phi}^{-1}(1-q/n)$.
 \end{proof}

\section{Additional numerical experiments}\label{supp:simu}

In this section, we provide numerical experiments for two scenarios for the alternatives:
\begin{itemize}
\item the alternatives $m_i,$ $1\leq i \leq n_1$, are linearly increasing from $0.01$ to $2\Delta$, that is, 
$$
m_i = 0.01 + (2\Delta-0.01)(i-1)/n_1  ,\:\:\: 1 \leq i \leq n_1 ;
$$
\item the alternatives $m_i,$ $1\leq i \leq n_1$, are generated as $n_1$ i.i.d. uniform variables in $(0.01,2\Delta)$ (previously and independently from the Monte-Carlo loop).
\end{itemize}

\begin{figure}[h!]
\begin{tabular}{ccc}
&FDP& TDP \\
\rotatebox{90}{\hspace{2cm}$\Delta=2$}&
\includegraphics[scale=0.27]{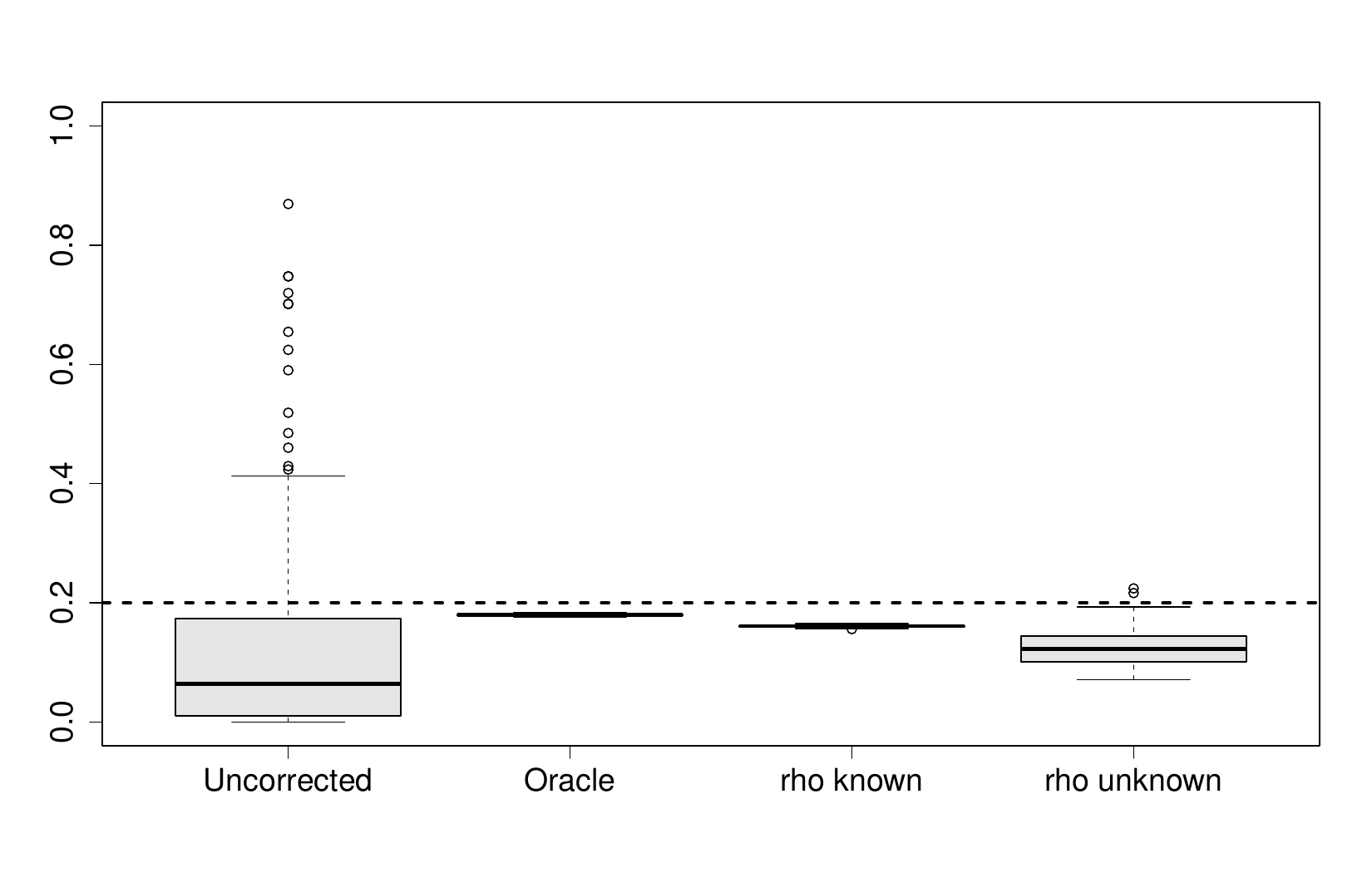}&
\includegraphics[scale=0.27]{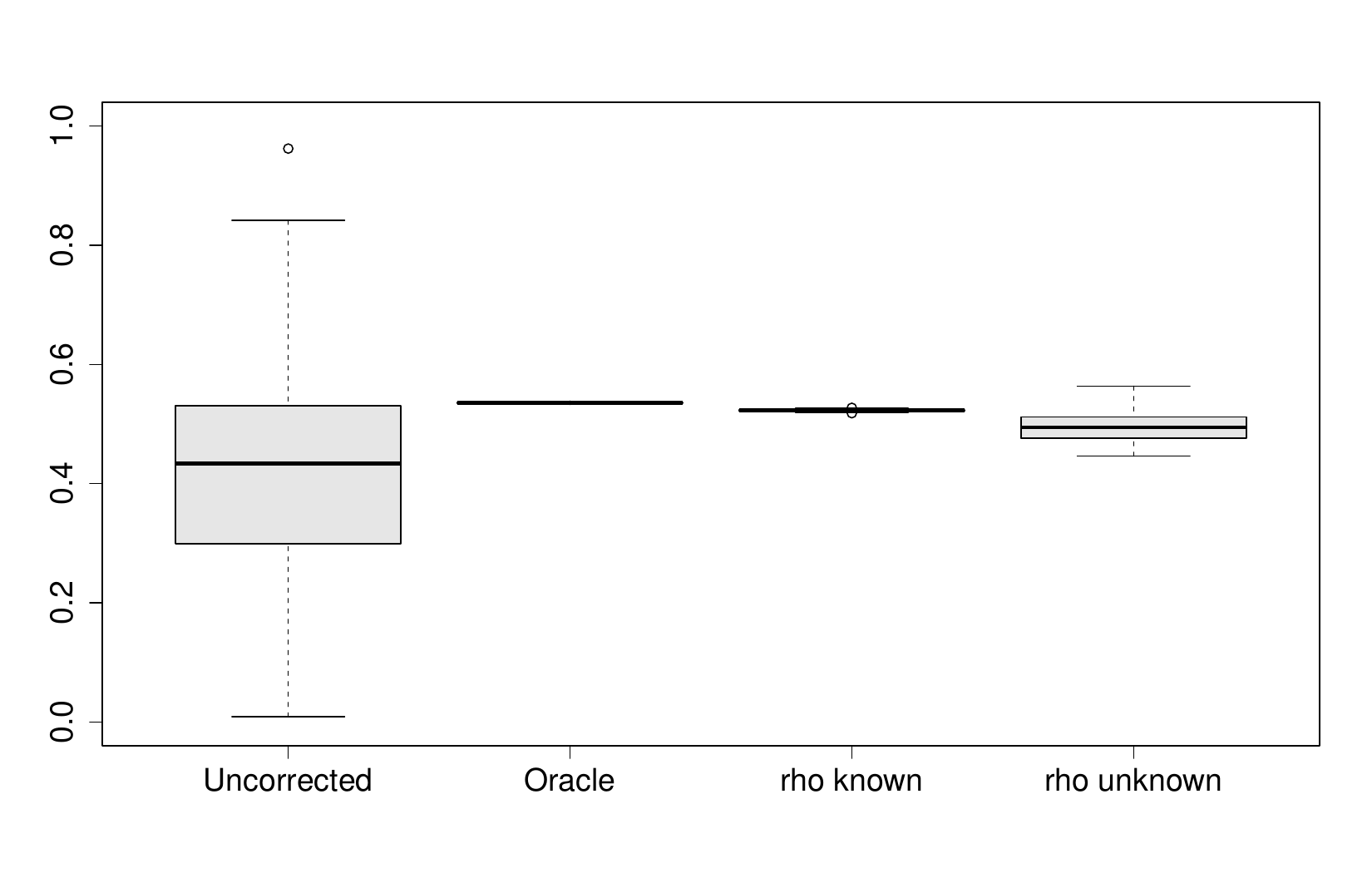}\\
\rotatebox{90}{\hspace{2cm}$\Delta=3$}&
\includegraphics[scale=0.27]{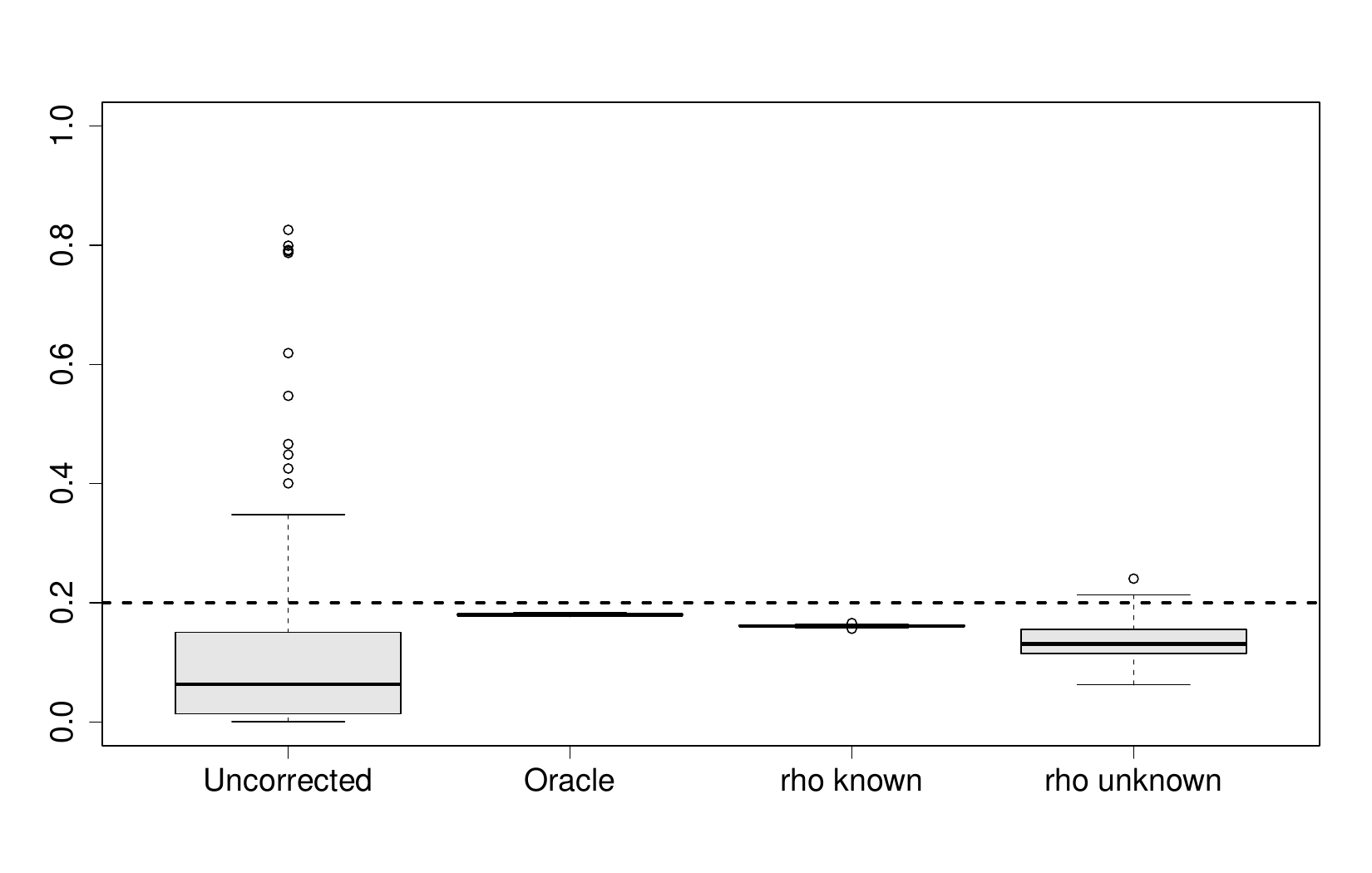}&
\includegraphics[scale=0.27]{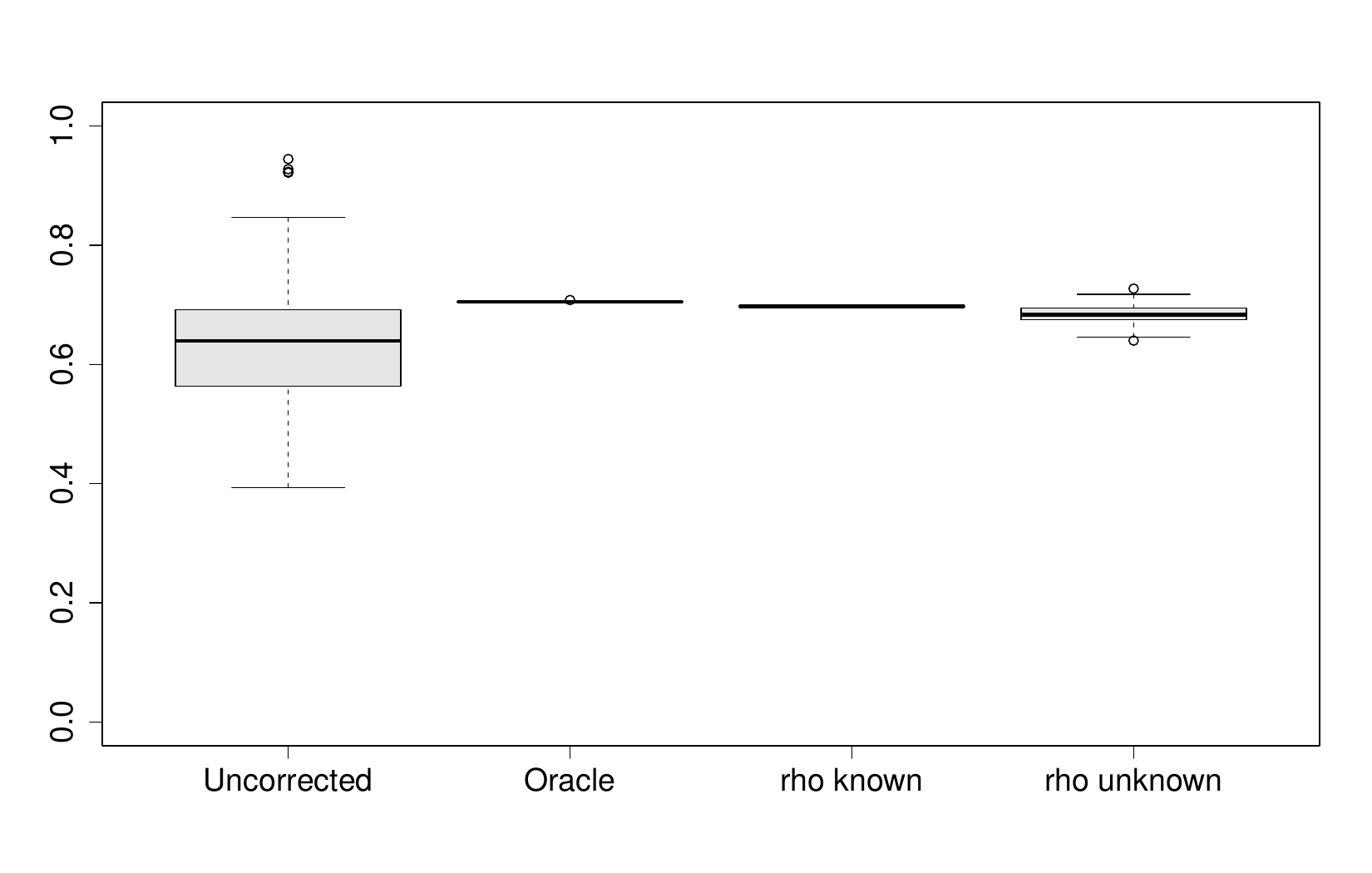}
\end{tabular}
\caption{Same as Figure~\ref{fig:boxplot} but for alternative $m_i$ linearly increasing from $0.01$ to $2\Delta$. } \label{fig:boxplot_lin}
\end{figure}

\begin{figure}[h!]
\begin{tabular}{ccc}
&FDP& TDP \\
\rotatebox{90}{\hspace{2cm}$\Delta=2$}&
\includegraphics[scale=0.27]{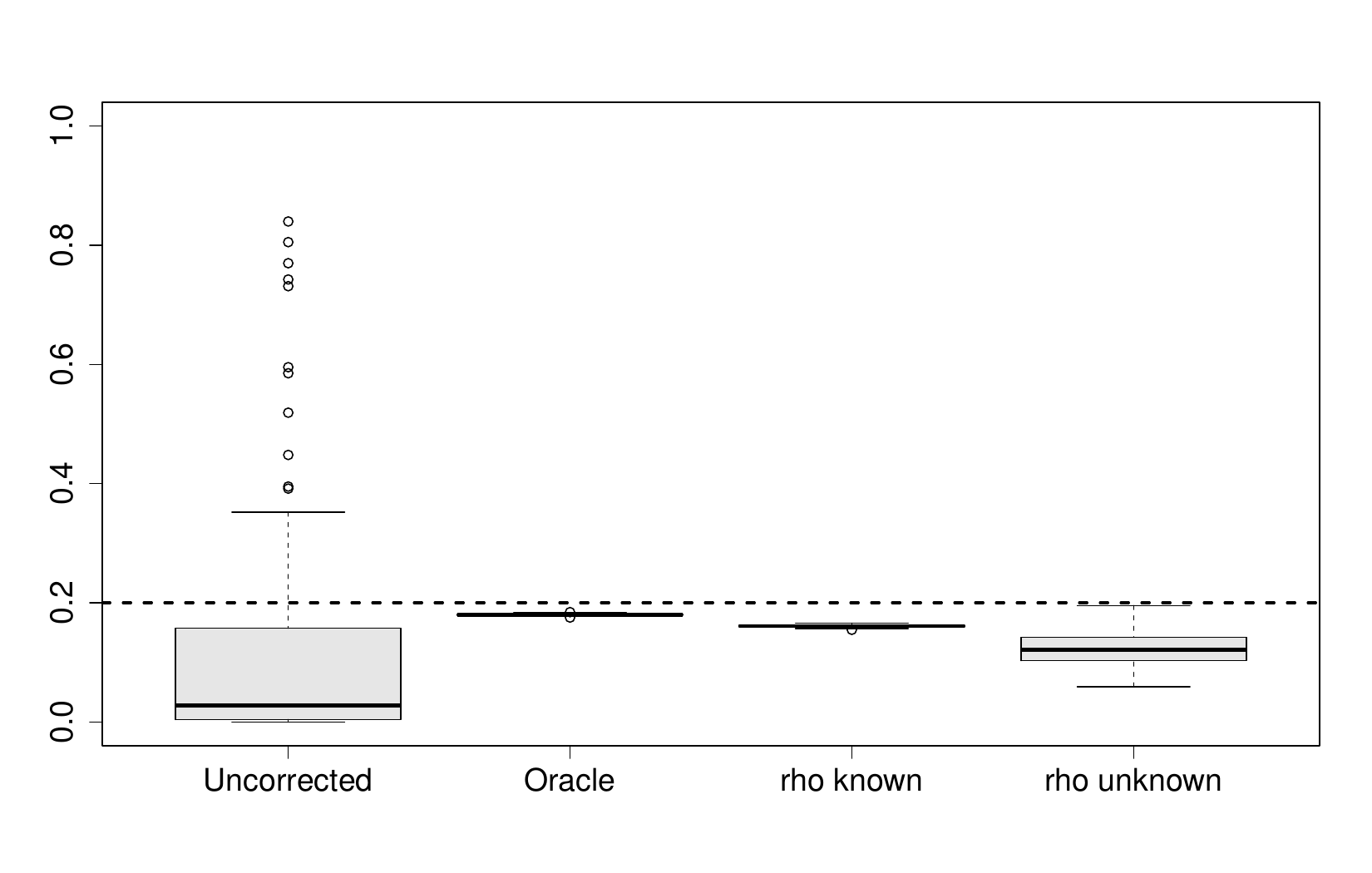}&
\includegraphics[scale=0.27]{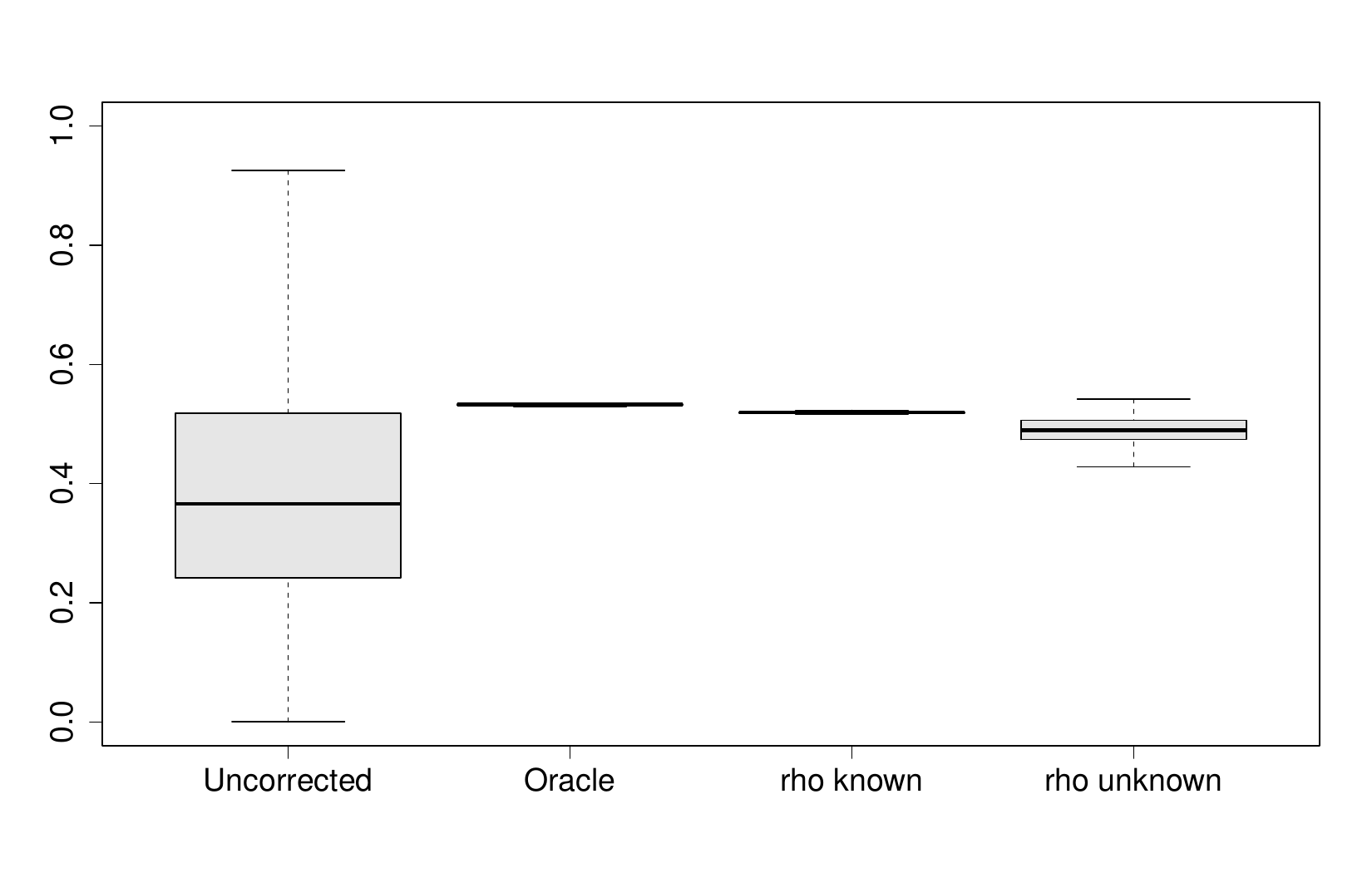}\\
\rotatebox{90}{\hspace{2cm}$\Delta=3$}&
\includegraphics[scale=0.27]{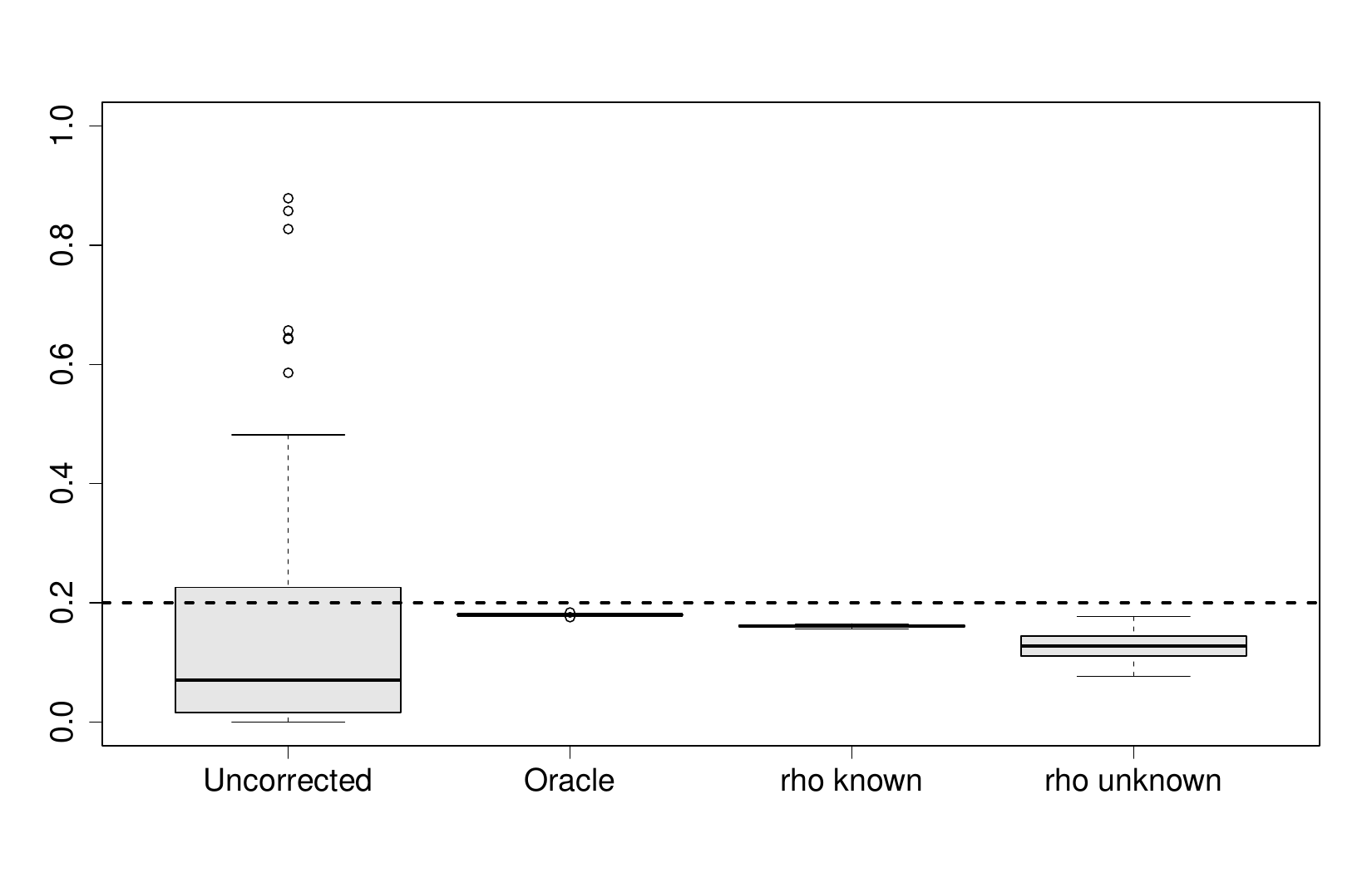}&
\includegraphics[scale=0.27]{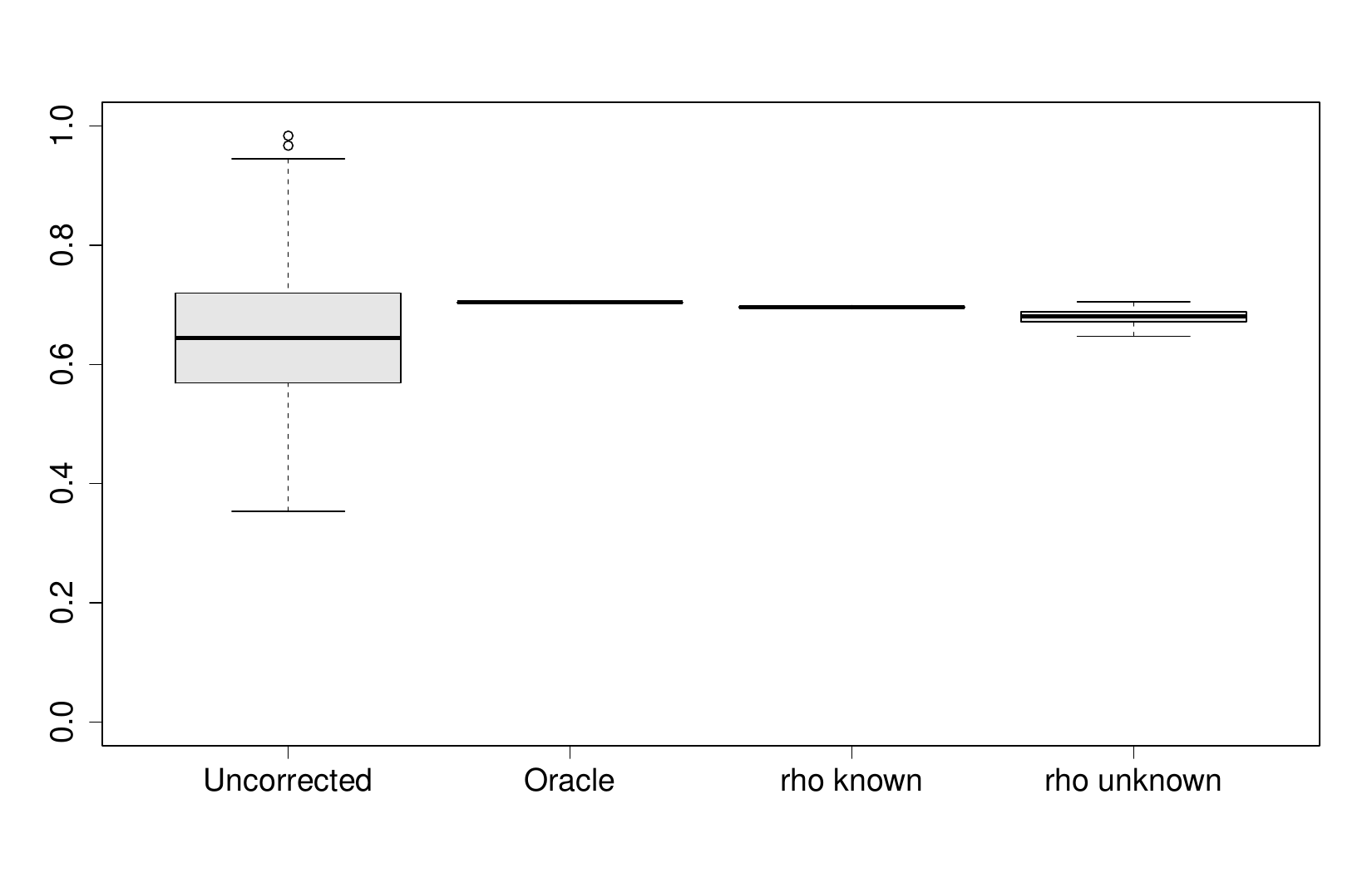}
\end{tabular}
\caption{Same as Figure~\ref{fig:boxplot} but for alternative $m_i$ i.i.d. uniform in $(0.01,2\Delta)$. } \label{fig:boxplot_rand}
\end{figure}

\begin{figure}[h!]
\begin{tabular}{cc}
\vspace{-5mm}
Uncorrected
& 
Oracle
\\
\includegraphics[scale=0.4]{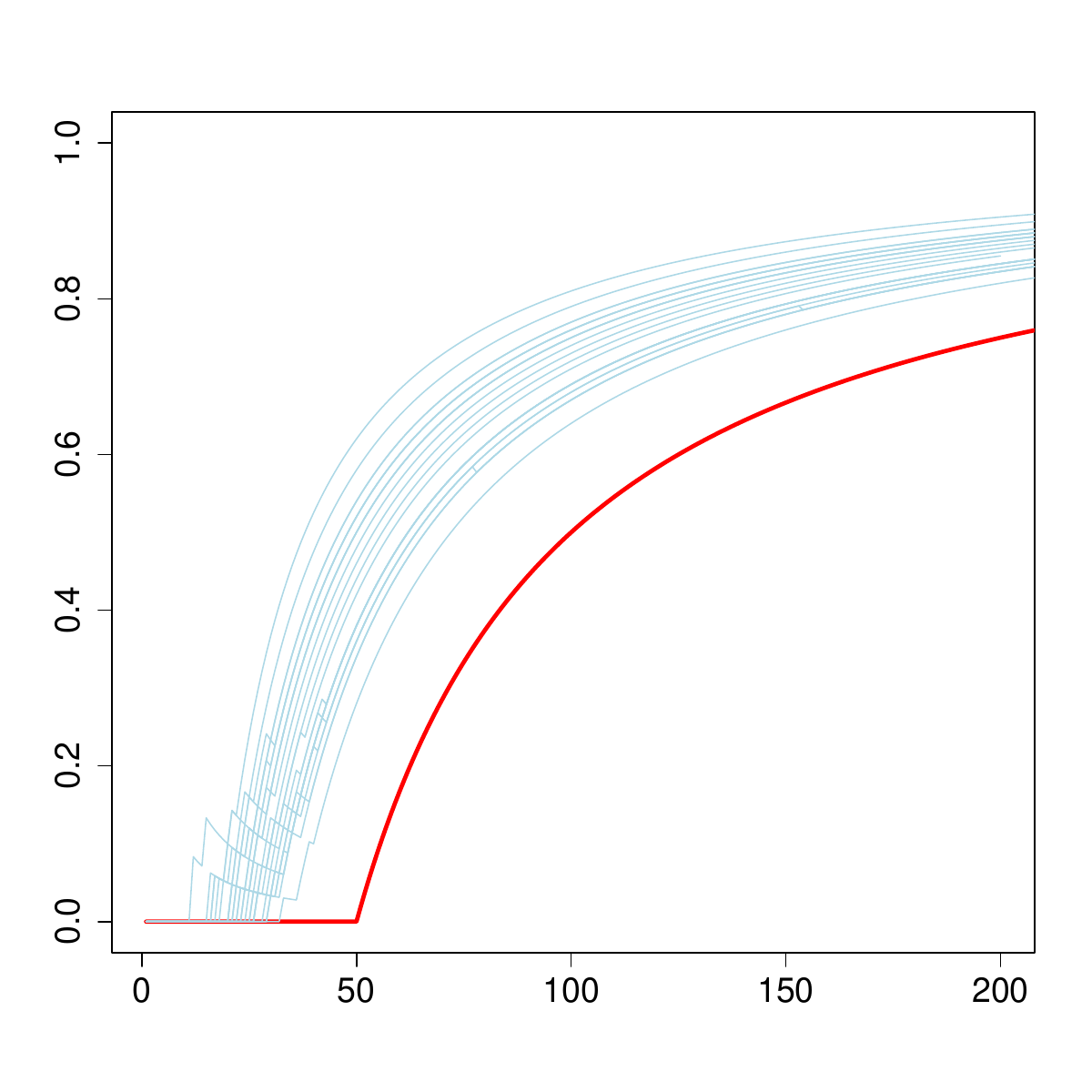}&\includegraphics[scale=0.4]{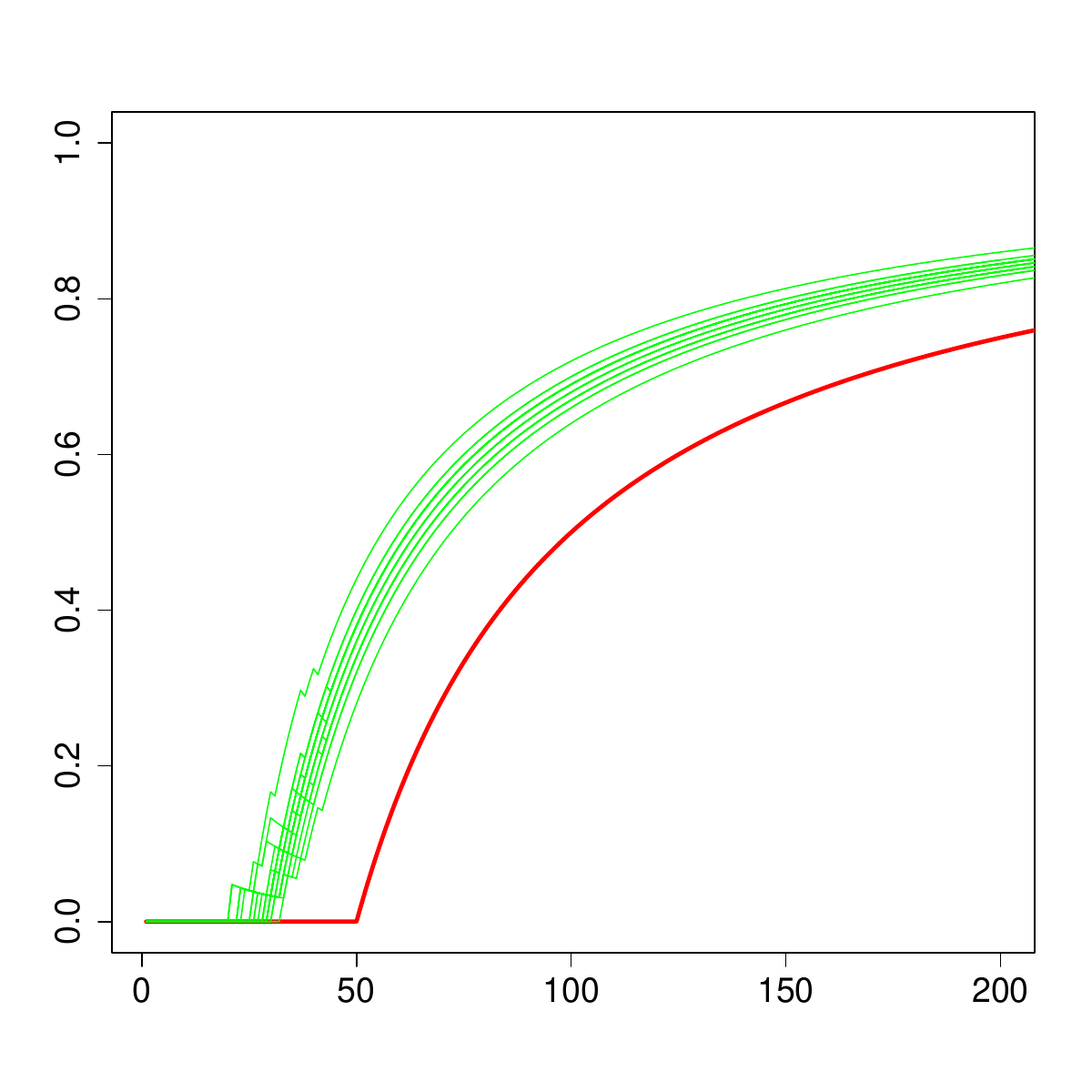}\\
\vspace{-5mm}
Correlation known
&
Correlation unknown
\\
\includegraphics[scale=0.4]{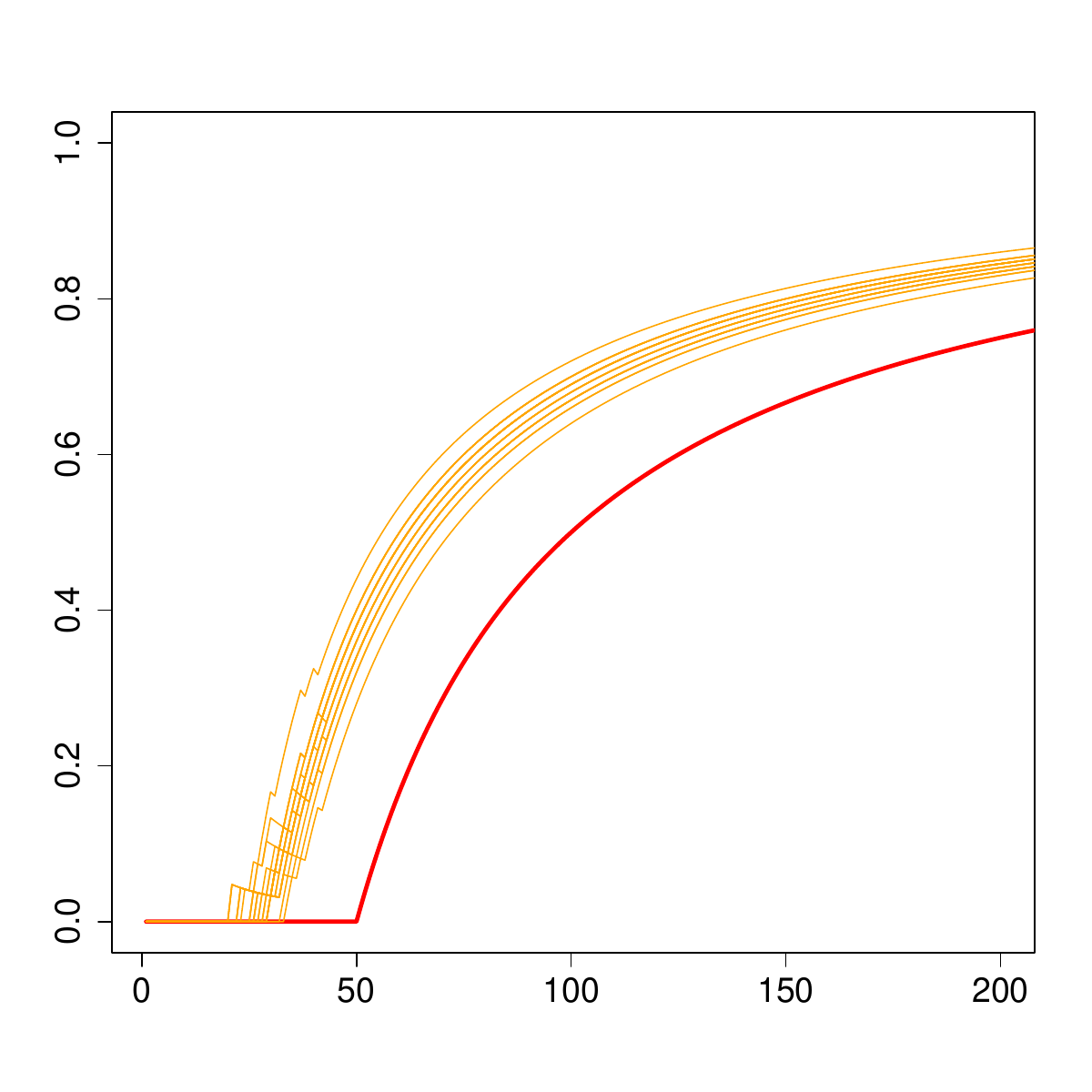}&\includegraphics[scale=0.4]{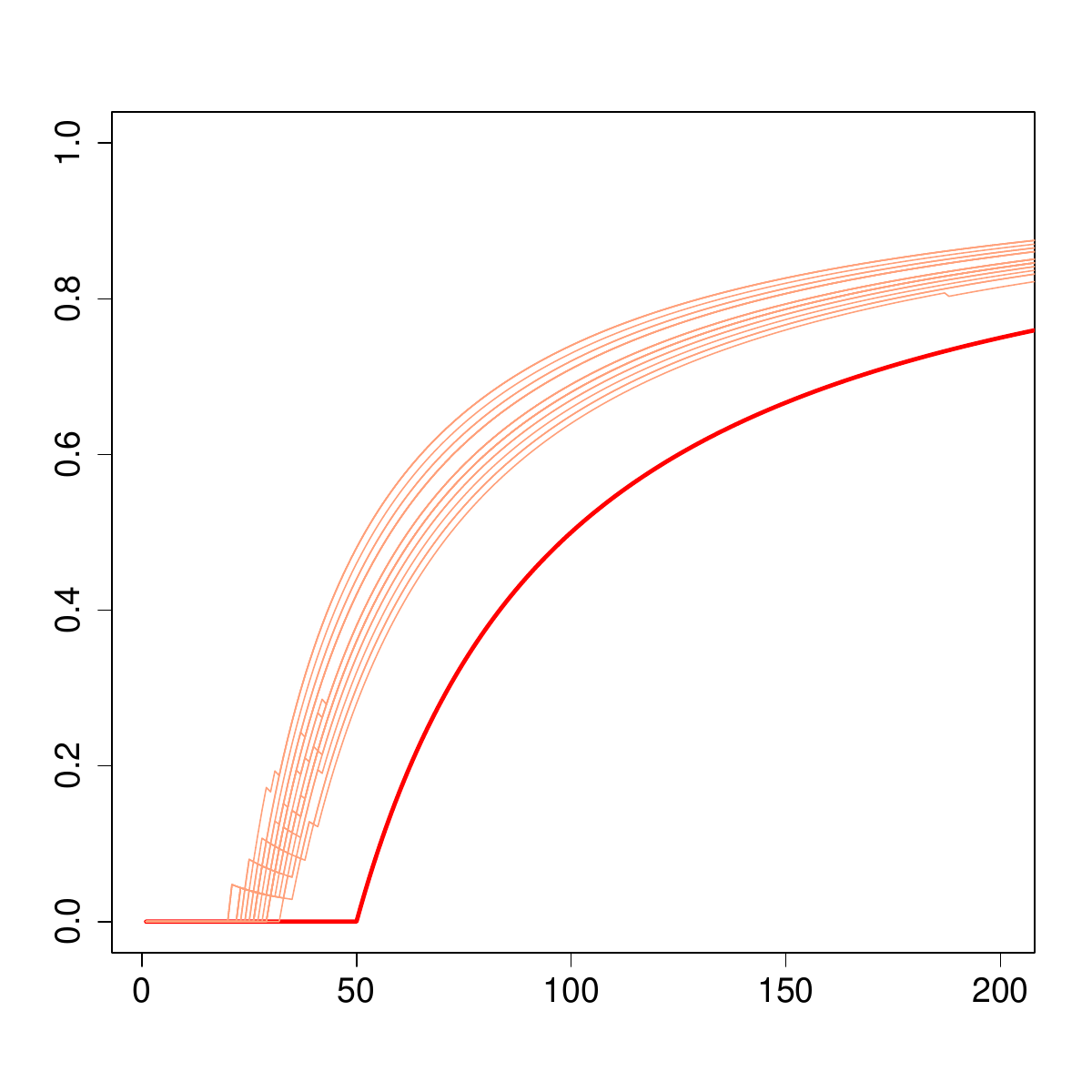}
\end{tabular}
\caption{Same as Figure~\ref{fig:equi:posthoc} but for alternative $m_i$ linearly increasing from $0.01$ to $2\Delta$.}\label{fig:equi:posthoc_lin}
\end{figure}

\begin{figure}[h!]
\begin{tabular}{cc}
\vspace{-5mm}
Uncorrected
& 
Oracle
\\
\includegraphics[scale=0.4]{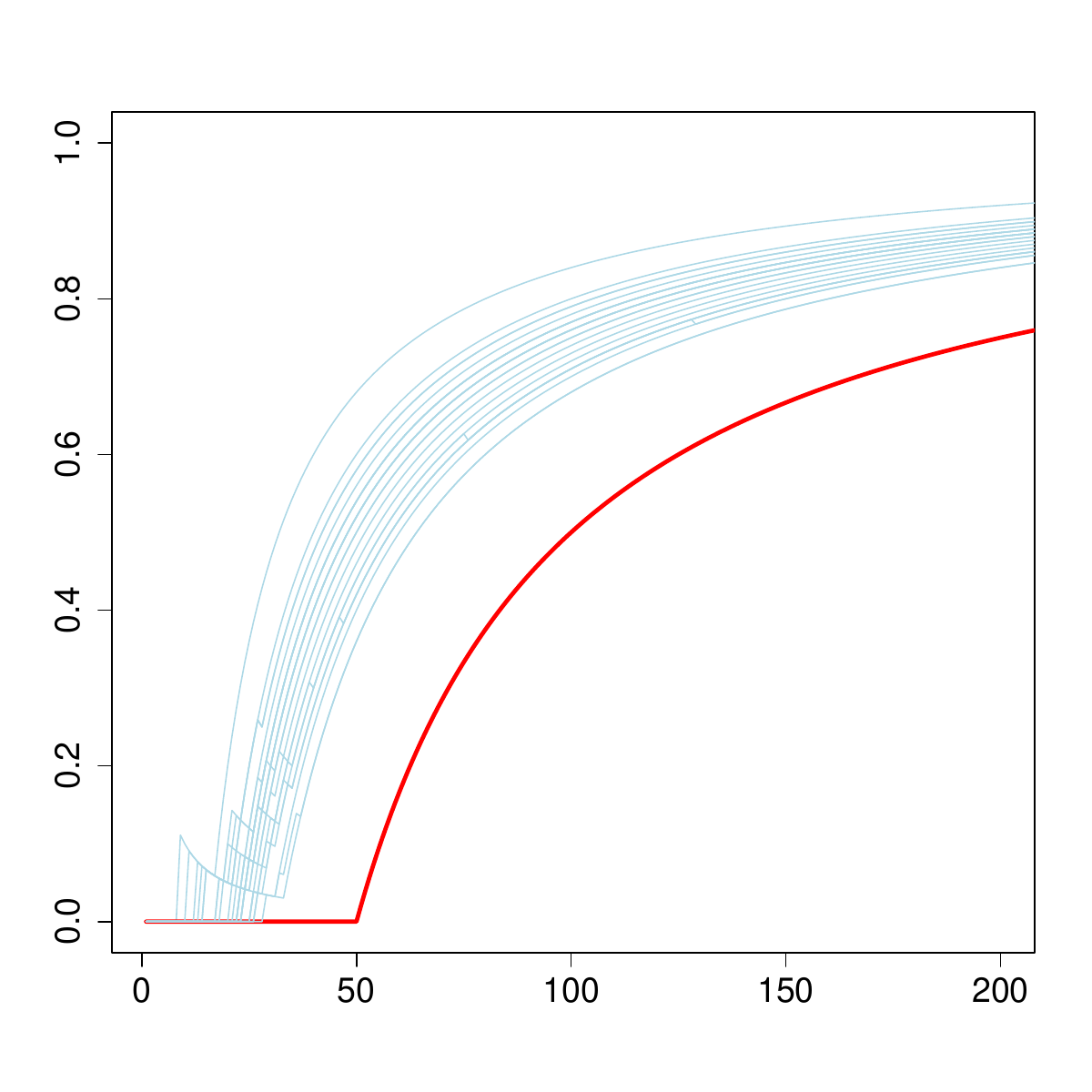}&\includegraphics[scale=0.4]{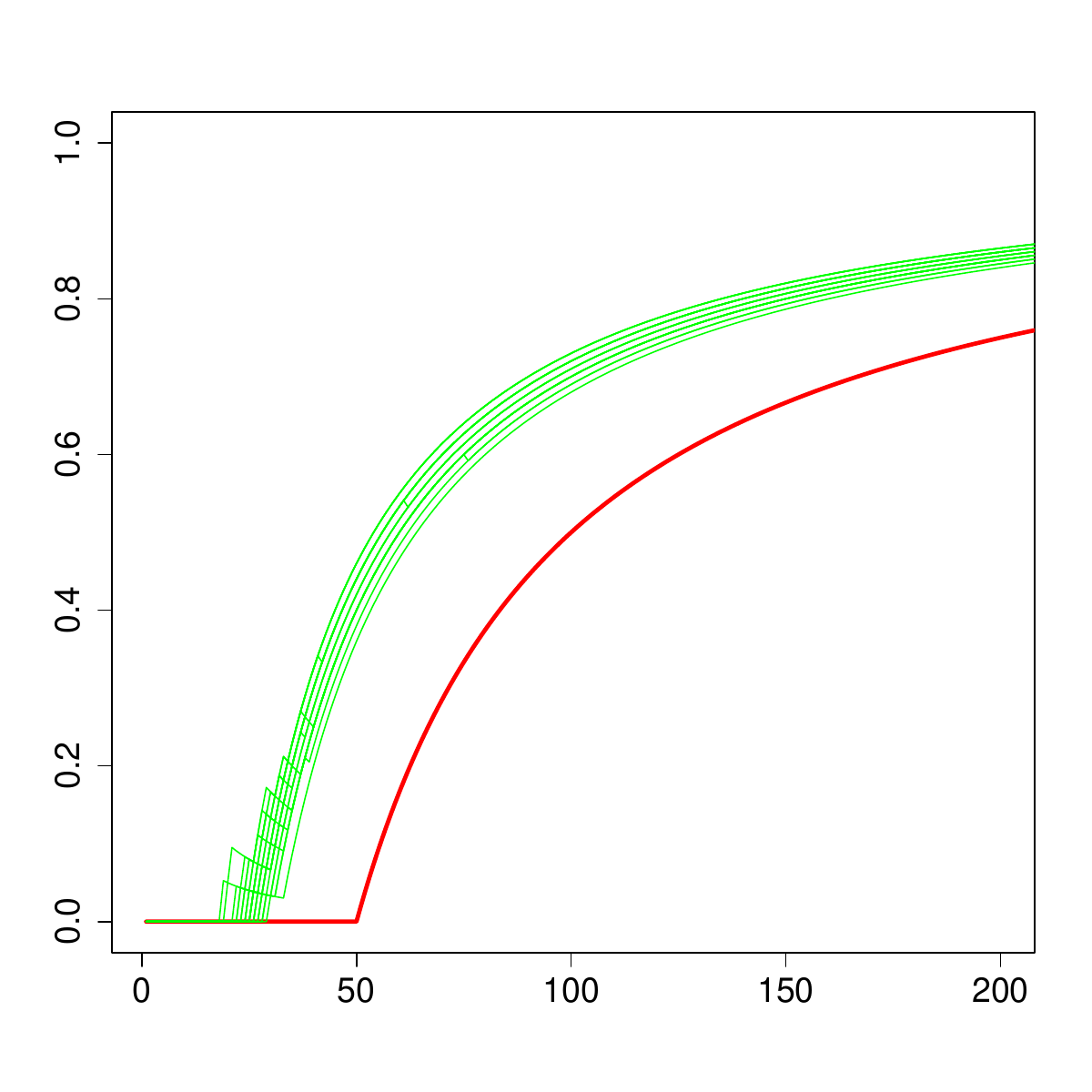}\\
\vspace{-5mm}
Correlation known
&
Correlation unknown
\\
\includegraphics[scale=0.4]{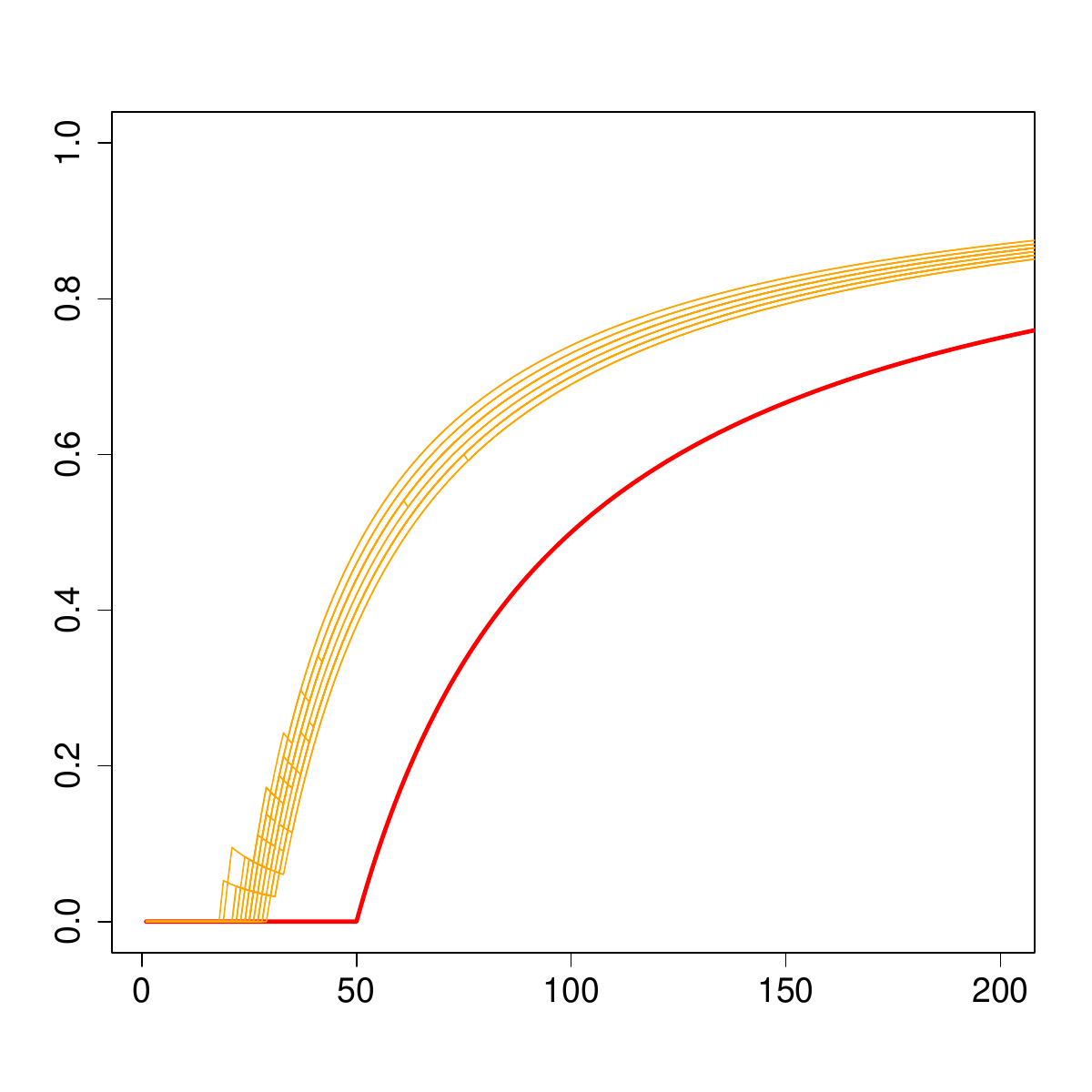}&\includegraphics[scale=0.4]{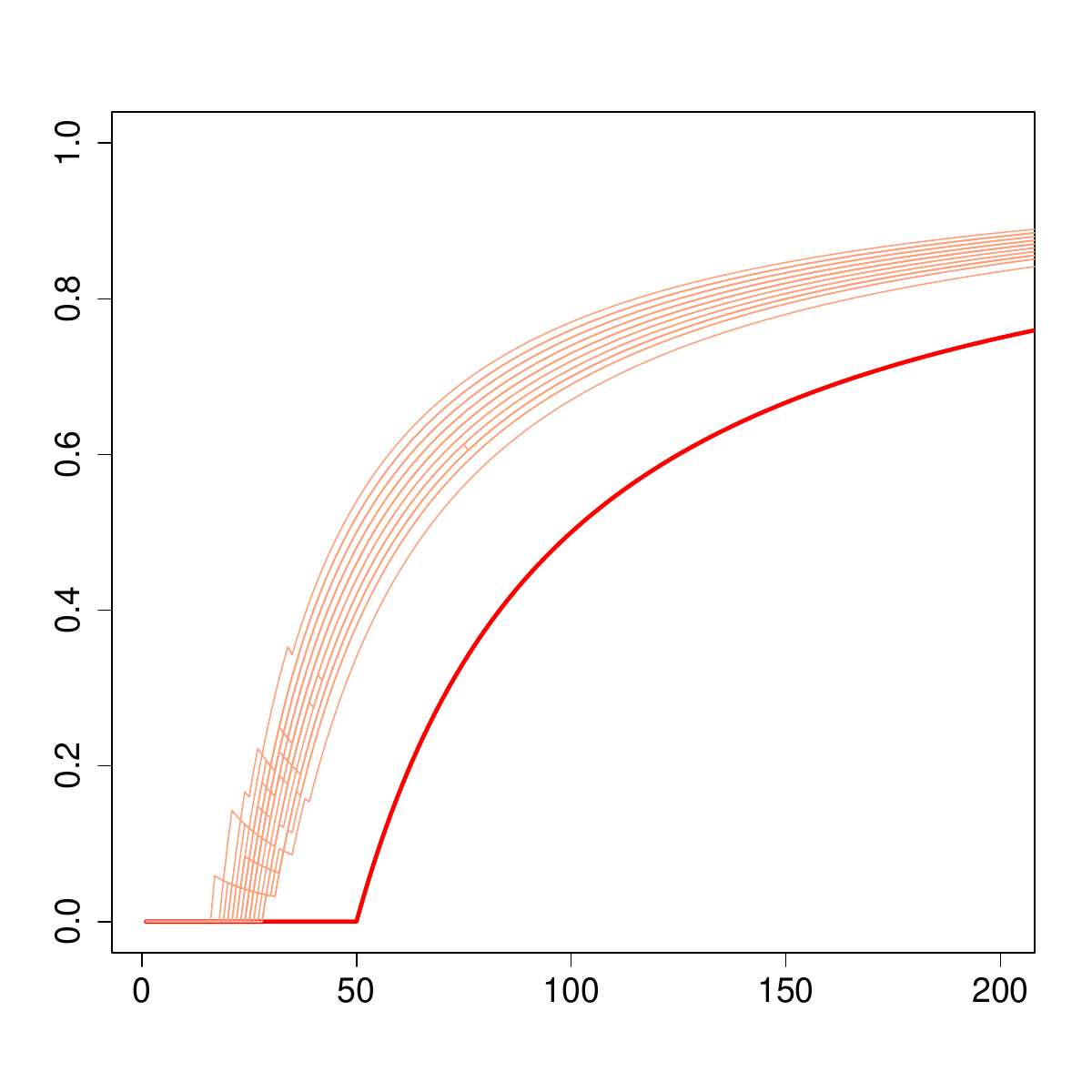}
\end{tabular}
\caption{Same as Figure~\ref{fig:equi:posthoc} but for alternative $m_i$ i.i.d. uniform in $(0.01,2\Delta)$.}\label{fig:equi:posthoc_rand}
\end{figure}

 \bibliographystyle{plain}
\bibliography{biblio}

\end{document}